\documentclass[12pt,twoside,a4paper]{amsart}
\usepackage{graphicx} 
\usepackage{amssymb,amscd,amsxtra,calc,mathrsfs}
\usepackage{amsmath}
\usepackage{bm}
\usepackage[all]{xy}
\usepackage{graphicx}
\usepackage{tikz}
\usepackage{float}
\usetikzlibrary{patterns}
\tikzset{>=latex}


\allowdisplaybreaks[1]

\numberwithin{equation}{section}

\title{Elementary links from prime Fano threefolds along two lines}
\date{\today}

\author{Kento Fujita} 
\address{Department of Mathematics, Graduate School of Science, 
Osaka University, Toyonaka, Osaka 560-0043, Japan}
\email{fujita@math.sci.osaka-u.ac.jp}

\newcommand{\pr}{\mathbb{P}}

\newcommand{\Z}{\mathbb{Z}}
\newcommand{\Q}{\mathbb{Q}}
\newcommand{\R}{\mathbb{R}}
\newcommand{\C}{\mathbb{C}}
\newcommand{\F}{\mathbb{F}}
\newcommand{\D}{\mathbb{D}}
\newcommand{\T}{\mathbb{T}}
\newcommand{\U}{\mathbb{U}}
\newcommand{\X}{\mathbb{X}}
\newcommand{\Y}{\mathbb{Y}}

\newcommand{\W}{\mathbb{W}}
\newcommand{\V}{\mathbb{V}}
\newcommand{\B}{{\bf B}}

\newcommand{\A}{\mathbb{A}}
\newcommand{\G}{\mathbb{G}}
\newcommand{\E}{\mathbb{E}}

\newcommand{\be}{\mathbf{e}}
\newcommand{\Bf}{\mathbf{f}}
\newcommand{\Bc}{\mathbf{c}}

\newcommand{\ND}{\operatorname{N}^1}

\newcommand{\Nef}{\operatorname{Nef}}
\newcommand{\Mov}{\operatorname{Mov}}

\newcommand{\Psef}{\operatorname{Psef}}

\newcommand{\Supp}{\operatorname{Supp}}
\newcommand{\Exc}{\operatorname{Exc}}
\newcommand{\Sing}{\operatorname{Sing}}

\newcommand{\Spec}{\operatorname{Spec}}
\newcommand{\Pic}{\operatorname{Pic}}

\newcommand{\Ext}{\operatorname{Ext}}
\newcommand{\Aut}{\operatorname{Aut}}
\newcommand{\Proj}{\operatorname{Proj}}

\newcommand{\PGL}{\operatorname{PGL}}

\newcommand{\mult}{\operatorname{mult}}

\newcommand{\MU}{\operatorname{MU}}

\newcommand{\KP}{\operatorname{KP}}

\newcommand{\sC}{\mathcal{C}}

\newcommand{\sO}{\mathcal{O}}
\newcommand{\sN}{\mathcal{N}}
\newcommand{\sE}{\mathcal{E}}

\newcommand{\sL}{\mathcal{L}}

\newcommand{\dm}{\mathfrak{m}}

\newcommand{\ttx}{\mathtt{x}}
\newcommand{\tty}{\mathtt{y}}
\newcommand{\ttz}{\mathtt{z}}
\newcommand{\ttt}{\mathtt{t}}
\newcommand{\ttw}{\mathtt{w}}
\newcommand{\ttf}{\mathtt{f}}
\newcommand{\ttg}{\mathtt{g}}
\newcommand{\tth}{\mathtt{h}}

\newtheorem{thm}{Theorem}[section]
\newtheorem{lemma}[thm]{Lemma}
\newtheorem{proposition}[thm]{Proposition}
\newtheorem{corollary}[thm]{Corollary}

\theoremstyle{definition}
\newtheorem{definition}[thm]{Definition}
\newtheorem{remark}[thm]{Remark}

\newtheorem{example}[thm]{Example}
\newtheorem*{ack}{Acknowledgments}

\begin{document}

\maketitle 

\begin{abstract}
For prime Fano threefolds $X$ of genus $g=12$, $10$ or $9$, and for 
totally disjoint pairs of lines $Z_1$, $Z_2$ in $X$, we establish links from the blowups 
of $X$ along $Z_1$ and $Z_2$. If $g=12$, then the links end with the blowups of 
Fano threefolds of type 2.21 along bi-cubic curves; if $g=10$, then the links 
end with the blowups of the projectivization of the tangent bundle of the 
projective plane along genus $2$ bi-quintic curves with a mild condition; if $g=9$, 
then the links end with conic bundles over the product of two projective lines 
with the discriminant loci of bidegree $(3,3)$. When $g=12$ or $g=10$, we also 
establish the converses of the above links. 
Moreover, we especially focus on the links when $g=12$ and the links are 
$\mathbb{G}_m$-equivariant. 
\end{abstract}

\setcounter{tocdepth}{1}
\tableofcontents

\part{Preliminaries}\label{part:prelim}

\section{Introduction}\label{section:intro}

We always work over an 
algebraically closed field $\Bbbk$ of characteristic zero. 
A \emph{prime Fano threefold} is defined to be 
a smooth Fano threefold $X$ such that its Picard group is generated by its 
anti-canonical divisor $-K_X$ (see Definition 
\ref{definition:X22}). 
The \emph{degree} of $X$ is defined to be its anti-canonical volume $(-K_X)^{\cdot 3}$, 
and the \emph{genus} of $X$ is defined to be the value $1+\frac{1}{2}(-K_X)^{\cdot 3}$. 
We mainly consider the cases that the genus $g$ of $X$ is one of $9,10,12$. 
For the cases, the anti-canonical divisor $-K_X$ is very ample and gives an embedding 
$X\hookrightarrow\pr^{g+1}$. A \emph{line} (resp., \emph{conic}) in $X$ is an irreducible 
curve in $X$ which is a line (resp., conic) inside $\pr^{g+1}$ under the above 
anti-canonical embedding. Two pairwise disjoint lines $Z_1$, $Z_2\subset X$ are said to 
be a \emph{totally disjoint pair of lines} if there is no line $Z\subset X$ in $X$ with 
$Z\cap Z_1\neq\emptyset$ and $Z\cap Z_2\neq\emptyset$ (see Definition 
\ref{definition:X22}). 
Prime Fano threefolds of genus $9$, $10$, $12$ have many interesting structures 
in the sense of birational geometry. 
For example, Iskovskikh \cite{Isk79} analyzed the Sarkisov link starting from the blowup 
of $X$ along a line. 
Later, Takeuchi \cite{Tak89} found another interesting Sarkisov link starting from 
the blowup of $X$ along a conic. 

Let us firstly consider the case that the genus of $X$ is equal to $12$. 
They observed that, Iskovskikh's link ends with the del Pezzo threefold of degree 
$5$ together with a rational quintic curve, and Takeuchi's link ends with 
the $3$-dimensional smooth hyperquadric together with a sextic rational curve. 
The \emph{del Pezzo threefold of degree $5$} is defined to be a smooth Fano threefold 
$V$ such that its Picard group is generated by the half $-\frac{1}{2}K_V$ of the 
anti-canonical divisor and $\left(-\frac{1}{2}K_V\right)^{\cdot 3}=5$ holds 
(see Definition \ref{definition:dP5}). 
Takao Fujita \cite{Fuj81} and Iskovskikh \cite{Isk77} independently showed that, 
the isomorphism class of such $V$ is unique, and it can be analyzed from the 
Sarkisov link \eqref{equation:fujita} 
starting from the blowup of $V$  along lines (i.e., curves with the 
half-anti-canonical degree $1$). 

One of the most interesting subclass among prime Fano threefolds $X$ of genus $12$ 
are the class of such $X$ with \emph{infinite} automorphism groups $\Aut(X)$. 
Such $X$ are classified by Kuznetsov, Prokhorov and Shramov, and their automorphism 
groups are determined \cite{KPS18}. They in particular showed that, if the identity 
component $\Aut^0(X)$ of $\Aut(X)$ is isomorphic to the multiplicative group 
$\G_m=\Spec\Bbbk[t^{\pm 1}]$, then the set of such isomorphism classes form 
a $1$-dimensional family $\{X_{22}^m(v)\}_{v\in\left\{0,1,-4,\infty\right\}}$. 
(We follow the parametrization by \cite{DFK}. See \S \ref{section:parameter} in detail.)
Moreover, we have 
$\Aut(X_{22}^m(v))\cong\G_m\rtimes\boldsymbol{\mu}_2$, where $\boldsymbol{\mu}_2$ 
is the cyclic group scheme of order $2$. In order to show this deep result, the authors 
of \cite{KPS18} firstly observed that there is a pair of disjoint $\G_m$-invariant 
lines $Z_1^m(v)$, $Z_2^m(v)$ in $X_{22}^m(v)$ such that the union 
$Z_1^m(v)\cup Z_2^m(v)$ of such line is $\Aut(X_{22}^m(v))$-invariant. 
The difficult step is to find an involution $\iota$ of $X_{22}^m(v)$, 
which becomes the generator of $\boldsymbol{\mu}_2$, 
\emph{swapping $Z_1^m(v)$ and $Z_2^m(v)$}. The authors of \cite{KPS18} showed 
the existence of such involution by using the theory of varieties of sums of powers 
\cite{M, DKK17}. Later, Kuznetsov and Prokhorov \cite{KP18} gave an alternative proof 
for the existence of such involution $\iota$. They firstly show that there exists an 
$\Aut(X_{22}^m(v))$-invariant
smooth conic by focusing on the Hilbert scheme of conics in $X_{22}^m(v)$. 
After running Takeuchi's link equivariantly from the above conic,  they get the desired 
involution. (Recently, Ito, Kanemitsu, Takamatsu and Tanaka \cite{IKTT} gave an 
alternative and simple proof for the existence of such a conic. See Remark 
\ref{remark:IKTT}.)

For the existence of such involution $\iota$ of the above $X_{22}^m(v)$, the authors in 
\cite{DFK} gave another proof. The idea of 
their proof is the following. Firstly we blowup $\sigma\colon X_0^m(v)\to X_{22}^m(v)$ 
along \emph{the union of $Z_1^m(v)$ and $Z_2^m(v)$} and let $F_1+F_2$ be the union of 
$\sigma$-exceptional divisors. Then, the variety $X_0^m(v)$ is 
a smooth weak Fano threefold and the ample model $\hat{Q}^m(v)$ of 
$\sigma^*(-2K_{X_{22}^m(v)})-3(F_1+F_2)$ is a Fano threefold of type 2.21 with an 
effective $\G_m$-action. A \emph{Fano threefold $\hat{Q}$ of type 2.21} is defined to be 
the blowup $\rho\colon\hat{Q}\to Q$ of smooth $3$-dimensional hyperquadric 
along a twisted quartic curve. 
A \emph{bi-cubic curve} in $\hat{Q}$ is a smooth 
rational curve $\hat{C}\subset\hat{Q}$ such that $\left(\rho^*\sO_Q(1)\cdot \hat{C}
\right)=\left(S\cdot\hat{C}\right)=3$ holds, where $S\subset \hat{Q}$ is the 
exceptional divisor of $\rho$ (see Definition \ref{definition:2-21a} in detail). 
From the structure of of $\hat{Q}^m(v)$, there exists a suitable involution on 
$\hat{Q}^m(v)$ such that we can export it to the desired involution $\iota$ on 
$X_{22}^m(v)$. This is the key idea of \cite[Lemma 23]{DFK}. 

The purpose of this article is to generalize the above idea obtained in \cite{DFK}. 
Moreover, we consider that the genus of prime Fano threefolds is not only $12$ but 
also $10$ and $9$.

More precisely, from prime Fano threefolds $X$ of genus $g\in\{12, 10, 9\}$ 
together with pairs of totally disjoint lines $Z_1$, $Z_2$ in $X$, 
we construct the Sarkisov-like elementary links starting from the blowups of 
$X$ along $Z_1$ and $Z_2$. In order to state our main theorem, we prepare the following 
notation. The \emph{del Pezzo threefold $U$ of degree $6$ and rank $2$} is 
nothing but $U=\pr_{\pr^2}(T_{\pr^2})$. The variety $U$ admits exactly $2$ numbers of 
$\pr^1$-bundle structures $\rho_1\colon U\to \pr^2$ and $\rho_2\colon U\to\pr^2$. 
An irreducible curve $\Gamma\subset U$ is said to be a \emph{bi-quintic curve} 
if $(\rho_i^*\sO(1)\cdot \Gamma)=5$ for $i=1,2$ 
(see Definition \ref{definition:dP6} in detail). Here is one of our main result in this 
article: 

\begin{thm}\label{thm:main}
Let $X$ be a prime Fano threefold of genus $g\in\{12,10,9\}$ and let 
$Z_1$, $Z_2\subset X$ be a totally disjoint pair of lines in $X$. 
Consider the blowup 
$\sigma\colon X_0\to X$ along $Z_1\cup Z_2$, and let $F_1$, $F_2\subset X_0$ be the 
exceptional divisors. Then $X_0$ is a smooth weak Fano threefold and its anti-canonical 
model $\alpha\colon X_0\to \bar{X}_0$ is a small morphism of relative Picard rank $2$. 
The $(\sigma^*K_X)$-flop $\alpha^+\colon X_0^+\to\bar{X}_0$ of $\alpha$ 
is again a \emph{smooth} weak Fano threefold, and 
the strict transform of $\sigma^*(-2K_X)-3(F_1+F_2)$ on $X_0^+$ is semiample and 
gives a contraction morphism $\tau\colon X_0^+\to\Y$. 
As a consequence, we get the following elementary link: 
\[
X \, \stackrel{\sigma}{\longleftarrow} X_0 \, \stackrel{\chi}{\dasharrow} \, X_0^+ \, 
\stackrel{\tau}{\longrightarrow} \, \Y,
\]
where $\chi:=(\alpha^+)^{-1}\circ\alpha$. Moreover, we have the following: 
\begin{enumerate}
\renewcommand{\theenumi}{\arabic{enumi}}
\renewcommand{\labelenumi}{(\theenumi)}
\item\label{thm:main1}
If $g=12$, then $\Y=\hat{Q}$ is a Fano threefold of type 2.21 and $\tau$ is the 
blowup of $\hat{Q}$ along a bi-cubic curve $\hat{C}$. 
\item\label{thm:main2}
If $g=10$, then $\Y=U$ is the del Pezzo threefold of degree $6$ and rank $2$ 
and $\tau$ is the 
blowup of $U$ along a smooth bi-quintic curve $\Gamma$ of genus $2$. 
Moreover, the curve $\Gamma$ satisfies that 
$\mult_{p_i}\left(\rho_i(\Gamma)\right)\leq 2$ holds for any $i=1,2$ and for any 
$p_i\in\rho_i(\Gamma)$. 
\item\label{thm:main3}
If $g=9$, then $\Y=\pr^1\times\pr^1$ and the morphism $\tau$ is a conic bundle 
with the discriminant locus $\Delta_\tau\in|\sO(3,3)|$. 
\end{enumerate}
\end{thm}

In fact, we will give \emph{all} of the small $\Q$-factorial modifications of $X_0$. 
See the diagrams \eqref{equation:big}, \eqref{equation:big2} and 
Theorem \ref{thm:diamond}. 

We remark that, since $\alpha$ and $\alpha^+$ are of relative Picard rank higher 
than $1$, it is not trivial that their flops are $\Q$-factorial. 
We also remark that, the link in Theorem \ref{thm:main} \eqref{thm:main1} has been 
partially observed in \cite[Remark 2.13]{CS} \emph{under the highly restricted 
assumption that there is an involution of $X$ swapping the two lines}. 
If there exists such an involution, then the situation is much easier; we can consider a 
usual equivariant Sarkisov link. 
Our assumption is much weaker, and it turned out to be much difficult to analyze 
the links, since the blowups of $X$ along the lines will be of Picard rank $3$. 
We overcome the difficulty by looking at all of small $\Q$-factorial modifications 
of the blowups. 

When $g\in\{12, 10\}$, we can also construct the reverse of the links 
constructed in Theorem \ref{thm:main}. Here is another main result in this article: 

\begin{thm}\label{thm:main-back}
\begin{enumerate}
\renewcommand{\theenumi}{\arabic{enumi}}
\renewcommand{\labelenumi}{(\theenumi)}
\item\label{thm:main-back1}
Let $\hat{Q}$ be a Fano threefold of type 2.21 and let 
$\hat{C}\subset\hat{Q}$ be a bi-cubic curve in $\hat{Q}$. Consider the blowup 
$\tau\colon X_0^+\to \hat{Q}$ along $\hat{C}$, and let $E^+\subset X_0^+$ be 
the exceptional divisor. Then $X_0^+$ is a smooth weak Fano threefold and 
its anti-canonical model $\alpha^+\colon X_0^+\to \bar{X}_0$ is a small morphism of 
relative Picard rank $2$. The $(\tau^*K_{\hat{Q}})$-flop $\alpha\colon X_0\to\bar{X}_0$ 
of $\alpha^+$ is again a \emph{smooth} weak Fano threefold, and 
the strict transform of $\tau^*(-2K_{\hat{Q}})-3E^+$ on $X_0$ is semiample and 
gives the birational morphism 
$\sigma\colon X_0\to X$, which is the blowup of a prime Fano threefold of genus $12$ 
along a totally disjoint pair of lines $Z_1$, $Z_2\subset X$ in $X$. 
As a consequence, there exists the following elementary link: 
\[
\hat{Q} \, \stackrel{\tau}{\longleftarrow} X_0^+ \, 
\stackrel{\chi^{-1}}{\dasharrow} \, X_0 \, \stackrel{\sigma}{\longrightarrow} \, X,
\]
where $\chi^{-1}:=\alpha^{-1}\circ\alpha^+$. 
Moreover, the links in Theorem \ref{thm:main} \eqref{thm:main1} and the above 
are converse to each other. 
In particular, we have 
\[
\Aut(X; Z_1\cup Z_2)\cong\Aut(X_0)\cong 
\Aut(X_0^+)\cong\Aut(\hat{Q}; \hat{C}) 
\]
under the above links. 
\item\label{thm:main-back2}
Let $U$ be the del Pezzo threefold of degree $6$ and rank $2$, and let 
$\Gamma\subset U$ be a smooth bi-quintic curve in $U$ of genus $2$ such that 
the multiplicity of the curve $\rho_i(\Gamma)\subset\pr^2$ at each point in 
$\rho_i(\Gamma)$ 
is at most $2$ for any $i\in\{1,2\}$. Consider the blowup 
$\tau\colon X_0^+\to U$ along $\Gamma$, and let $E^+\subset X_0^+$ be 
the exceptional divisor. Then $X_0^+$ is a smooth weak Fano threefold and 
its anti-canonical model $\alpha^+\colon X_0^+\to \bar{X}_0$ is a small morphism of 
relative Picard rank $2$. The $(\tau^*K_U)$-flop $\alpha\colon X_0\to\bar{X}_0$ 
of $\alpha^+$ 
is again a \emph{smooth} weak Fano threefold, and 
the strict transform of $\tau^*\left(-\frac{5}{2}K_U\right)-3E^+$ on $X_0$ 
is semiample and gives the birational morphism 
$\sigma\colon X_0\to X$, which is the blowup of a prime Fano threefold of genus $10$ 
along a totally disjoint pair of lines $Z_1$, $Z_2\subset X$ in $X$. 
As a consequence, there exists the following elementary link: 
\[
U \, \stackrel{\tau}{\longleftarrow} X_0^+ \, 
\stackrel{\chi^{-1}}{\dasharrow} \, X_0 \, \stackrel{\sigma}{\longrightarrow} \, X,
\]
where $\chi^{-1}:=\alpha^{-1}\circ\alpha^+$. 
Moreover, the links in Theorem \ref{thm:main} \eqref{thm:main2} and the above 
are converse to each other. 
In particular, we have 
\[
\Aut(X; Z_1\cup Z_2)\cong\Aut(X_0)\cong 
\Aut(X_0^+)\cong\Aut(U; \Gamma) 
\]
under the above links. 
\end{enumerate}
\end{thm}

In \S \ref{section:flop}, especially in Theorem \ref{thm:flop-curve}, 
we analyze the flopping and flopped curves in the diagram \eqref{equation:big}. 
In \S \ref{section:configuration}, as a corollary of the results in \S \ref{section:flop}, 
we analyze possibilities for the configurations of lines in prime Fano threefold 
of genus $12$. In fact, we get the following: 

\begin{thm}[{see Theorem \ref{thm:configuration} in detail}]\label{thm:configuration-intro}
Let $X$ be a prime Fano threefold of genus $12$ and take any $m\in\{4,5,6\}$. 
Then there is no pairwise distinct lines $Z'_1,\dots,Z'_m$ in $X$ satisfying 
$Z'_i\cap Z'_{i+1}\neq\emptyset$ for all $1\leq i\leq m$, 
where we set $Z'_{m+1}:=Z'_1$. 
\end{thm}

The above theorem is not a main result in this article, but it looks important 
for the further studies of lines in prime Fano threefolds of genus $12$.

This article is divided into 3 numbers of parts. Part \ref{part:prelim} is a preliminary part. 
In \S \ref{section:3flops}, we recall and see several general theory of birational geometry. 
Moreover, we see several examples of $3$-dimensional elementary flops, which will be 
crucial in Part \ref{part:construction-links} and Part \ref{part:application}. 
In \S \ref{section:fano3}, 
we recall and see fundamental properties of several known smooth Fano threefolds. 
Especially we will focus on the del Pezzo threefold of degree $5$ and the del Pezzo 
threefold of degree $6$ and rank $2$. In \S \ref{section:prime3}, we recall 
Iskovskikh's double projection from several prime Fano threefold along lines, and 
then we will see basic properties of totally disjoint pairs of lines in prime Fano threefolds. 
One of the main goal of the section is to prove Proposition 
\ref{proposition:totally-blowup}. 

Part \ref{part:construction-links} is the main part of the article. In \S \ref{section:go}, 
we not only prove Theorem \ref{thm:main} but also describe all rational contraction 
maps starting from the blowup of prime Fano threefolds of genus $g\in\{12, 10, 9\}$ 
along totally disjoint pairs of lines. See the big diagrams \eqref{equation:big}, 
\eqref{equation:big2} and 
Theorem \ref{thm:diamond} in detail. 
In \S \ref{section:back}, we see the converse of the links studied in \S \ref{section:go} 
when $g\in\{12,10\}$. Theorem \ref{thm:main-back} is an immediate consequence of 
Corollaries \ref{corollary:back-g12} and \ref{corollary:back-g10}. 
In \S \ref{section:flop}, we analyze the flopping and flopped curves for each steps of 
elementary flops in \eqref{equation:big2}. 

In Part \ref{part:application}, we see several applications of the diagrams 
\eqref{equation:big} and \eqref{equation:big2} when $g=12$. 
In \S \ref{section:configuration}, we analyze the configurations 
of lines in prime Fano threefolds of genus $12$. 
In \S \ref{section:special} and \S \ref{section:parameter}, we consider a 
special case where the link \eqref{equation:big} are effectively $\G_m$-equivariant.

\begin{ack}
First of all, the author would like to thank Adrien Dubouloz 
and Takashi Kishimoto for many discussions on prime Fano threefolds of genus $12$ 
with effective $\G_m$-actions during our collaboration on the paper \cite{DFK}. 
In fact, an idea in \cite[Theorem 22]{DFK} is the starting point of the paper. 
The author would like to thank Alexander Kuznetsov for various 
constructive suggestions. 
For example, the proofs of Proposition \ref{proposition:totally-blowup} 
\eqref{proposition:totally-blowup1} and Lemma \ref{lemma:gm} can be simplified 
thanks to his comments. 
The author would like to thank Akihiro Kanemitsu for his deep insights on 
prime Fano threefolds. In fact, he suggested the author to consider 
Proposition \ref{proposition:non-existence-surface}. 
The author would like thank Yujiro Kawamata 
and Evgeny Shinder for many important comments, especially about \S \ref{section:flop}. 
The author would like to thank Lu Qi for organizing an interesting workshop 
``Workshop on K-stability'' at East China Normal University on January 2026, 
where a question that led to the cases genus $9$ or $10$ was raised by an audience.
The author was supported by JSPS KAKENHI Grant Number 26K00601, 
Royal Society International Collaboration Award 
ICA\textbackslash 1\textbackslash 23109 and Asian Young Scientist Fellowship. 
\end{ack}

\section{On basics of birational geometry and three-dimensional 
flops}\label{section:3flops}

We firstly see an elementary fact on the lengths of the intersections of curves and 
submanifolds. 

\begin{lemma}\label{lemma:lengths}
Let $o\in \U$ be a germ of a smooth variety with a closed point, 
and let $\Delta$, $\Xi\subsetneq\U$ be closed (irreducible and reduced) 
subvarieties such that 
$\Delta$ is smooth and $\Xi$ is a curve with $\Xi\not\subset\Delta$. 
Let $\sigma_\Delta\colon\U_\Delta\to\U$ (resp., $\sigma_o\colon\U_o\to \U$)
be the blowup along $\Delta\subset\U$ (resp., along $o\in\U$), let 
$\E_\Delta\subset\U_\Delta$ (resp., $\E_o\subset\U_o$) 
be the exceptional divisor of $\sigma_\Delta$ (resp., $\sigma_o$), and let 
$\Xi_\Delta\subset\U_\Delta$ (resp., $\Xi_o\subset\U_o$)
be the strict transform of $\Xi\subset\U$. 
\begin{enumerate}
\renewcommand{\theenumi}{\arabic{enumi}}
\renewcommand{\labelenumi}{(\theenumi)}
\item\label{lemma:length1}
Let $\Xi^\nu\to\Xi$ be the normalization, and let us 
consider the fiber product 
\[
\Xi^\nu\cap\Delta:=\Xi^\nu\times_\U\Delta
\subset\Xi^\nu. 
\]
Then we have 
\[
\operatorname{length}\left(\sO_{\Xi^\nu\cap\Delta}\right)
=\left(\E_\Delta\cdot\Xi_\Delta\right)_o, 
\]
where $\left(\E_\Delta\cdot\Xi_\Delta\right)_o$ is the local intersection number 
over $o\in\U$, i.e., the sum of the local intersection numbers over all points 
in $\Xi_\Delta$ over $o\in\U$. 
\item\label{lemma:length2}
Let us set $\F_o:=\sigma_\Delta^{-1}(o)\subset\E_\Delta$ with the reduced 
structure. 
We again consider the fiber product
\[
\Xi^\nu\cap\F_o:=\Xi^\nu\times_{\U_\Delta}\F_o
\subset\Xi^\nu, 
\]
where $\Xi^\nu\to\U_\Delta$ is defined to be the composition 
$\Xi^\nu\to\Xi_\Delta\subset\U_\Delta$. Then we have 
\[
\operatorname{length}\left(\sO_{\Xi^\nu\cap\F_o}\right)
=\left(\E_o\cdot\Xi_o\right). 
\]
\end{enumerate}
\end{lemma}

\begin{proof}
\eqref{lemma:length1} 
Let $\nu\colon\Xi^\nu\to\Xi$ be 
the normalization and let $\nu_\Delta\colon\Xi^\nu\to\Xi_\Delta$ be the 
induced morphism. We get the following diagram: 
\[\xymatrix{
\Xi^\nu \ar[r]^-{\nu_\Delta} \ar[dr]_-{\nu}& \Xi_\Delta \ar@{}[r]|{\subset}
\ar[d]^-{\sigma_\Delta|_{\Xi_\Delta}}
& \U_\Delta \ar@{}[r]|{\supset} \ar[d]^-{\sigma_\Delta} & \E_\Delta 
\ar[d]^-{\sigma_\Delta|_{\E_{\Delta}}} \\
& \Xi \ar@{}[r]|{\subset} &  \U \ar@{}[r]|{\supset} & \Delta
}\]
Let $I_\Delta\subset\sO_\U$ (resp., $I_\Xi\subset\sO_\U$) be the ideal sheaf 
of $\Delta\subset\U$ (resp., $\Xi\subset\U$). 
Let $J\subset\sO_{\Xi^\nu}$ be the ideal 
sheaf of $\Xi^\nu\cap\Delta\subset\Xi^\nu$. From the definition of 
$\Xi^\nu\cap\Delta$, there is a natural surjection 
\[
\nu^*\left(I_\Delta|_\Xi\right)\twoheadrightarrow J.
\]
By the definition of the blowup, there is a canonical surjection 
\[
\sigma_\Delta^* I_\Delta\twoheadrightarrow\sO_\U(-\E_\Delta). 
\]
By taking $\nu_\Delta^*$, we get the surjection 
\[
\nu^*\left(I_\Delta|_\Xi\right)\twoheadrightarrow\nu_\Delta^*\left(
\sO_\U(-\E_\Delta)|_{\Xi_\Delta}\right).
\]
The above surjection factors through $\nu^*\left(I_\Delta|_\Xi\right)
\twoheadrightarrow J$. Moreover, since $\Xi^\nu$ is a smooth curve, the ideal 
sheaf $J$ on $\Xi^\nu$ is invertible. Thus we have an isomorphism 
\[
J\cong \nu_\Delta^*\left(\sO_\U(-\E_\Delta)|_{\Xi_\Delta}\right).
\]
Recall that the local intersection number $\left(\E_\Delta\cdot\Xi_\Delta\right)_o$ 
is nothing but the sum of the orders of zeros of 
$\nu_\Delta^*\left(\sO_\U(-\E_\Delta)|_{\Xi_\Delta}\right)$. Thus we get the 
assertion \eqref{lemma:length1}. 

\eqref{lemma:length2} 
We may assume that $\{o\}\subsetneq\Delta$. Let $\sigma_{\F_o}\colon
\tilde{\U}\to\U_\Delta$ be the blowup along $\F_o\subset\U_\Delta$, let 
$\tilde{\E}_o\subset\tilde{\U}$ be the $\sigma_{\F_o}$-exceptional divisor 
and let $\tilde{\Xi}\subset\tilde{\U}$ be the strict transform of $\Xi\subset\U$. 
It is well-known that there is a birational morphism $\sigma_{\Delta'}\colon\tilde{\U}
\to\U_o$ with $\sigma_o\circ\sigma_{\Delta'}=\sigma_\Delta\circ\sigma_{\F_o}$. 
Moreover, the morphism $\sigma_{\Delta'}$ is the blowup along 
$\Delta':=(\sigma_o)^{-1}_*\Delta\subset\U_o$, the $\sigma_{\Delta'}$-exceptional 
divisor coincides with $(\sigma_{\F_o})^{-1}_*\E_\Delta$, and 
$\tilde{\E}_o=(\sigma_{\Delta'})^{-1}_*\E_o$. 
We summarize the diagram: 
\[\xymatrix{
& \tilde{\E}_o \ar@{}[d]|{\cap}\ar[r] & 
\F_o \ar@{}[d]|-{\cap} \ar@{}[r]|-{:=} & \sigma_\Delta^{-1}(o) \\
\left(\sigma_{\F_o}\right)^{-1}_*\E_\Delta \ar[d] \ar@{}[r]|-{\subset} 
& \tilde{\U} \ar[r]^-{\sigma_{\F_o}} \ar[d]_-{\sigma_{\Delta'}} & 
\U_\Delta \ar[d]^-{\sigma_\Delta} \ar@{}[r]|-{\supset} & \E_\Delta \ar[d] \\
\Delta' \ar@{}[r]|-{\subset} & \U_o \ar[r]_-{\sigma_o}  \ar@{}[d]|-{\cup} & 
\U \ar@{}[r]|-{\supset} \ar@{}[d]|-*[@]{\ni} & \Delta \\
& \E_o \ar[r] & o & 
}\]
We have already seen in \eqref{lemma:length1} that 
\[
\operatorname{length}\left(\sO_{\Xi^\nu\cap\F_o}\right)=\left(\tilde{\E}_o\cdot
\tilde{\Xi}\right)
\]
holds. On the other hand, since $(\sigma_{\Delta'})^*\E_o=\tilde{\E}_o$, we have 
\[
\left(\tilde{\E}_o\cdot\tilde{\Xi}\right)=\left(\E_o\cdot(\sigma_{\Delta'})_*
\tilde{\Xi}\right)=\left(\E_o\cdot\Xi_o\right).
\]
Thus we get the assertion \eqref{lemma:length2}. 
\end{proof}

For the minimal model program, we refer the readers to \cite{KoMo}. 
We fix several terminologies. 

\begin{definition}\label{definition:birational}
Let $\X$ be a normal projective variety. 
\begin{enumerate}
\renewcommand{\theenumi}{\arabic{enumi}}
\renewcommand{\labelenumi}{(\theenumi)}
\item\label{definition:birational1}
Let $\chi\colon\X\dashrightarrow\X'$ be a birational map between normal projective 
varieties. The \emph{exceptional locus} $\Exc(\chi)\subset\X$ of $\chi$ is defined 
to be the smallest closed subset of $\X$ such that the restriction 
$\chi|_{\X\setminus\Exc(\chi)}$ is an isomorphism onto its image. 
\item\label{definition:birational2} \cite[Definition 1.8]{HK}
A birational map $\chi\colon \X\dashrightarrow\X'$ between normal projective varieties 
is said to be \emph{small} 
if both $\Exc(\chi)\subset \X$ and $\Exc(\chi^{-1})\subset\X'$ are of codimension 
bigger than $1$. If moreover both $\X$ and $\X'$ are $\Q$-factorial, then 
we say that $\chi$ is a \emph{small $\Q$-factorial modification of $\X$}. 
\item\label{definition:birational3}
A \emph{contraction morphism} $\beta\colon \X \to \Y$ is a morphism with 
$\Y$ normal projective and $\beta_*\sO_{\X}=\sO_{\Y}$. 
If moreover $\beta$ is not an isomorphism and all curves in $\X$ contracted 
by $\beta$ are numerically proportional, then we say that $\beta$ is an 
\emph{elementary contraction morphism}. 
\item\label{definition:birational4}
Assume that $\beta\colon\X\to \Y$ is a contraction morphism. 
Let $A_{\X}$ be a $\Q$-Cartier $\Q$-divisor on $\X$. We say that $\beta$ is 
\emph{$A_{\X}$-negative} (resp., \emph{$A_{\X}$-trivial}, \emph{$A_{\X}$-positive})
if $A_{\X}$ is anti-ample over $\Y$ (resp., $A_{\X}$ is numerically trivial over $\Y$, 
$A_{\X}$ is ample over $\Y$). 
\item\label{definition:birational5} \cite[\S 6.1]{KoMo}
Assume that a contraction morphism $\beta\colon \X\to \Y$ is small. 
We say that $\beta$ is a \emph{flopping contraction} if $K_{\X}$ is $\Q$-Cartier and 
$\beta$ is $K_{\X}$-trivial. 
\emph{Flopping curves of $\beta$} are the curves in $\X$ contracted by $\beta$. 
If moreover $\beta$ is an elementary contraction morphism, then we call $\beta$ an 
\emph{elementary flopping contraction}. 
\item\label{definition:birational6} \cite[\S 6.1]{KoMo}
Assume that a contraction morphism $\beta\colon \X\to \Y$ is small, $\X$ is 
$\Q$-factorial and $A_{\X}$ is a $\Q$-divisor on $\X$ such that $\beta$ is 
$A_{\X}$-negative. The \emph{$A_{\X}$-flip of $\beta$} is the diagram 
\[\xymatrix{
\X \ar[dr]_-{\beta} \ar@{-->}[rr]^-{\chi} && \X^+ \ar[dl]^-{\beta^+}\\
& \Y&
}\]
(or, the birational map $\chi\colon\X\dashrightarrow\X^+$) such that the morphism 
$\beta^+$ is obtained by 
\[
\Proj_{\Y}\bigoplus_{m\geq 0}\beta_*\sO_{\X}\left(\lfloor m A_{\X}\rfloor\right). 
\]
If $\beta$ is a flopping contraction, then we call it the \emph{$A_{\X}$-flop of $\beta$}, 
and curves in $\X^+$ contracted by $\beta^+$ are said to be \emph{flopped curves 
of $\beta$} (or, \emph{of $\chi$}). 
We remark that the $A_{\X}$-flip may not exist in general. Moreover, the variety $\X^+$ 
may not be $\Q$-factorial in general. 
If $\beta$ is an elementary contraction morphism, then the $A_{\X}$-flip does not 
depend on the choice of $A_{\X}$, and the rational map $\chi$ 
is a small $\Q$-factorial modification of $\X$. 
If $\beta$ is an elementary contraction and 
a $\Q$-divisor $H_{\X}$ on $\X$ satisfies that $\beta$ is $H_\X$-negative, then we 
sometimes say that 
\emph{$\chi$ is $H_{\X}$-negative} for simplicity. 
If $\beta$ is an elementary flopping 
contraction, then we call $\chi$ the \emph{elementary flop of $\beta$}. 
\end{enumerate}
\end{definition}

We frequently use the following result: 

\begin{thm}[{\cite[Proposition 2.2 and Theorem 2.4]{Kol} and \cite[Corollary 
A.16]{Pr25}}]\label{thm:kollar}
Let $\X$ be a $3$-dimensional smooth projective variety and let $\beta\colon\X\to \Y$ 
be an elementary flopping contraction morphism. Then the flop 
$\beta^+\colon\X^+\to\Y$ of $\beta$ exists, and $\X^+$ is smooth. 
Moreover, for any point $p\in\Y$ with $p\in\beta\left(\Exc(\beta)\right)$, 
we have a (non-standard) isomorphism $\beta^{-1}(p)\cong(\beta^+)^{-1}(p)$ such that, 
for any irreducible component $\mathbf{B}\subset\beta^{-1}(p)$, if $\mathbf{B}^+
\subset(\beta^+)^{-1}(p)$ be the image of $\mathbf{B}$ under the above isomorphism, 
then we have 
\[
\left(H_{\X^+}\cdot \mathbf{B}^+\right)=-\left(H_{\X}\cdot \mathbf{B}\right)
\] 
for any $\Q$-divisor $H_{\X}$ on $\X$, where 
$H_{\X^+}:=((\beta^+)^{-1}\circ\beta)_*H_\X$. 
\end{thm}

If $\beta^{-1}(p)$ is irreducible, then the irreducible curve $(\beta^+)^{-1}(p)$ is 
said to be \emph{the flopped curve of (the flopping curve) $\beta^{-1}(p)$ with respects 
to the elementary flop $(\beta^+)^{-1}\circ\beta$}. 

The following lemma is just an application of the negativity lemma 
\cite[Lemma 3.39]{KoMo}, but is powerful in order to show the results in 
\S \ref{section:flop} and \S \ref{section:configuration}. 

\begin{lemma}[{cf.\ \cite[Lemma 3.8]{C}}]\label{lemma:negativity}
Let $\X$ be a normal $\Q$-factorial projective variety, $A_{\X}$ be a $\Q$-divisor 
on $\X$, and let $\beta\colon\X\to \Y$ be a small and $A_{\X}$-negative 
contraction morphism. Assume that the $A_{\X}$-flip 
\[\xymatrix{
\X \ar[dr]_-{\beta} \ar@{-->}[rr]^-{\chi} && \X^+ \ar[dl]^-{\beta^+}\\
& \Y&
}\]
of $\beta$ exists and $\X^+$ is $\Q$-factorial. 
\begin{enumerate}
\renewcommand{\theenumi}{\arabic{enumi}}
\renewcommand{\labelenumi}{(\theenumi)}
\item\label{lemma:negativity1}
We have $\Exc(\chi)=\Exc(\beta)$. 
\item\label{lemma:negativity2}
Assume that an irreducible curve $\Xi\subset\X$ satisfies that 
$\Xi\not\subset\Exc(\chi)$, and let $\Xi^+:=\chi_*\Xi\subset\X^+$ be the 
strict transform of $\Xi$ to $\X^+$. Set $A_{\X^+}:=\chi_*A_{\X}$. 
\begin{enumerate}
\renewcommand{\theenumii}{\roman{enumii}}
\renewcommand{\labelenumii}{(\theenumii)}
\item\label{lemma:negativity21}
We have 
\[
\left(A_{\X}\cdot\Xi\right)\geq \left(A_{\X^+}\cdot \Xi^+\right). 
\]
\item\label{lemma:negativity22}
If moreover $\Xi\cap\Exc(\chi)\neq \emptyset$, then we have 
\[
\left(A_{\X}\cdot\Xi\right)> \left(A_{\X^+}\cdot \Xi^+\right). 
\]
\item\label{lemma:negativity23}
If moreover both $A_{\X}$ on $\X$ and $A_{\X^+}$ on $\X^+$ are Cartier divisors and 
if we set $m:=\#\left(\Xi\cap\Exc(\chi)\right)_{\operatorname{red}}$, then we have 
\[
\left(A_{\X}\cdot\Xi\right)-\left(A_{\X^+}\cdot \Xi^+\right)\geq m. 
\]
\end{enumerate}
\end{enumerate}
\end{lemma}

\begin{proof}
\eqref{lemma:negativity1}
Since $\chi$ is an isomorphism on $\X\setminus\Exc(\beta)$, we have 
$\Exc(\chi)\subset\Exc(\beta)$. Note that $\Exc(\beta)$ is the union of curves 
$B\subset\X$ contracted by $\beta$. Take any such curve $B\subset\X$. 
Assume that $B\not\subset\Exc(\chi)$. 
Then we can consider the strict transform $B^+\subset\X^+$ of $B$ on $\X^+$. 
Note that $-A_{\X}$ is $\beta$-ample. Thus there exists an ample Cartier 
divisor $H_{\Y}$ on $\Y$ such that $\beta^*H_{\Y}-A_{\X}$ is an ample $\Q$-divisor 
on $\X$. Hence there exists an effective $\Q$-divisor $D_{\X}\sim_\Q \beta^*H_{\Y}
-A_{\X}$ on $\X$ such that $B\not\subset\Supp D_{\X}$ and $\left(D_{\X}\cdot
B\right)>0$. Since $D_{\X^+}:=\chi_*D_{\X}$ satisfies that $D_{\X^+}\sim_\Q
(\beta^+)^*H_{\Y}-A_{\X^+}$ and $B^+\not\subset\Supp D_{\X^+}$, we have 
\[
-\left(A_{\X^+}\cdot B^+\right)=\left(D_{\X^+}\cdot B^+\right)\geq 0. 
\]
However, since $B^+$ is contracted by $\beta^+$ and $A_{\X^+}$ is $\beta^+$-ample, 
this leads to a contradiction. Thus we have $\Exc(\chi)\supset\Exc(\beta)$ and 
we get the assertion \eqref{lemma:negativity1}.

\eqref{lemma:negativity2}
Let 
\[\xymatrix{
&\tilde{\X} \ar[dl]_-{\gamma} \ar[dr]^-{\gamma^+}& \\
\X \ar@{-->}[rr]^-{\chi} && \X^+
}\]
be a resolution of indeterminacy of $\chi$ with $\tilde{\X}$ normal such that 
$\gamma$ is an isomorphism over $\X\setminus\Exc(\chi)$. 
Let $\E_1,\dots,\E_m\subset\tilde{\X}$ be the set of $\gamma$-exceptional 
(equivalently, $\gamma^+$-exceptional) prime divisors on $\tilde{\X}$. 
Note that the union $\bigcup_{i=1}^m\gamma(\E_i)$ is contained in $\Exc(\chi)$. 
Let $\tilde{\Xi}\subset\tilde{\X}$ be the strict transform of $\Xi$ on $\tilde{\X}$. 
Moreover, set 
\[
\gamma^*A_{\X}-(\gamma^+)^*A_{\X^+}=:\E=\sum_{i=1}^m e_i\E_i. 
\]
Note that 
\[
\left(A_{\X}\cdot\Xi\right)-\left(A_{\X^+}\cdot \Xi^+\right)
=\left(\E\cdot\tilde{\Xi}\right)
\]
holds. Since $\gamma_*\E=0$ and $-\E$ is $\gamma$-nef, the $\Q$-divisor $\E$ is 
effective by the negativity lemma \cite[Lemma 3.39]{KoMo}. 
Then, 
from the assumption, we have 
$\left(\E\cdot\tilde{\Xi}\right)\geq 0$. 

From now on, assume that there exists a point $x\in\Xi\cap\Exc(\chi)$. 
Again by the negativity lemma \cite[Lemma 3.39]{KoMo},  
either $\gamma^{-1}(x)\cap\Supp\E=\emptyset$ or 
$\gamma^{-1}(x)\subset\Supp\E$ holds. Assume that 
$\gamma^{-1}(x)\cap\Supp\E=\emptyset$. Take any irreducible curve $B\subset\X$
with $\beta_*B=0$ and $x\in B$. (By \eqref{lemma:negativity1}, we have 
$\Exc(\beta)=\Exc(\chi)$.) Let us also take an irreducible curve 
$\tilde{B}\subset\tilde{\X}$ with $\gamma(\tilde{B})=B$. Since 
$\tilde{B}\not\subset\Supp\E$ and $(\beta^+\circ\gamma^+)_*\tilde{B}=0$, we have 
\[
0\leq\left(\E\cdot\tilde{B}\right)=\left(A_{\X}\cdot\gamma_*\tilde{B}\right)
-\left(A_{\X^+}\cdot(\gamma^+)_*\tilde{B}\right)<0, 
\]
a contradiction. Thus $\gamma^{-1}(x)\subset\Supp\E$ holds. 
In particular, we get 
$\left(\E\cdot\tilde{\Xi}\right)>0$. 
If both $A_{\X}$ and $A_{\X^+}$ are Cartier, then $\E$ is also Cartier. 
Thus, for any point $x\in\Xi\cap\Exc(\chi)$, the local intersection number 
$\left(\E\cdot\tilde{\Xi}\right)$ over $x\in\Xi$ is a positive integer. 
Thus we get the assertion \eqref{lemma:negativity2}. 
\end{proof}

\begin{definition}\label{definition:cones}
Let $\X$ be a normal projective variety and let $\ND(\X)$ be the set of numerically 
equivalence classes of $\R$-Cartier $\R$-divisors in $\X$. 
The \emph{nef cone $\Nef(\X)$} (resp., \emph{movable cone $\Mov(\X)$}, 
\emph{pseudo-effective cone} $\Psef(\X)$) 
is the smallest closed cone in $\ND(\X)$ containing 
the classes of all nef (resp., movable, effective) Cartier divisors on $\X$. 
If $\X$ is $\Q$-factorial and 
$\X\dashrightarrow\X'$ is a small $\Q$-factorial modification, 
then we can canonically identify the spaces $\ND(\X)$ and $\ND(\X')$. 
We often consider the cones $\Nef(\X)$ and $\Nef(\X')$ in the same space 
$\ND(\X)$ under the above identification. 
\end{definition}

\begin{definition}\label{definition:fano}
\begin{enumerate}
\renewcommand{\theenumi}{\arabic{enumi}}
\renewcommand{\labelenumi}{(\theenumi)}
\item\label{definition:fano1}
A \emph{smooth Fano manifold} (resp., \emph{smooth weak Fano manifold}) 
is defined to be 
a smooth projective variety $\X$ with $-K_{\X}$ ample (resp., nef and big). 
We say that $\X$ is a \emph{smooth Fano threefold} 
(resp., \emph{smooth weak Fano threefold}) if moreover $\X$ is $3$-dimensional. 
For a smooth weak Fano manifold $\X$, the anti-canonical divisor $-K_{\X}$ is 
semiample by the base point free theorem \cite[Theorem 3.3]{KoMo}. 
The \emph{anti-canonical model of $\X$} is 
the birational contraction morphism $\alpha\colon\X\to \bar{\X}$ defined by 
sufficiently divisible multiples of $-K_{\X}$. 
Obviously, if $\X$ is a smooth Fano manifold, then the anti-canonical model of $\X$ is 
$\X$ itself. 
\item\label{definition:fano2}
More generally, for smooth projective varieties $\X$, $\U$ together with a morphism 
$\pi\colon\X\to\U$, assume that $-K_{\X}$ is nef and big over $\U$. 
Then, again by the base point free theorem \cite[Theorem 3.24]{KoMo}, 
we can consider the \emph{anti-canonical model} 
$\X\to\bar{\X}\to\U$ 
\emph{of $\X$ over} $\U$ defined to be 
\[
\bar{\X}:=\Proj_{\U}\bigoplus_{m\geq 0}\pi_*\sO_{\X}(-mK_{\X}).
\]
\end{enumerate}
\end{definition}

The following lemma is probably well-known. We give a proof for the 
readers' convenience. 

\begin{lemma}[{cf.\ \cite[Proposition 4.5]{MM83}}]\label{lemma:image-weak}
Let $\Y$ be a $3$-dimensional smooth projective variety, let $\Xi\subset\Y$ be a 
smooth irreducible non-rational curve, and let $\sigma_\Xi\colon\X\to\Y$ be 
the blowup of $\Y$ along $\Xi\subset\Y$. If $\X$ is a smooth weak Fano threefold, 
then so is $\Y$. 
\end{lemma}

\begin{proof}
Note that $-K_{\X}+\E=\sigma_\Xi^*(-K_{\Y})$, where $\E$ is the exceptional divisor of 
$\sigma_\Xi$. Thus, the anti-canonical divisor $-K_{\Y}$ of $\Y$ is obviously big. 
Thus it is enough to show the inequality $\left(-K_{\Y}\cdot\Theta\right)\geq 0$
for any irreducible curve $\Theta\subset\Y$. 
If $\Theta\neq\Xi$, then we have 
$\left(-K_{\Y}\cdot\Theta\right)=\left(-K_{\X}\cdot(\sigma_\Xi)^{-1}_*\Theta\right)
+\left(\E\cdot(\sigma_\Xi)^{-1}_*\Theta\right)\geq 0$. 
On the other hand, since $\Xi$ is non-rational, 
we have $\left(-K_{\Y}\cdot\Xi\right)=\left((-K_{\X})^{\cdot 2}\cdot\E\right)-2
+2p_a(\Xi)\geq 0$ (see \cite[Lemma 2.1]{MM85}). 
\end{proof}

\begin{proposition}\label{proposition:MDS}
Let $\X$ be a smooth weak Fano threefold. 
\begin{enumerate}
\renewcommand{\theenumi}{\arabic{enumi}}
\renewcommand{\labelenumi}{(\theenumi)}
\item\label{proposition:MDS1}
The set of small $\Q$-factorial modifications of $\X$ 
is a finite set 
\[
\left\{\X=\X^1,\dots,\X^m\right\}.
\]
(Precisely, we consider the set of birational maps $\X\dashrightarrow\X^i$.)
Moreover, each $\X^i$ is a smooth weak Fano threefold. 
In particular, for any $1\leq i\leq m$, the nef cone $\Nef\left(\X^i\right)$ of $\X^i$ 
contains the class $\left[-K_\X\right]\in\ND(\X)$ of $-K_\X$. 
\item\label{proposition:MDS2}
Assume moreover that the anti-canonical model $\X\to\bar{\X}$ of $\X$ is small.
Then the class $\left[- K_{\X}\right]$ belongs to the interior of the 
movable cone of $\X$.
\end{enumerate}
\end{proposition}

\begin{proof}
If $\X$ is a smooth Fano threefold, then $\X$ itself is the unique small $\Q$-factorial 
modification of $\X$ by \cite[Theorem 3.3]{Mo82}, and the assertion is trivial. 
We assume that $\X$ is not a smooth Fano threefold from now on. 
Note that $\X$ is a Mori dream space by \cite[Corollary 1.3.2]{BCHM}. 
By \cite[Definition 1.10]{HK}, the number 
of small $\Q$-factorial modifications is finite and 
the movable cone of $\X$ is the union of the nef cones of $\X^i$. 
Moreover, for any $1\leq i\leq m$, the modification $\X^1\dashrightarrow\X^i$ 
can be decomposed into a sequence 
\[\xymatrix{
\X^{j_l} \ar[rd] \ar@{-->}[rr]& & \X^{j_{l+1}}\ar[ld]\\
&\bar{\X}^{j_l}&
}\]
of elementary flips. In this case, the nef cone $\Nef\left(\bar{\X}^{j_l}\right)$ of 
$\bar{\X}^{j_l}$
is a facet of both $\Nef\left(\X^{j_l}\right)$ and $\Nef\left(\X^{j_{l+1}}\right)$. 
Assume that we can inductively show that $\X^{j_l}$ is a smooth weak Fano 
threefold. Then the class of $-K_{\X}$ is contained in $\Nef\left(\X^{j_l}\right)$ 
since $\X^{j_l}$ is a smooth weak Fano threefold. 
If the class of $-K_{\X}$ is not contained in 
$\Nef\left(\bar{\X}^{j_l}\right)$, then the contraction $\X^{j_l}\to \bar{\X}^{j_l}$ is a 
$K_{\X^{j_l}}$-negative elementary contraction. Since the contraction is 
small, it contradicts \cite[Theorem 3.3]{Mo82}. Thus 
the class of $-K_{\X}$ is contained in $\Nef\left(\bar{\X}^{j_l}\right)$. Then 
$\X^{j_l}\dashrightarrow\X^{j_{l+1}}$ is an elementary flop. Thus $\X^{j_{l+1}}$ is 
also a \emph{smooth} weak Fano threefold by Theorem \ref{thm:kollar}. 
Thus we get the assertion \eqref{proposition:MDS1}. 

If the assertion \eqref{proposition:MDS2} does not hold, then the class of $-K_\X$ 
belongs to a facet of 
the movable cone of $\X$. In other words, there exists $1\leq i\leq m$ and 
there exists an elementary non-small contraction 
$\X^i\to\Y^i$ such that the class of $-K_\X$ 
belong to $\Nef\left(\Y^i\right)$. But it leads to a contradiction since 
the small anti-canonical model $\X^i\to\bar{\X}$ of $\X^i$ 
is factored by $\X^i\to\Y^i$. 
Thus the assertion \eqref{proposition:MDS2} also holds. 
\end{proof}

We prepare the following proposition which is an analogue of the 
Castelnuovo--Mumford regularity theorem \cite[\S 1.8]{L1}. On the other hand, we 
consider regularities with respect to \emph{non-ample} line bundles. 

\begin{proposition}\label{proposition:CM}
Let $\X, \mathbb{S}$ be smooth projective varieties with $\dim \X=n$ and 
$\dim \mathbb{S}=n-1$, let $\pi\colon\X\to \mathbb{S}$ be a flat morphism such that 
the anti-canonical divisor $-K_{\X}$ is $\pi$-nef and $\pi$-big. Take any 
$\pi$-nef line bundle $\sL$ on $\X$. 
\begin{enumerate}
\renewcommand{\theenumi}{\arabic{enumi}}
\renewcommand{\labelenumi}{(\theenumi)}
\item\label{proposition:CM1}
For any $j\geq 1$, we have $R^j\pi_*\sL=0$. Moreover, the line bundle 
$\sL$ is $\pi$-free, i.e., the adjoint homomorphism
$\pi^*\pi_*\sL\to\sL$ is surjective. 
\item\label{proposition:CM2}
Take any ample and globally generated 
line bundle $A$ on $\mathbb{S}$. Assume that 
\[
H^j\left(\X, \sL\otimes\pi^*A^{\otimes -j}\right)=0 \quad\text{for all }j\geq 1. 
\]
Then we have the following: 
\begin{enumerate}
\renewcommand{\theenumii}{\roman{enumii}}
\renewcommand{\labelenumii}{(\theenumii)}
\item\label{proposition:CM21}
For any $k\in\Z_{\geq 0}$, the line bundle $\sL\otimes\pi^*A^{\otimes k}$ is 
globally generated. 
\item\label{proposition:CM22}
For any $k\in\Z_{\geq 0}$, the evaluation homomorphism 
\[
H^0\left(\X,\sL\right)\otimes_\Bbbk H^0\left(\X,\pi^*A^{\otimes k}\right)
\to H^0\left(\X,\sL\otimes\pi^*A^{\otimes k}\right)
\]
is surjective. 
\item\label{proposition:CM23}
For any $k\in\Z_{\geq 0}$ and for any $j\geq 1$, we have 
$H^j\left(\X, \sL\otimes\pi^*A^{\otimes k-j}\right)=0.$
\end{enumerate}
\end{enumerate}
\end{proposition}

\begin{proof}
\eqref{proposition:CM1}
By the Kawamata--Viehweg 
vanishing theorem (see \cite[\S 9.1]{L2}), we have $R^j\pi_*\sL=0$ 
for any $j\geq 1$. In particular, we have $R^j\pi_*\sO_{\X}=0$ 
for any $j\geq 1$. By 
the cohomology and base change theorem \cite[\S 5, Corollaries 3 and 4]{mumford}, 
for any closed point $p\in\mathbb{S}$, if we set $l_p:=\pi^{-1}(p)\subset\X$, 
we have 
\begin{equation}\label{equation:prop-mumford}
\pi_*\sL\otimes\Bbbk(p)\cong H^0\left(l_p,\sL|_{l_p}\right)
\end{equation}
and $H^j\left(l_p, \sL|_{l_p}\right)=0$ for any $j\geq 1$. In particular, we have 
$H^1\left(l_p,\sO_{l_p}\right)=0$. By \cite[Lemma 2.8]{nakayama}, the line bundle 
$\sL|_{l_p}$ on $l_p$ is globally generated. Thus, from the above isomorphism 
\eqref{equation:prop-mumford} and Nakayama's lemma, we get the assertion 
\eqref{proposition:CM1}, 

\eqref{proposition:CM2}
Firstly, it is enough to show the assertion \eqref{proposition:CM22} for $k=1$ and 
the assertion \eqref{proposition:CM23} for $k=1$. Indeed, those assertions imply the 
following:  
\begin{itemize}
\item
The assertion \eqref{proposition:CM23} holds for any $k\in\Z_{\geq 0}$ by induction 
on $k$. 
\item
The assertion \eqref{proposition:CM22} holds for any $k\in\Z_{\geq 0}$ by induction 
on $k$. In fact, assume that \eqref{proposition:CM22} holds for $k=1$ and $k=k_0$. 
From the commutative diagram
\[\xymatrix{
H^0\left(\X,\sL\right)\otimes_\Bbbk H^0\left(\X,\pi^* A\right)
\otimes_\Bbbk H^0\left(\X,\pi^* A^{\otimes k_0}\right) \ar@{->>}[d] \ar[r] & 
H^0\left(\X,\sL\right)\otimes_\Bbbk H^0\left(\X,\pi^* A^{\otimes k_0+1}\right) 
\ar[d] \\
H^0\left(\X,\sL\otimes\pi^*A\right)\otimes_\Bbbk 
H^0\left(\X,\pi^* A^{\otimes k_0}\right) \ar@{->>}[r] & 
H^0\left(\X,\sL\otimes\pi^* A^{\otimes k_0+1}\right), 
}\]
we get the assertion \eqref{proposition:CM22} for $k=k_0+1$. 
\item
For the assertion \eqref{proposition:CM21}, we may assume that $k=0$. Since 
$A$ is ample, there exists $m\gg 0$ such that the coherent sheaf $\pi_*\sL\otimes
 A^{\otimes m}$ is globally generated. 
Note that the line bundle $\sL\otimes\pi^*A^{\otimes m}$ is $\pi$-free by 
\eqref{proposition:CM1}. From the sequence 
\begin{eqnarray*}
H^0\left(\X,\sL\otimes\pi^* A^{\otimes m}\right)\otimes_\Bbbk\sO_{\X}
&=&\pi^*\left(H^0\left(\mathbb{S}, \pi_*\sL\otimes A^{\otimes m}\right)
\otimes_\Bbbk\sO_{\mathbb{S}}\right)\\
\twoheadrightarrow
\pi^*\left(\pi_*\sL\otimes A^{\otimes m}\right)&\twoheadrightarrow& 
\sL\otimes \pi^*A^{\otimes m}, 
\end{eqnarray*}
we get that $\sL\otimes\pi^*A^{\otimes m}$ is globally generated. 
From the assertion \eqref{proposition:CM22} for $k=m$ and the following 
commutative diagram
\[\xymatrix{
H^0\left(\X,\sL\right)\otimes_\Bbbk H^0\left(\X,\pi^* A^{\otimes m}\right)
\otimes_\Bbbk \sO_{\X} \ar@{->>}[d] \ar[r] & 
\left(H^0\left(\X,\sL\right)\otimes_\Bbbk\sO_{\X}\right)\otimes\pi^* 
A^{\otimes m} \ar[d] \\
H^0\left(\X,\sL\otimes\pi^*A^{\otimes m}\right)\otimes_\Bbbk 
\sO_{\X} \ar@{->>}[r] & 
\sL\otimes\pi^* A^{\otimes m}, 
}\]
we get the assertion \eqref{proposition:CM21} for $k=0$. 
\end{itemize}
Since $A$ is globally generated, the complete linear system of $A$ induces a 
finite morphism $\phi\colon \mathbb{S}\to\pr^N$. The pullback of the 
Koszul complex on $\pr^N$ gives the exact sequence 
\[
0\leftarrow\sO_{\X}\leftarrow
H^0\left(\mathbb{S},A\right)\otimes_\Bbbk\pi^*A^{\otimes -1}\leftarrow
\wedge^2H^0\left(\mathbb{S},A\right)\otimes_\Bbbk\pi^*A^{\otimes -2}
\leftarrow
\wedge^3H^0\left(\mathbb{S},A\right)\otimes_\Bbbk\pi^*A^{\otimes -3}
\leftarrow
\cdots.
\]
For any $j \geq 1$, if we tensor the above by $\sL\otimes\pi^* A^{\otimes 1-j}$, then 
we can show the equality $H^j\left(\X,\sL\otimes\pi^*A^{\otimes 1-j}\right)=0$. 
Thus we get the assertion \eqref{proposition:CM23} for $k=1$. 
If we tensor the above by $\sL\otimes\pi^*A$, then we can show 
that the evaluation homomorphism 
\[
H^0\left(\X, \sL\right)\otimes_\Bbbk\pi^*H^0\left(\mathbb{S},A\right)
\to H^0\left(\X, \sL\otimes\pi^*A\right)
\]
is surjective. Thus the assertion \eqref{proposition:CM22} holds for $k=1$. 
\end{proof}

We see several examples of three-dimensional flops. 

\begin{example}\label{example:3flop}
Let $\X$ be a $3$-dimensional smooth projective variety and let 
$\Xi_1$, $\Xi_2\subset\X$ be distinct irreducible smooth projective curves such that 
\[
\left(\Xi_1\cap\Xi_2\right)_{\operatorname{red}}=\left\{o_1,\dots,o_m\right\}
\]
with $m\geq 1$. For $\{i,j\}=\{1,2\}$, let $\sigma_i\colon\X_i\to\X$ be the blowup 
along $\Xi_i\subset\X$, let $\Xi'_j\subset\X_i$ be the strict transform of 
$\Xi_j\subset\X$ on $\X_i$, and let let $\tau_j\colon \X_{ij}\to \X_i$ be the blowup 
along $\Xi'_j\subset\X_i$. 
Moreover, for any $1\leq k\leq m$, let $\B^k_{ij}\subset\X_{ij}$ be the strict transform 
of the curve $\sigma^{-1}_i(o_k)\subset \X_i$ to $\X_{ij}$. 
Obviously, the induced birational map $\chi\colon \X_{12}\dashrightarrow\X_{21}$ 
is a small $\Q$-factorial modification. 
By Lemma \ref{lemma:lengths}, we have $\left(-K_{\X_{ij}}\cdot \B^k_{ij}\right)=0$. 
In particular, the anti-canonical divisor $-K_{\X_{ij}}$ is nef and big over $\X$. 
Let 
$\beta_{ij}\colon\X_{ij}\to \X'$ be the anti-canonical model of $\X_{ij}$ over $\X$. 
Both $\beta_{12}$ and $\beta_{21}$ are elementary 
small contractions. 
Since $\chi$ is an isomorphism in codimension $1$, 
the image $\X'$ does not depend on $i$, $j$. 
Let $\gamma\colon\X'\to\X$ be the natural morphism. 
We get the following diagram: 
\[\xymatrix{
\X_{12} \ar@{-->}[rr]^-{\chi} \ar[dr]^-{\beta_{12}} \ar[d]_-{\tau_2}&& 
\X_{21} \ar[dl]_-{\beta_{21}} \ar[d]^-{\tau_1}\\
\X_1 \ar[dr]_-{\sigma_1} & \X' \ar[d]_(.4){\gamma} & \X_2 \ar[dl]^-{\sigma_2} \\
 &\X &
}\]
Note that, the exceptional divisor $\Exc(\tau_2)$ of $\tau_2$ on $\X_{12}$ is 
satisfies that $\beta_{12}$ is $\Exc(\tau_2)$-positive, but 
its strict transform $\chi_*\Exc(\tau_2)$ to $\X_{21}$ satisfies that 
$\beta_{21}$is $(\chi_*\Exc(\tau_2))$-negative. 
Thus $\chi$ is an elementary flop of $\beta_{12}$. Moreover, the flopping curves 
of $\beta_{12}$ are equal to $\B_{12}^1,\dots,\B_{12}^m\subset\X_{12}$, flopped to 
$\B_{21}^1,\dots,\B_{21}^m\subset\X_{21}$. 
If $\Xi_1$ and $\Xi_2$ are transversal at one point
(i.e., $m=1$ and ${\operatorname{length}}_{o_1}\left(\sO_{\Xi_1\cap\Xi_2}\right)=1$), 
then the rational map $\chi$ is nothing but the Atiyah flop. 
\end{example}

\begin{definition}[{\cite{Bea}, \cite[\S 4]{MM85}}]\label{definition:conic-bundle}
Let $\mathbb{S}$ be a smooth projective surface and let $\Y$ be a $3$-dimensional 
smooth projective variety. A morphism $\pi\colon\Y\to\mathbb{S}$ is said to be 
a \emph{conic bundle} (over $\mathbb{S}$) if $\pi$ is flat, 
$\pi_*\sO_{\Y}=\sO_{\mathbb{S}}$ and $-K_{\Y}$ is $\pi$-ample. 
As in \cite[Proposition 4.3]{MM85}, this is equivalent to the condition that every closed 
fiber of $\pi$ is scheme-theoretically isomorphic to a plane conic. 
The \emph{discriminant locus} $\Delta_\pi$ of the conic bundle $\pi$ is defined to be 
\[
\Delta_\pi:=\left\{s\in\mathbb{S}\mid \pi^{-1}(s) \text{ is not smooth}
\right\}\subset\mathbb{S}
\]
with the reduced structure. As in \cite[Proposition 1.2]{Bea}, the locus $\Delta_\pi$ 
is an effective divisor on $\mathbb{S}$ and has at worst nodal singularities. Moreover, 
for $s\in\Delta_\pi$, the fiber $\pi^{-1}(s)$ is non-reduced if and only if $s$ lies on 
the singular point of $\Delta_\pi$. 
When $\Delta_\pi=\emptyset$, we say that the conic bundle $\pi$ is a 
\emph{$\pr^1$-bundle} (over $\mathbb{S}$). 
\end{definition}

\begin{lemma}\label{lemma:conic-bundle}
Let $\Y$, $\mathbb{S}$ be smooth projective varieties with $\dim\Y=3$ and 
$\dim\mathbb{S}=2$, and let $\pi\colon\Y\to\mathbb{S}$ be a $\pr^1$-bundle. 
Assume that a smooth irreducible curve $\Xi\subset\Y$ satisfies that the restriction 
morphism $\pi|_{\Xi}\colon\Xi\to\pi(\Xi)$ is birational. Let $\sigma_\Xi\colon\X\to\Y$ 
be the blowup of $\Y$ along the curve $\Xi\subset\Y$. 
\begin{enumerate}
\renewcommand{\theenumi}{\arabic{enumi}}
\renewcommand{\labelenumi}{(\theenumi)}
\item\label{lemma:conic-bundle1}
The anti-canonical divisor $-K_{\X}$ is $(\pi\circ\sigma_\Xi)$-ample if and only if 
the morphism $\pi|_{\Xi}\colon\Xi\to\pi(\Xi)$ is an isomorphism (i.e., 
the curve $\pi(\Xi)$ is smooth). Moreover, if the above condition is satisfied, then 
the morphism $\pi\circ\sigma_\Xi\colon\X\to\mathbb{S}$ is a conic bundle, and 
$\Delta_{\pi\circ\sigma_\Xi}=\pi(\Xi)$ holds. 
\item\label{lemma:conic-bundle2}
The anti-canonical divisor $-K_{\X}$ is $(\pi\circ\sigma_\Xi)$-nef if and only if 
$\mult_p\left(\pi(\Xi)\right)\leq 2$ holds for any $p\in\pi(\Xi)$. 
\end{enumerate}
\end{lemma}

\begin{proof}
Firstly, if $\pi|_{\Xi}\colon\Xi\to\pi(\Xi)$ is an isomorphism, then it is well-known 
(see \cite[Definition 4.11]{MM85}) that $\pi\circ\sigma_\Xi$ is a conic bundle
with $\Delta_{\pi\circ\sigma_\Xi}=\pi(\Xi)$. 
Take any point $p\in\pi(\Xi)$. Set $l:=\pi^{-1}(p)\subset\Y$  and 
$\tilde{l}:=(\sigma_\Xi)^{-1}_*l\subset\X$. It is enough to show the equality 
\begin{equation}\label{equation:conic-bundle}
\left(-K_{\X}\cdot\tilde{l}\right)=2-\mult_p\left(\pi(\Xi)\right). 
\end{equation}
By Lemma \ref{lemma:lengths}, we have 
$\left(-K_{\X}\cdot\tilde{l}\right)=2-\operatorname{length}\left(\sO_{\Xi\cap l}\right)$. 
Let $\sigma_p\colon\mathbb{S}'\to\mathbb{S}$ be the blowup of $\mathbb{S}$ along 
$p\in\mathbb{S}$, let $\be_p\subset\mathbb{S}'$ be the exceptional curve of 
$\sigma_p$, and let us set $\Xi':=(\sigma_p)^{-1}_*(\pi(\Xi))\subset\mathbb{S}'$. 
Then, by the definition of multiplicity, we have 
$\mult_p\left(\pi(\Xi)\right)=\left(\Xi'\cdot\be_p\right)$. 
Let us consider the fiber product
\[\xymatrix{
\Y' \ar[r]^-{\sigma_l} \ar[d]_-{\pi'}& \Y \ar[d]^-{\pi}\\
\mathbb{S}' \ar[r]_-{\sigma_p}& \mathbb{S}.
}\]
Then the morphism $\sigma_l$ is the blowup of $\Y$ along $l\subset\Y$, and the 
exceptional divisor $\E_l\subset\Y'$ of $\sigma_l$ is equal to $(\pi')^*\be_p$. 
Again by Lemma \ref{lemma:lengths} and the projection formula, we have 
\[
\operatorname{length}\left(\sO_{\Xi\cap l}\right)
=\left(\E_l\cdot(\sigma_l)^{-1}_*\Xi\right)=\left((\pi')^*\be_p
\cdot(\sigma_l)^{-1}_*\Xi\right)=\left(\be_p\cdot\Xi'\right)=\mult_p\left(\pi(\Xi)\right). 
\]
Thus we get the desired equality \eqref{equation:conic-bundle}.
\end{proof}

\begin{proposition}\label{proposition:conic-bundle}
Let $\Y$, $\mathbb{S}$ be smooth projective varieties with $\dim\Y=3$ and 
$\dim\mathbb{S}=2$, and let $\pi\colon\Y\to\mathbb{S}$ be a conic bundle with the 
relative Picard number $\rho(\Y/\mathbb{S})$ equal to $1$. Let $\Xi\subset\Y$ be 
a smooth irreducible curve such that $\pi|_\Xi\colon\Xi\to\pi(\Xi)$ is birational and 
$\pi(\Xi)\not\subset\Delta_\pi$. Let $\sigma\colon\X\to\Y$ be the blowup of $\Y$ 
along $\Xi\subset\Y$ and let $\E\subset\X$ be the exceptional divisor of $\sigma$. 
Assume that $-K_{\X}$ is nef over $\mathbb{S}$. 
Since $-K_{\X}$ is nef and big over $\mathbb{S}$, we can take the anti-canonical model 
\[\xymatrix{
\X \ar[dr]_-{\pi\circ\sigma} \ar[r]^-{\beta} & \bar{\X} \ar[d]^-{\gamma} \\
&\mathbb{S}
}\]
of $\X$ over $\mathbb{S}$. Let $\beta^+\colon \X^+\to\bar{\X}$ be 
\begin{itemize}
\item
equal to $\beta$ if $\beta$ is an isomorphism, 
\item
the flop of $\beta$ if $\beta$ is an elementary flopping contraction.
\end{itemize}
Note that $\rho(\X^+/\mathbb{S})=2$. Let $\sigma^+\colon\X^+\to\Y^+$ over 
be the elementary contraction morphism over $\mathbb{S}$ together 
with the commutative diagram  
\[\xymatrix{
\X^+ \ar[d]_-{\gamma\circ\beta^+} \ar[r]^-{\sigma^+} & \Y^+ \ar[dl]^-{\pi^+} \\
\mathbb{S}&
}\]
defined to be 
\begin{itemize}
\item
the other elementary contraction morphism over $\mathbb{S}$ other than $\sigma$ 
if $\beta$ is an isomorphism, 
\item
the unique $K_{\X^+}$-negative elementary contraction over $\mathbb{S}$ if 
$\beta$ is an elementary flopping contraction. 
\end{itemize}
Moreover, set $\chi:=(\beta^+)^{-1}\circ\beta$. As a consequence, we get 
the following diagram: 
\[\xymatrix{
\X \ar@{-->}[rr]^-{\chi} \ar[dr]^-{\beta} \ar[d]_-{\sigma}&& 
\X^+ \ar[dl]_-{\beta^+} \ar[d]^-{\sigma^+}\\
\Y \ar[dr]_-{\pi} & \bar{\X} \ar[d]_(.4){\gamma} & \Y^+ \ar[dl]^-{\pi^+} \\
 &\mathbb{S} &
}\]
We have the following: 
\begin{enumerate}
\renewcommand{\theenumi}{\arabic{enumi}}
\renewcommand{\labelenumi}{(\theenumi)}
\item\label{proposition:conic-bundle1}
The morphism $\sigma^+$ is birational and contracts the prime divisor 
$\F^+\subset\Y^+$ onto a smooth curve $\Xi^+\subset\Y^+$. 
The prime divisor $\F^+$ is equal to 
the strict transform of the prime divisor $\pi^{-1}\left(\pi(\Xi)\right)\subset\Y$ to 
$\X^+$. 
Moreover, $\Y^+$ is smooth and the morphism $\sigma^+$ is the blowup of $\Y^+$ 
along $\Xi^+\subset\Y^+$. 
\item\label{proposition:conic-bundle2}
The morphism $\pi^+\colon\Y^+\to\mathbb{S}$ is a conic bundle with 
$\rho(\Y^+/\mathbb{S})=1$. Moreover, the restriction morphism 
$\pi^+|_{\Xi^+}\colon\Xi^+\to\pi^+(\Xi^+)$ is birational, and 
$\pi(\Xi)=\pi^+(\Xi^+)$ holds. Moreover, we have $\Delta_\pi=\Delta_{\pi^+}$. 
\end{enumerate}
\end{proposition}

\begin{proof}
Since $\pi$ is a conic bundle, the morphism $\beta$ is an isomorphism over 
$\mathbb{S}\setminus\left(\pi(\Xi)\cap\Delta_\pi\right)$. 
Moreover, since $\rho(\X/\mathbb{S})=2$, we have the following: 
\begin{itemize}
\item
If $\beta$ is an isomorphism, then $-K_{\X}$ is ample over $\mathbb{S}$. Thus 
there is the unique $K_{\X}$-negative elementary contraction morphism $\sigma^+$ 
over $\mathbb{S}$ other than $\sigma$. 
\item
If $\beta$ is not an isomorphism, then $\beta$ must be an elementary flopping 
contraction. Thus $\X^+$ is smooth by Theorem \ref{thm:kollar} and 
$\rho(\X^+/\mathbb{S})=2$ holds. Hence there exists a unique $K_{\X^+}$-negative 
elementary contraction morphism $\sigma^+$ over $\mathbb{S}$ since $-K_{\X^+}$ is 
nef and big over $\mathbb{S}$. 
\end{itemize}

Set $\F^+:=\chi_*\sigma^{-1}_*\pi^{-1}\left(\pi(\Xi)\right)\subset\X^+$. 
We firstly remark that every fiber of $\pi^+\circ\sigma^+$ is of dimension $1$. 
Assume that $\sigma^+\colon\X^+\to\Y^+$ is a conic bundle. 
Since $\F^+$ is 
disjoint from a general fiber of $\sigma^+$, then there exists a line bundle 
$\sL^+$ on $\Y^+$ such that $(\sigma^+)^*\sL^+\cong\sO_{\X^+}\left(\F^+\right)$
by the contraction theorem \cite[\S 3]{Mo82}. On the other hand, let us take 
a general point $q\in\pi(\Xi)$ and let us set $l:=\pi^{-1}(q)\subset\Y$. 
The strict transform $l^+\subset\F^+\subset\X^+$ of $l\subset\Y^+$ to $\X^+$ 
satisfies that, since $\chi$ is an isomorphism around a neighborhood of 
$\sigma^{-1}_*l$, $\left(\F^+\cdot l^+\right)
=\left(\chi_*^{-1}\F^+\cdot \sigma^{-1}_* l\right)=-1$. However, since 
the birational morphism $\Y^+\to\mathbb{S}$ must be an isomorphism over 
$q\in\mathbb{S}$, the curve $l^+$ must be contracted by $\sigma^+$. 
This implies that 
$-1=\left(\F^+\cdot l^+\right)=\left((\sigma^+)^*\sL^+\cdot l^+\right)=0$, 
a contradiction. Thus the morphism $\sigma^+$ is not a conic bundle. 

By the classification of three-dimensional elementary $K_{\X^+}$-negative 
contraction morphisms \cite[\S 3]{Mo82}, the variety $\Y^+$ is smooth and the 
morphism $\sigma^+\colon \X^+\to\Y^+$ is the blowup of $\Y^+$ along 
a smooth curve $\Xi^+$ in $\Y^+$. 
Let $\E^+\subset\X^+$ be the strict transform of $\E\subset\X$ to $\X^+$. 
Since $\sigma$ is $\E$-negative and $\E$ is a nonzero effective divisor, 
the morphism $\sigma^+$ is $\E^+$-positive. 
(Indeed, if $\beta$ is not an isomorphism, then $\beta$ is $\E$-positive 
by considering curves over $\mathbb{S}$ passing through $\E$ but are not contained 
in $\E$. Thus $\beta^+$ is $\E^+$-negative. By the same reason, we can show that 
$\sigma^+$ is $\E^+$-positive. If $\beta$ is an isomorphism, then the proof 
is much easier.)
On the other hand, observe that $\E^++\F^+$ is obtained by the pullback of 
$\pi(\Xi)\subset\mathbb{S}$. Thus $\sigma^+$ is $\F^+$-negative. Therefore, 
the exceptional divisor of $\sigma^+$ is equal to $\F^+$. 
In particular, $\pi^+(\Xi^+)=\pi(\Xi)$ holds. Since $\rho(\Y^+/\mathbb{S})=1$, the 
anti-canonical divisor $-K_{\Y^+}$ is $\pi^+$-ample. Hence $\pi^+$ is a conic bundle. 
Moreover, from the construction, we have 
$\pi(\Xi)\not\subset\Delta_{\pi^+}$ and 
$\Delta_{\pi^+}\setminus\pi(\Xi)=\Delta_\pi\setminus\pi(\Xi)$. 
Thus we get the assertions \eqref{proposition:conic-bundle1} and 
\eqref{proposition:conic-bundle2}. 
\end{proof}

\begin{remark}\label{remark:conic-bundle}
Under the assumption in Proposition \ref{proposition:conic-bundle}, assume that 
$\pi$ is a $\pr^1$-bundle. In this case, the morphism $\pi^+$ is also a $\pr^1$-bundle. 
By Lemma \ref{lemma:conic-bundle}, the $\chi$ is an elementary flop if and only if 
$\pi(\Xi)$ is not a smooth curve (i.e., the curve $\pi(\Xi)$ has a multiplicity $2$ point). 
More precisely, let $\left\{o_1,\dots,o_m\right\}\subset\pi(\Xi)$ be the set of 
singular points of $\pi(\Xi)$. Then the set of flopping curves (resp., flopped curves) 
of $\chi$ is nothing but the set 
\[
\left\{\sigma^{-1}_*\left(\pi^{-1}(o_i)\right)\right\}_{1\leq i\leq m} \quad\left(\text{resp., }
\left\{(\sigma^+)^{-1}_*\left((\pi^+)^{-1}(o_i)\right)\right\}_{1\leq i\leq m}\right). 
\]
If $\pi(\Xi)$ is a smooth curve, then $\chi$ is 
the identity morphism and the above diagram is nothing but a classical elementary 
transform of $\pr^1$-bundles in the sense of Maruyama \cite{maruyama}. 
\end{remark}

\begin{proposition}\label{proposition:dP-fibration}
Let $\T$ be a smooth projective curve and let $\Y$ be a $3$-dimensional smooth 
projective variety. Let $\pi\colon\Y\to\T$ be a morphism such that 
$\pi_*\sO_{\Y}=\sO_{\T}$, $-K_{\Y}$ is $\pi$-ample and the relative Picard number 
$\rho(\Y/\T)$ is equal to $1$. Let $\Xi\subset\Y$ be a smooth irreducible curve such 
that $\pi|_{\Xi}\colon\Xi\to\T$ is surjective. Let $\sigma\colon\X\to\Y$ be the 
blowup of $\Y$ along $\Xi\subset\Y$ and let $\E\subset\X$ be the exceptional 
divisor of $\sigma$. Assume that $-K_{\X}$ is nef and big over $\T$ and the 
anti-canonical model 
\[\xymatrix{
\X \ar[dr]_-{\pi\circ\sigma} \ar[r]^-{\beta} & \bar{\X} \ar[d]^-{\gamma} \\
&\T
}\]
of $\X$ over $\T$ satisfies that $\beta$ is small. (We allow the case $\beta$ is an 
isomorphism. Note that, if $\beta$ is not an isomorphism, then 
$\beta$ is an elementary flopping contraction since $\rho(\X/\T)=2$.)
Let $\beta^+\colon\X^+\to\bar{\X}$ be 
\begin{itemize}
\item
equal to $\beta$ if $\beta$ is an isomorphism, 
\item
the flop of $\beta$ if $\beta$ is an elementary flopping contraction. 
\end{itemize}
Note that $\rho(\X^+/\T)=2$. As in Proposition \ref{proposition:conic-bundle}, 
we can uniquely get the $K_{\X^+}$-negative elementary contraction 
$\sigma^+\colon\X^+\to\Y^+$ (other than $\sigma$) over $\T$ 
together with the commutative diagram
\[\xymatrix{
\X^+ \ar[d]_-{\gamma\circ\beta^+} \ar[r]^-{\sigma^+} & \Y^+ \ar[dl]^-{\pi^+} \\
\T.&
}\]
Moreover, set $\chi:=(\beta^+)^{-1}\circ\beta$ and $\E^+:=\chi_*\E$. 
As a consequence, we get the 
following diagram:
\[\xymatrix{
\X \ar@{-->}[rr]^-{\chi} \ar[dr]^-{\beta} \ar[d]_-{\sigma}&& 
\X^+ \ar[dl]_-{\beta^+} \ar[d]^-{\sigma^+}\\
\Y \ar[dr]_-{\pi} & \bar{\X} \ar[d]_(.4){\gamma} & \Y^+ \ar[dl]^-{\pi^+} \\
 &\T &
}\]
We have the following: 
\begin{enumerate}
\renewcommand{\theenumi}{\arabic{enumi}}
\renewcommand{\labelenumi}{(\theenumi)}
\item\label{proposition:dP-fibration1}
The variety $\Y^+$ is a smooth projective variety such that $-K_{\Y^+}$ is ample 
over $\T$, $\rho(\Y^+/\T)=1$ and $\dim \Y^+\in\{2,3\}$. 
(In particular, if $\dim \Y^+=2$, then $\Y^+$ is a geometrically ruled surface over $\T$.)
\item\label{proposition:dP-fibration2}
There uniquely exists a positive rational number $u\in\Q_{>0}$ such that 
$\sigma^+$ is $(-K_{\X^+}-u \E^+)$-trivial. 
\item\label{proposition:dP-fibration3}
Let $\mathbb{S}\subset\Y$ be a general fiber of $\pi$ and set 
$\tilde{\mathbb{S}}:=\sigma_*^{-1}\mathbb{S}$. Then $\tilde{\mathbb{S}}$ is a smooth 
del Pezzo surface. Moreover, the $\Q$-divisor 
$-K_{\tilde{\mathbb{S}}}-u(\E|_{\tilde{\mathbb{S}}})$ is non-ample 
but semiample, and induces the contraction morphism 
$\sigma^+|_{\tilde{\mathbb{S}}}\colon\tilde{\mathbb{S}}\twoheadrightarrow
\sigma^+(\tilde{\mathbb{S}})$. 
\end{enumerate}
\end{proposition}

\begin{proof}
Since $\rho(\Y/\T)=1$ and $\Xi$ is horizontal with respects to $\pi$, any closed 
fiber of $\pi\circ\sigma\colon\X\to\T$ is irreducible. Thus the morphism $\sigma^+$ 
does not have a $2$-dimensional fiber. By the classification of three-dimensional 
elementary $K_{\X^+}$-negative contraction morphisms \cite[\S 3]{Mo82}, 
the variety $\Y^+$ is smooth, and either 
\begin{itemize}
\item
the dimension of $\Y^+$ is equal to $2$ and the morphism $\sigma^+$ 
is a conic bundle, or
\item
the dimension of $\Y^+$ is equal to $3$ 
and there exists a smooth irreducible curve $\Theta\subset\Y^+$ such that 
$\sigma^+$ is the blowup of $\Y^+$ along $\Theta\subset\Y^+$, and 
the curve $\Theta$ is horizontal with respects to $\pi^+$ (since $\pi^+\circ\sigma^+$ 
has no $2$-dimensional fibers).
\end{itemize}
As in the proof of Proposition \ref{proposition:conic-bundle}, the morphism $\sigma^+$ is 
$\E^+$-positive. Moreover, since $\sigma^+$ is an elementary contraction morphism 
and $\sigma^+$ is $(-K_{\X^+})$-positive, there uniquely exists $u\in\Q_{>0}$ such that 
$\sigma^+$ is $(-K_{\X^+}-u \E^+)$-trivial. 

From the assumption, $\tilde{\mathbb{S}}$ is a del Pezzo surface and $\chi$ is an 
isomorphism around a neighborhood of $\tilde{\mathbb{S}}$. Moreover, from the 
possibility of the structure of $\sigma^+$, the restriction morphism 
$\sigma^+|_{\tilde{\mathbb{S}}}\colon\tilde{\mathbb{S}}\twoheadrightarrow
\sigma^+(\tilde{\mathbb{S}})$ is not an isomorphism. 
Note that $K_{\tilde{\mathbb{S}}}+u(\E|_{\tilde{\mathbb{S}}})
=\left(\sigma|_{\tilde{\mathbb{S}}}\right)^*K_{\mathbb{S}}+(1+u)\E|_{\tilde{\mathbb{S}}}$ is 
not pseudo-effective 
on $\tilde{\mathbb{S}}$. 
Since $\rho(\Y^+/\T)=1$, this implies that $-K_{\X^+}-u \E^+$ is obtained by 
the pullback of a $\pi^+$-ample $\Q$-divisor on $\Y^+$. 
In particular, the $\Q$-divisor $-K_{\tilde{\mathbb{S}}}-u(\E|_{\tilde{\mathbb{S}}})$ is 
the pullback of an ample $\Q$-divisor on $\sigma^+(\tilde{\mathbb{S}})$
and we get the assertion. 
\end{proof}

\begin{example}\label{example:dP-fibration}
Under the assertion of Proposition \ref{proposition:dP-fibration}, assume moreover that 
there exists a rank $3$ vector bundle $\sE$ on $\T$ such that $\pi$ is obtained by 
$\Y=\pr_\T(\sE)\to\T$, and assume moreover that $\deg(\pi|_\Xi)=4$. 
Then $\mathbb{S}=\pr^2$ and $\tilde{\mathbb{S}}$ is a smooth del Pezzo surface of 
anti-canonical degree $5$. Let us set $\be_1+\cdots+\be_4:=\E|_{\tilde{\mathbb{S}}}$. 
The value $u\in\Q_{>0}$ satisfies that $-K_{\tilde{\mathbb{S}}}-u(\be_1+\cdots+\be_4)$ is 
nef but not ample. This implies that $u=\frac{1}{2}$. 
Moreover, the semiample $\Q$-divisor 
$-K_{\tilde{\mathbb{S}}}-\frac{1}{2}(\be_1+\cdots+\be_4)$ induces a conic bundle 
$\tilde{\mathbb{S}}\to\pr^1$ such that there are exactly $3$ numbers of 
reducible fibers. Therefore, $\sigma^+\colon\X^+\to\Y^+$ is a conic bundle, 
$\pi^+\colon\Y^+\to\T$ is a geometrically ruled surface, 
and $\left(\Delta_{\sigma^+}\cdot l^+\right)=3$, where $l^+\subset \Y^+$ is a general 
fiber of $\pi^+$. 
\end{example}

\section{On several smooth Fano threefolds}\label{section:fano3}

We recall the definition and basic properties of several important 
smooth Fano threefolds. 

\begin{proposition}[{\cite{MM, MM85}}]\label{proposition:2-21}
Let $Q_1$ be the $3$-dimensional smooth hyperquadric in $\pr^4$, let 
$\Gamma_1\subset Q_1$ be a twisted quartic curve (under the natural embedding 
$\Gamma_1\subset Q_1\subset \pr^4$), and let $\rho_1\colon\hat{Q}\to Q_1$ be the 
blowup of $Q_1$ along $\Gamma_1$ with the exceptional divisor $S_1\subset\hat{Q}$. 
\begin{enumerate}
\renewcommand{\theenumi}{\arabic{enumi}}
\renewcommand{\labelenumi}{(\theenumi)}
\item\label{proposition:2-211}
The threefold $\hat{Q}$ is a smooth Fano threefold in Mori--Mukai's list type 2.21 
\cite[Table 2]{MM}. 
\item\label{proposition:2-212}
The other elementary contraction of $\hat{Q}$ is also the blowup 
$\rho_2\colon\hat{Q}\to Q_2$ of the $3$-dimensional smooth hyperquadric $Q_2$ 
in $\pr^4$ along a twisted quartic curve $\Gamma_2\subset Q_2$. 
Let $S_2\subset\hat{Q}$ be the exceptional divisor of $\rho_2$. 
Then we have 
\[
\rho_2^*\sO_{Q_2}(1)\sim\rho_1^*\sO_{Q_1}(2)-S_1, \quad
\rho_1^*\sO_{Q_1}(1)\sim\rho_2^*\sO_{Q_2}(2)-S_2, \quad 
-K_{\hat{Q}}\sim\rho_1^*\sO_{Q_1}(1)+\rho_2^*\sO_{Q_2}(1), 
\]
where $\sO_{Q_i}(1)$ is the ample generator of $\Pic Q_i$. 
\end{enumerate}
\end{proposition}

\begin{proof}
This is easy and well-known. See \cite[(7.4)]{MM85} for example. 
\end{proof}

As in Proposition \ref{proposition:2-21}, in this article, whenever the 
$3$-dimensional smooth hyperquadric $Q_i$ appears, 
the ample generator of $\Pic Q_i$ is denoted by $\sO_{Q_i}(1)$. 

\begin{definition}\label{definition:2-21a}
\begin{enumerate}
\renewcommand{\theenumi}{\arabic{enumi}}
\renewcommand{\labelenumi}{(\theenumi)}
\item\label{definition:2-21a1}
A $3$-dimensional smooth projective variety $\hat{Q}$ is said to be 
\emph{a Fano threefold of type 2.21} 
if there is a smooth $3$-dimensional hyperquadric $Q_1\subset\pr^4$ and 
a twisted quartic $\Gamma_1^{Q_1}\subset Q_1\subset\pr^4$ such that 
$\hat{Q}$ is obtained 
by the blowup $\rho_1\colon\hat{Q}\to Q_1$ along $\Gamma_1^{Q_1}\subset Q_1$. 
\item\label{definition:2-21a2}
Let $\hat{Q}$ be a Fano threefold of type 2.21, and let $\rho_1\colon\hat{Q}\to Q_1$ 
and $\rho_2\colon\hat{Q}\to Q_2$ be the two distinct contractions 
as in Proposition \ref{proposition:2-21}. 
An irreducible curve $\hat{C}\subset\hat{Q}$ is said to be 
\emph{a bi-cubic curve in $\hat{Q}$} if 
$\left(\rho_i^*\sO_{Q_i}(1)\cdot \hat{C}\right)=3$ hold for $i=1$ and $2$. 
Similarly, an irreducible curve $\sC^+\subset\hat{Q}$ is said to be a 
\emph{bi-line in $\hat{Q}$} if  
$\left(\rho_i^*\sO_{Q_i}(1)\cdot \sC^+\right)=1$ hold for $i=1$ and $2$. 
\end{enumerate}
\end{definition}

\begin{lemma}\label{lemma:2-21-bi-cubic}
Let $\hat{Q}$ be a Fano threefold of type 2.21, let $\rho_1\colon\hat{Q}\to Q_1$ 
and $\rho_2\colon\hat{Q}\to Q_2$ be the two distinct elementary contractions. 
\begin{enumerate}
\renewcommand{\theenumi}{\arabic{enumi}}
\renewcommand{\labelenumi}{(\theenumi)}
\item\label{lemma:2-21-bi-cubic1}
Let $\hat{C}\subset\hat{Q}$ be a bi-cubic curve in $\hat{Q}$. 
Then both $(\rho_1)_*\hat{C}\subset Q_1\subset\pr^4$ and 
$(\rho_2)_*\hat{C}\subset Q_2\subset\pr^4$ are twisted cubic curves. 
In particular, the curve $\hat{C}$ is a smooth rational curve. 
Moreover, we have 
\[
\operatorname{length}\left(\sO_{\left((\rho_i)_*\hat{C}\right)
\cap\Gamma_i^{Q_i}}\right)=3,
\]
where $\Gamma_i^{Q_i}\subset Q_i$ is 
the center of the blowup $\rho_i$. 
\item\label{lemma:2-21-bi-cubic2}
Conversely, for any twisted cubic curve $C_1\subset Q_1$ with 
$\operatorname{length}\left(\sO_{\Gamma_1^{Q_1}\cap C_1}\right)=3$, 
the strict transform $(\rho_1)^{-1}_*C_1$ is a bi-cubic curve in $\hat{Q}$. 
\end{enumerate}
\end{lemma}

\begin{proof}
\eqref{lemma:2-21-bi-cubic1}
Since both $(\rho_1)_*\hat{C}\subset Q_1\subset\pr^4$ and 
$(\rho_2)_*\hat{C}\subset Q_2\subset\pr^4$ are curves of degree $3$ inside 
$3$-dimensional smooth hyperquadrics, they must be twisted cubic curves. 
The rest of the assertion follows from Lemma \ref{lemma:lengths}. 

\eqref{lemma:2-21-bi-cubic2}
Follows immediately from Proposition \ref{proposition:2-21} and Lemma 
\ref{lemma:lengths}. 
\end{proof}

\begin{lemma}\label{lemma:2-31}
Let $Q$ be the $3$-dimensional smooth hyperquadric and let $Z\subset Q$ be a line, 
i.e., a curve with $\left(\sO_Q(1)\cdot Z\right)=1$. 
Let $\phi\colon Y\to Q$ be the blowup of $Q$ along $Z$ and let $F\subset Y$ be the 
exceptional divisor of $\phi$. Then $Y$ is a smooth Fano threefold with 
$(-K_Y)^{\cdot 3}=46$, which is in Mori--Mukai's list type 2.31 \cite[Table 2]{MM}. 
The complete linear system $|\phi^*\sO_Q(1)-F|$ on $Y$ gives the morphism 
$\psi\colon Y\to \pr^2$. We summarize the diagram: 
\begin{equation}\label{equation:2-31}
\xymatrix{
 & F \ar@{}[r]|{\subset} & Y 
 \ar@{->}[ld]_{\phi} \ar@{->}[rd]^{\psi}
&  \\
 Z \ar@{}[r]|{\subset} & Q && 
 \pr^2.
}
\end{equation}
Moreover, there exists a closed point $p\in\pr^2$ corresponds to a coherent 
ideal sheaf $\dm_p\subset\sO_{\pr^2}$ and a non-split exact sequence 
\[
0\to\sO_{\pr^2}\to\sE\to\dm_p\to 0
\]
of coherent sheaves on $\pr^2$ such that $\sE$ is a vector bundle of rank $2$ on 
$\pr^2$ and the morphism $\psi$ is isomorphic to the projective space bundle 
$\pr_{\pr^2}(\sE)\to\pr^2$. 
\end{lemma}

\begin{proof}
See \cite[Theorem]{SW}. Note that 
\[
\Ext^1\left(\dm_p,\sO_{\pr^2}\right)\cong 
H^0\left(\pr^2, \mathscr{E}xt^1\left(\dm_p,\sO_{\pr^2}\right)\right)\cong
H^0\left(\pr^2, \mathscr{E}xt^2\left(\Bbbk(p),\sO_{\pr^2}\right)\right)\cong\Bbbk
\]
by \cite[Lemma A]{horrocks}. 
\end{proof}

\begin{definition}[{\cite{Fuj81, Isk77}}]\label{definition:dP5}
A smooth Fano threefold $V$ is said to be \emph{the del Pezzo threefold of degree $5$} 
if its Picard group $\Pic V$ is generated by an ample invertible sheaf $\sO_V(1)$, 
$-K_V\sim\sO_V(2)$ and $\left(\sO_V(1)\right)^{\cdot 3}=5$ holds. 
In this article, for a del Pezzo threefold $V_k$ of degree $5$, the ample generator 
of $\Pic V_k$ is always denoted by $\sO_{V_k}(1)$. 
(In this article, we will consider several numbers of del Pezzo threefolds of degree $5$ 
at once.) 
An irreducible curve $Z\subset V$ is said to be a \emph{line in $V$} if 
$\left(\sO_V(1)\cdot Z\right)=1$ holds; a smooth rational irreducible curve 
$\Gamma\subset V$ 
is said to be a \emph{smooth rational quintic curve in $V$} 
if $\left(\sO_V(1)\cdot \Gamma\right)=5$ holds. 
\end{definition}

\begin{thm}[{\cite{Fuj81, Isk77}}]\label{thm:dP5}
Let $V$ be a del Pezzo threefold of degree $5$. 
\begin{enumerate}
\renewcommand{\theenumi}{\arabic{enumi}}
\renewcommand{\labelenumi}{(\theenumi)}
\item\label{thm:dP51}
The isomorphism class of del Pezzo threefolds of degree $5$ is unique. 
In other words, for any del Pezzo threefold $V'$ of degree $5$, we have $V\cong V'$. 
\item\label{thm:dP52}
The complete linear system $|\sO_V(1)|$ is very ample, and gives an embedding 
$V\hookrightarrow\pr^6$. 
\item\label{thm:dP53}
For any line $Z\subset V$, its normal bundle $\sN_{Z/V}$ is isomorphic to 
either 
\[
\sO_{\pr^1}^{\oplus 2} \quad \text{or} \quad \sO_{\pr^1}(1)\oplus\sO_{\pr^1}(-1).
\] 
\item\label{thm:dP54}
Let $Z\subset V$ be a line in $V$, 
let $\phi\colon Y\to V$ be the blowup along $Z$, and let $F\subset Y$ be 
the exceptional divisor of $\phi$. Then $Y$ is a smooth Fano threefold 
in Mori--Mukai's list type 2.26 \cite[Table 2]{MM}. The other elementary contraction 
$\psi\colon Y\to Q$ is birational, the image $Q$ is the smooth hyperquadric 
in $\pr^4$, and the morphism $\psi$ is the blowup of $Q$ along a twisted cubic curve 
$C\subset Q\subset\pr^4$. Let $E\subset Y$ be the exceptional divisor of $\psi$. 
Then we have 
\[
E\sim \phi^*\sO_V(1)-2F,\quad F\sim\psi^*\sO_Q(1)-E,\quad 
-K_Y\sim\phi^*\sO_V(1)+\psi^*\sO_Q(1).
\]
We summarize the diagram: 
\begin{equation}\label{equation:fujita}
\xymatrix{
 & F \ar@{}[r]|{\subset} & Y 
\ar@{}[r]|{\supset} \ar@{->}[ld]_{\phi} \ar@{->}[rd]^{\psi}
& E & \\
 Z \ar@{}[r]|{\subset} & V \ar@{-->}[rr]_-{\psi\circ\phi^{-1}} && 
 Q \ar@{}[r]|{\supset} & C.
}
\end{equation}
\item\label{thm:dP55}
Under the assumption and notation in \eqref{thm:dP54}, 
the restriction morphism $\phi|_E\colon E\to \phi(E)$ is the normalization 
morphism of the hyperplane section $\phi(E)\subset V$ singular along $Z$, and 
\[
E\cong\begin{cases}
\pr_{\pr^1}\left(\sO\oplus\sO(1)\right) & 
\text{if } \sN_{Z/V}\cong\sO_{\pr^1}^{\oplus 2},\\
\pr_{\pr^1}\left(\sO\oplus\sO(3)\right) & 
\text{if } \sN_{Z/V}\cong\sO_{\pr^1}(1)\oplus\sO_{\pr^1}(-1). 
\end{cases}
\]
Moreover, the image $\psi(F)\subset Q$ is a smooth hyperplane section of $Q$ 
if $\sN_{Z/V}\cong\sO_{\pr^1}^{\oplus 2}$; a singular hyperplane section of $Q$ 
if $\sN_{Z/V}\cong\sO_{\pr^1}(1)\oplus\sO_{\pr^1}(-1)$. 
\item\label{thm:dP56}
Under the assumption and notation in \eqref{thm:dP54}, there is a commutative 
diagram 
\[\xymatrix{
V \ar@{-->}[r]^-{\psi\circ\phi^{-1}} \ar@{^{(}->}[d] & Q  \ar@{^{(}->}[d] \\
\pr^6 \ar@{-->}[r]& \pr^4,
}\]
where the vertical inclusions are the natural embeddings and the rational map 
$\pr^6\dashrightarrow\pr^4$ is the linear projection from the line $Z\subset\pr^6$. 
\end{enumerate}
\end{thm}

\begin{proof}
For \eqref{thm:dP51}, see \cite[I, Theorem 4.2]{Isk77} 
(see also \cite[Theorem 9.12]{Fuj81}). 
For \eqref{thm:dP52}, see \cite[I, Proposition 4.4]{Isk77} or \cite[\S 9]{Fuj81}. 
For \eqref{thm:dP53}, see \cite[I, Proposition 5.2]{Isk77} or \cite[Corollary 8.2]{Fuj81}. 
For \eqref{thm:dP54}, see \cite[Proposition 9.11]{Fuj81} (see also \cite[I, (6.5)]{Isk77}). 
For \eqref{thm:dP55}, see \cite[(7.4), (7.5) and \S 9]{Fuj81}. 
For \eqref{thm:dP56}, see \cite[Proposition 8.3]{Fuj81}. 
\end{proof}

We can consider the converse of the link \eqref{equation:fujita}.

\begin{proposition}[{\cite{Fuj81}}]\label{proposition:fujita}
Let $Q$ be the $3$-dimensional smooth hyperquadric in $\pr^4$ and let $C\subset Q$ 
be a twisted cubic curve under the natural embedding $C\subset Q\subset \pr^4$. 
Let $\psi\colon Y\to Q$ be the blowup of $Q$ along $C$ with the exceptional divisor 
$E\subset Y$. Then the variety $Y$ is a smooth Fano threefold in Mori--Mukai's list type 
2.26 \cite[Table 2]{MM}. The other elementary contraction $\phi\colon Y\to V$ 
is birational, the image $V$ is the del Pezzo threefold of degree $5$ and the morphism 
$\phi$ is the blowup of $V$ along a line $Z$ in $V$. Thus, the diagram 
\begin{equation}\label{equation:fujita-converse}
\xymatrix{
 & E \ar@{}[r]|{\subset} & Y 
\ar@{}[r]|{\supset} \ar@{->}[ld]_{\psi} \ar@{->}[rd]^{\phi}
& F & \\
 C \ar@{}[r]|{\subset} & Q \ar@{-->}[rr]_-{\phi\circ\psi^{-1}} && 
 V \ar@{}[r]|{\supset} & Z
}
\end{equation}
is nothing but the converse of the diagram \eqref{equation:fujita}. 
\end{proposition}

\begin{proof}
See \cite[(7.4) and (7.5)]{Fuj81}. 
\end{proof}

The following lemma is probably well-known. 

\begin{lemma}\label{lemma:dP5quintic}
Let $V$ be the del Pezzo threefold of degree $5$ and let $\Gamma\subset V$ 
is a smooth rational quintic curve in $V$. Then, under the half-anti-canonical embedding 
$V\hookrightarrow\pr^6$, the linear span $\langle\Gamma\rangle$ 
of the curve $\Gamma\subset\pr^6$ is a hyperplane in $\pr^6$. In particular, 
the curve $\Gamma$ is a twisted quintic curve in $\pr^6$. 
\end{lemma}

\begin{proof}
Assume that $\langle\Gamma\rangle$ is not a hyperplane. 
Since $\left(\sO_V(1)\right)^{\cdot 3}=5$ and any member of $|\sO_V(1)|$ is 
irreducible and reduced, the curve $\Gamma$ must be a complete intersection of $V$ 
and codimension $2$ linear subspace of $\pr^6$. Then the arithmetic genus of 
$\Gamma$ must be equal to $1$. This leads to a contradiction. 
\end{proof}

\begin{definition}\label{definition:dP6}
\emph{The del Pezzo threefold of degree $6$ and rank $2$} is defined to be 
the variety $U$ defined by the effective divisor 
$U\subset\pr^2_{x_{10}x_{11}x_{12}}\times\pr^2_{x_{20}x_{21}x_{22}}$ defined by the 
equation $\sum_{i=0}^2x_{1i}x_{2i}=0$. The $U$ is a smooth Fano threefold. Let 
$\rho_1\colon U\to\pr^2_{x_{10}x_{11}x_{12}}$ 
and $\rho_2\colon U\to\pr^2_{x_{20}x_{21}x_{22}}$ be the projections, and set 
$\sO_U(a_1,a_2):=\rho_1^*\sO_{\pr^2}(a_1)\otimes\rho_2^*\sO_{\pr^2}(a_2)$
for any $a_1,a_2\in\Z$. Then $\sO_U(-K_U)\cong\sO_U(2,2)$ and 
$\left(\sO_U(1,1)\right)^{\cdot 3}=6$. We have $\Pic U=\Z\left[\sO_U(1,0)\right]\oplus
\Z\left[\sO_U(0,1)\right]$. 
In particular, the Picard rank of $U$ is equal to $2$. 
Moreover, for each $i\in\{1,2\}$, the morphism 
$\rho_i$ is isomorphic to the projective space bundle $\pr_{\pr^2}(T_{\pr^2})\to\pr^2$
of the tangent bundle $T_{\pr^2}$ of $\pr^2$. An irreducible curve $\Gamma\subset U$ 
is said to be a \emph{bi-quintic curve} in $U$ if $\left(\sO_U(1,0)\cdot \Gamma\right)
=\left(\sO_U(0,1)\cdot \Gamma\right)=5$ holds. 
Similarly, an irreducible curve $\sC^+\subset U$ is said to be a 
\emph{bi-line in $U$} if  
$\left(\sO_U(1,0)\cdot \sC^+\right)=\left(\sO_U(0,1)\cdot\sC^+\right)=1$ holds. 
\end{definition}

\begin{lemma}\label{lemma:dP6-multiplicity}
Let $U$ be the del Pezzo threefold of degree $6$ and rank $2$, let $l\subset U$ 
be a closed fiber of the morphism $\rho_2\colon U\to\pr^2$, and let us take 
any $m\in\Z_{\geq 1}$. Take any irreducible $S\in|\sO_U(1,m)|$. Then the 
multiplicity $\mult_l S$ of $S$ along $l$ satisfies that $\mult_l S\leq m$. 
\end{lemma}

\begin{proof}
Let $\sigma\colon\tilde{U}\to U$ be the blowup of $U$ along $l$, let 
$\E\subset \tilde{U}$ be the exceptional divisor of $\sigma$, and set 
$\F_U:=\rho_1^{-1}\left(\rho_1(l)\right)$, $\F:=\sigma^{-1}_*\F_U$. A general fiber 
$l'$ of $\rho_1\circ\sigma|_\F\colon \F\to \rho_1(l)$ satisfies that 
$\left(\E\cdot l'\right)=1$. Let us set $m':=\mult_l S$, and set 
$\tilde{S}:=\sigma^{-1}_*S\sim\sigma^*\sO_U(1,m)-m'\E$. Since 
$\F\not\subset\tilde{S}$, we have 
$0\leq\left(\tilde{S}\cdot l'\right)=m-m'$. 
The we get the assertion. 
\end{proof}

\begin{lemma}\label{lemma:dP6}
Let $U$ be the del Pezzo threefold of degree $6$ and rank $2$. Take any irreducible 
$S\in|\sO_U(1,1)|$. 
\begin{enumerate}
\renewcommand{\theenumi}{\arabic{enumi}}
\renewcommand{\labelenumi}{(\theenumi)}
\item\label{lemma:dP61}
The surface $S$ is normal, and has at worst du Val singularities. 
The canonical divisor 
$-K_S$ of $S$ is Cartier with $(-K_S)^{\cdot 2}=6$. 
\item\label{lemma:dP62}
Let $\nu\colon\tilde{S}\to S$ be the minimal resolution of $S$. 
Consider the dual graph of the configuration of all negative curves on $\tilde{S}$. 
We represent $(-1)$-curves by $\bullet$ and $(-2)$-curves by $\circ$. 
Then one of the following cases occurs: 
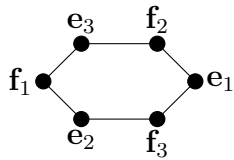
\begin{enumerate}
\renewcommand{\theenumii}{\roman{enumii}}
\renewcommand{\labelenumii}{(\theenumii)}
\item\label{lemma:dP621}
\[\begin{tikzpicture}
\draw (2,1) node[right]{$\be_1$} --(1.5, 1.5) node[above]{$\Bf_2$}-- 
(.5, 1.5) node[above]{$\be_3$}--(0,1) node[left]{$\Bf_1$} -- 
(.5,.5) node[below]{$\be_2$}--(1.5,.5) node[below]{$\Bf_3$}--cycle; 
\filldraw [fill=black] (2,1) circle [radius=.1];
\filldraw [fill=black] (1.5,1.5) circle [radius=.1];
\filldraw [fill=black] (.5,1.5) circle [radius=.1];
\filldraw [fill=black] (0,1) circle [radius=.1];
\filldraw [fill=black] (.5,.5) circle [radius=.1];
\filldraw [fill=black] (1.5,.5) circle [radius=.1];
\end{tikzpicture}\]
(in other words, $S$ is smooth), 
\item
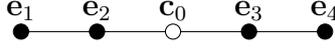
\label{lemma:dP622}
\[\begin{tikzpicture}
\draw (0,0) node[above]{$\be_1$}--(1,0) node[above]{$\be_2$}--
(2,0) node[above]{$\Bc_0$}--(3,0) node[above]{$\be_3$}--(4,0) node[above]{$\be_4$}; 
\filldraw [fill=black] (0,0) circle [radius=.1];
\filldraw [fill=black] (1,0) circle [radius=.1];
\filldraw [fill=white] (2,0) circle [radius=.1];
\filldraw [fill=black] (3,0) circle [radius=.1];
\filldraw [fill=black] (4,0) circle [radius=.1];
\end{tikzpicture}\]
(in this case, $S$ has exactly one $A_1$ singular point), 
\item
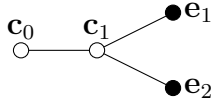
\label{lemma:dP623}
\[\begin{tikzpicture}
\draw (0,1) node[above]{$\Bc_0$}--(1,1) node[above]{$\Bc_1$}--(2,1.5) 
node[right]{$\be_1$};
\draw (1,1) --(2,.5) node[right]{$\be_2$};
\filldraw [fill=white] (0,1) circle [radius=.1];
\filldraw [fill=white] (1,1) circle [radius=.1];
\filldraw [fill=black] (2,1.5) circle [radius=.1];
\filldraw [fill=black] (2,.5) circle [radius=.1];
\end{tikzpicture}\]
(in this case, $S$ has exactly one $A_2$ singular point).
\end{enumerate}
\end{enumerate}
\end{lemma}

\begin{proof}
\eqref{lemma:dP61}
Note that $S$ is Gorenstein, irreducible, reduced 
and projective surface such that $\omega_S^\vee\cong\sO_U(1,1)|_S$ is ample. 
Assume that $S$ is not normal. 
Let $l\subset S$ be the conductor for the normalization of $S$. 
By \cite[Lemma 3.9, (3.34), Lemma 3.35]{Mo82}, $l\subset S$ is an irreducible and 
reduced curve satisfying $\left(\omega_S^\vee\cdot l\right)=1$. 
Thus $l$ is a fiber of $\rho_1$ or $\rho_2$. On the other hand, since 
$l$ is the conductor, we have $\mult_l S\geq 2$. This contradicts with 
Lemma \ref{lemma:dP6-multiplicity}. 
Thus $S$ must be normal. 
Moreover, since $S$ is inside $U$, for any point $p\in S$, the number of curves 
$l''$ in $S$ with $p\in l''$ and $\left(\omega_S^\vee\cdot l''\right)=1$ is at most two. 
By \cite[Proofs of Proposition 1.2 and Theorem 2.2]{HW}, 
our $S$ has only du Val singularities. 
(Indeed, $S$ cannot be the cone of an elliptic curve from the above observation.)
Thus we get the assertion \eqref{lemma:dP61}. 

\eqref{lemma:dP62}
Obviously, $S$ is not isomorphic to $\pr^2$. Moreover, the morphisms 
$\rho_1|_S$, $\rho_2|_S$ are birational. Thus the morphisms 
$\rho_1|_S$, $\rho_2|_S$ give nontrivial pairwise distinct birational contraction 
morphisms. In particular, the Picard number $\rho(S)$ of $S$ is bigger than $1$. 
By \cite[Proposition 8.3]{CT}, the possibilities for the dual graphs are as follows: 
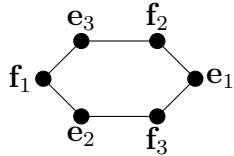
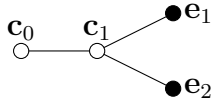
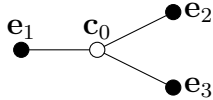
\begin{enumerate}
\renewcommand{\theenumi}{\roman{enumi}}
\renewcommand{\labelenumi}{(\theenumi)}
\item\label{lemma:dP6-proof1}
\[\begin{tikzpicture}
\draw (2,1) node[right]{$\be_1$} --(1.5, 1.5) node[above]{$\Bf_2$}-- 
(.5, 1.5) node[above]{$\be_3$}--(0,1) node[left]{$\Bf_1$} -- 
(.5,.5) node[below]{$\be_2$}--(1.5,.5) node[below]{$\Bf_3$}--cycle; 
\filldraw [fill=black] (2,1) circle [radius=.1];
\filldraw [fill=black] (1.5,1.5) circle [radius=.1];
\filldraw [fill=black] (.5,1.5) circle [radius=.1];
\filldraw [fill=black] (0,1) circle [radius=.1];
\filldraw [fill=black] (.5,.5) circle [radius=.1];
\filldraw [fill=black] (1.5,.5) circle [radius=.1];
\end{tikzpicture}\]
\item\label{lemma:dP6-proof2}
\[\begin{tikzpicture}
\draw (0,0) node[above]{$\be_1$}--(1,0) node[above]{$\be_2$}--
(2,0) node[above]{$\Bc_0$}--(3,0) node[above]{$\be_3$}--(4,0) node[above]{$\be_4$}; 
\filldraw [fill=black] (0,0) circle [radius=.1];
\filldraw [fill=black] (1,0) circle [radius=.1];
\filldraw [fill=white] (2,0) circle [radius=.1];
\filldraw [fill=black] (3,0) circle [radius=.1];
\filldraw [fill=black] (4,0) circle [radius=.1];
\end{tikzpicture}\]
\item\label{lemma:dP6-proof3}
\[\begin{tikzpicture}
\draw (0,1) node[above]{$\Bc_0$}--(1,1) node[above]{$\Bc_1$}--(2,1.5) 
node[right]{$\be_1$};
\draw (1,1) --(2,.5) node[right]{$\be_2$};
\filldraw [fill=white] (0,1) circle [radius=.1];
\filldraw [fill=white] (1,1) circle [radius=.1];
\filldraw [fill=black] (2,1.5) circle [radius=.1];
\filldraw [fill=black] (2,.5) circle [radius=.1];
\end{tikzpicture}\]
\item\label{lemma:dP6-proof4}
\[\begin{tikzpicture}
\draw (0,1) node[above]{$\be_1$}--(1,1) node[above]{$\Bc_0$}--(2,1.5) 
node[right]{$\be_2$};
\draw (1,1) --(2,.5) node[right]{$\be_3$};
\filldraw [fill=black] (0,1) circle [radius=.1];
\filldraw [fill=white] (1,1) circle [radius=.1];
\filldraw [fill=black] (2,1.5) circle [radius=.1];
\filldraw [fill=black] (2,.5) circle [radius=.1];
\end{tikzpicture}\]
\item\label{lemma:dP6-proof5}
\[\begin{tikzpicture}
\draw (0,0) node[above]{$\Bc_0$}--(1,0) node[above]{$\be_1$}--
(2,0) node[above]{$\Bc_1$}--(3,0) node[above]{$\be_2$}; 
\filldraw [fill=white] (0,0) circle [radius=.1];
\filldraw [fill=black] (1,0) circle [radius=.1];
\filldraw [fill=white] (2,0) circle [radius=.1];
\filldraw [fill=black] (3,0) circle [radius=.1];
\end{tikzpicture}\]
\item\label{lemma:dP6-proof6}
\[\begin{tikzpicture}
\draw (0,0) node[above]{$\Bc_0$}--(1,0) node[above]{$\Bc_1$}--
(2,0) node[above]{$\be_1$}--(3,0) node[above]{$\Bc_2$}; 
\filldraw [fill=white] (0,0) circle [radius=.1];
\filldraw [fill=white] (1,0) circle [radius=.1];
\filldraw [fill=black] (2,0) circle [radius=.1];
\filldraw [fill=white] (3,0) circle [radius=.1];
\end{tikzpicture}\]
\end{enumerate}

Since $\rho(S)>1$, the case \eqref{lemma:dP6-proof6} does not occur. 
Assume the case \eqref{lemma:dP6-proof5}. Then $\rho(S)=2$ holds. 
In particular, any elementary contraction morphism from $S$ must be onto $\pr^2$. 
However, if we contract the curves $\Bc_0$, $\Bc_1$, $\be_2$, then we get a 
contraction morphism $S\to\pr(1,1,2)$. This leads to a contradiction. 
Finally, assume the case \eqref{lemma:dP6-proof4}. Take any elementary contraction 
morphism from $S$. Then the image must be isomorphic to $\pr^1\times\pr^1$. 
(For example, if we contract the curves $\be_1$ and $\Bc_0$, then we get a birational 
morphism $\tilde{S}\to\pr^1\times\pr^1$.)
Thus $S$ does not admit a contraction morphism onto $\pr^2$, a contradiction. 
Therefore we get the assertion \eqref{lemma:dP62}. 
\end{proof}

\section{Prime Fano threefolds and lines}\label{section:prime3}

We recall several fundamental properties of prime Fano threefolds and see 
basic properties of totally disjoint pairs of lines in them.

\begin{definition}\label{definition:X22}
\begin{enumerate}
\renewcommand{\theenumi}{\arabic{enumi}}
\renewcommand{\labelenumi}{(\theenumi)}
\item\label{definition:X221}
A \emph{prime Fano threefold $X$} is a smooth Fano threefold 
with $\Pic X=\Z\left[\sO_X(-K_X)\right]$. 
The \emph{degree} of $X$ is defined to be the anti-canonical volume $(-K_X)^{\cdot 3}$, 
and the \emph{genus} of $X$ is defined to be the value $1+\frac{1}{2}(-K_X)^{\cdot 3}$. 
A \emph{line} $Z$ (resp., a \emph{conic} $\mathcal{C}$) in $X$ is an 
irreducible curve on $X$ such that $\left(-K_X\cdot Z\right)=1$ (resp., 
$\left(-K_X\cdot \mathcal{C}\right)=2$) holds. 
The \emph{Hilbert scheme of lines in $X$} is denoted by $\Sigma(X)$. 
More precisely, the scheme $\Sigma(X)$ is the Hilbert scheme of $X$ whose 
Hilbert polynomial is equal to $t+1$ with respects to the anti-canonical divisor $-K_X$. 
\item\label{definition:X222}
Let $X$ be a prime Fano threefold and let $Z_1$, $Z_2$ be lines in 
$X$. We say that $Z_1$, $Z_2$ is \emph{a totally disjoint} (resp., 
\emph{an absolutely disjoint}) \emph{pair of lines in $X$} if 
$Z_1\cap Z_2=\emptyset$ and there is no line $Z$ in $X$ satisfying both 
$Z\cap Z_1\neq\emptyset$ and $Z\cap Z_2\neq\emptyset$ (resp., 
if $Z_1\cap Z_2=\emptyset$ and there are no lines 
$Z'_1$, $Z'_2$ in $X$ satisfying $Z_1\cap Z'_1\neq \emptyset$, 
$Z_2\cap Z'_2\neq \emptyset$ and $Z'_1\cap Z'_2\neq\emptyset$). 
Clearly, if $Z_1$, $Z_2$ is an absolutely disjoint pair of lines in $X$, then 
the pair is totally disjoint. 
\end{enumerate}
\end{definition}

\begin{thm}[{see \cite{Isk79, Isk89}}]\label{thm:iskovskikh}
Let $X$ be a prime Fano threefold of genus $g\geq 5$. 
\begin{enumerate}
\renewcommand{\theenumi}{\arabic{enumi}}
\renewcommand{\labelenumi}{(\theenumi)}
\item\label{thm:iskovskikh1}
The value $g$ satisfies that $g\in\Z$, $g\leq 12$ and $g\neq 11$. 
\item\label{thm:iskovskikh2}
The complete linear system $|-K_X|$ is very ample and 
gives an embedding $X\hookrightarrow\pr^{g+1}$. Moreover, under the anti-canonical 
embedding $X\subset\pr^{g+1}$, the variety $X$ is the scheme-theoretic intersection of 
hyperquadrics in $\pr^{g+1}$ containing $X$. 
\item\label{thm:iskovskikh3}
The Hilbert scheme $\Sigma(X)$ of lines in $X$ is non-empty and purely 
$1$-dimensional projective scheme. (We only use the result for the case $g\geq 8$ 
in this article.)
\item\label{thm:iskovskikh4}
A closed point $[Z]\in\Sigma(X)$ is a smooth point if and only if 
$\sN_{Z/X}\cong\sO_{\pr^1}\oplus\sO_{\pr^1}(-1)$; a singular point if and only if 
$\sN_{Z/X}\cong\sO_{\pr^1}(1)\oplus\sO_{\pr^1}(-2)$. 
\end{enumerate}
\end{thm}

\begin{proof}
For \eqref{thm:iskovskikh1}, see \cite[IV, Theorem 3.1]{Isk79}. See also 
\cite[Theorem 8.3]{Pr25}. 
For \eqref{thm:iskovskikh2}, see 
\cite[I, Propositions 4.9 and 6.1, II, Theorem 3.4 and IV, Proposition 1.3]{Isk79}. 
See also \cite[Theorem 6.7]{Pr25}. 
For \eqref{thm:iskovskikh3}, see \cite{shokurov, reid} 
and \cite[III, Proposition 2.1]{Isk79}, see also 
\cite[Theorem 0.2]{Tak89} and \cite[Theorem 8.2]{Pr25} for the case $g\geq 8$. 
For \eqref{thm:iskovskikh4}, see \cite[\S 1, Lemma 1 and Proposition 1]{Isk89}. 
See also \cite[Proposition 8.9]{Pr25}. 
\end{proof}

\begin{thm}[{see \cite{Isk89}}]\label{thm:double-projection-from-line}
Let $X$ be a prime Fano threefold of genus $g\geq 7$, let $Z\subset X$ be 
a line in $X$, let $\sigma\colon X'\to X$ be the blowup of $X$ along $Z$, and let 
$F'\subset X'$ be the exceptional divisor of $\sigma$. 
\begin{enumerate}
\renewcommand{\theenumi}{\arabic{enumi}}
\renewcommand{\labelenumi}{(\theenumi)}
\item\label{thm:double-projection-from-line1}
The variety $X'$ is not a smooth Fano threefold but a smooth weak Fano threefold. 
\item\label{thm:double-projection-from-line2}
The anti-canonical model $\beta\colon X'\to \bar{X}$ of $X'$ is small. 
The set of $\beta$-exceptional irreducible curves in $X'$ is equal to the set consisting 
of the strict transforms $\sigma^{-1}_*Z'$ of lines $Z'\subset X$ apart from $Z$ with 
$Z\cap Z'\neq\emptyset$, and the $(-3)$-curve in $F'$ 
(if $\sN_{Z/X}\cong\sO_{\pr^1}(1)\oplus\sO_{\pr^1}(-2)$). 
\item\label{thm:double-projection-from-line3}
Let $\beta^+\colon X^+\to \bar{X}$ be the flop of the elementary flopping contraction 
$\beta$, let us set $\chi_{X' X^+}:=(\beta^+)^{-1}\circ \beta\colon X'\dashrightarrow 
X^+$, and let $\tau\colon X^+\to \X$ be the unique 
$K_{X^+}$-negative elementary contraction morphism. 
(We note that $X^+$ is a smooth weak Fano threefold by Proposition 
\ref{proposition:MDS}.) We summarize the diagram: 
\begin{equation}\label{equation:iskovskikh}
\xymatrix{
 & F' \ar@{}[r]|{\subset}  & X'  \ar@{-->}[rr]^-{\chi_{X'X^+}} 
\ar@{->}[ld]_-{\sigma} \ar@{->}[rd]_-{\beta} 
& & X^+ \ar@{->}[rd]^-{\tau} \ar@{->}[ld]^-{\beta^+}&  \\
 Z \ar@{}[r]|{\subset} & X && \bar{X} && \X 
}
\end{equation}
Then $\tau$ is determined by the ample model of the 
semiample divisor 
\[
(\chi_{X' X^+})_*\left(\sigma^*(-K_X)-2F'\right)
\]
on $X^+$. Moreover, we have the following: 
\begin{enumerate}
\renewcommand{\theenumii}{\roman{enumii}}
\renewcommand{\labelenumii}{(\theenumii)}
\item\label{thm:double-projection-from-line31}
If $g=12$, then $\tau$ is birational, $\X$ is equal to the del Pezzo threefold $V$ 
of degree $5$, and $\tau$ is the blowup of $V$ along a smooth rational quintic curve 
$\Gamma\subset V$ with the exceptional divisor $S^+\subset X^+$. Moreover, 
we have 
\[
S^+\sim(\chi_{X' X^+})_*\left(\sigma^*(-K_X)-3F'\right).
\]
\item\label{thm:double-projection-from-line32}
If $g=10$, then $\tau$ is birational, $\X$ is equal to the $3$-dimensional smooth 
hyperquadric $Q$, and $\tau$ is the blowup of $Q$ along a smooth curve 
$\Gamma\subset Q$ of genus $2$ and $\left(\sO_Q(1)\cdot \Gamma\right)=7$ 
with the exceptional divisor $S^+\subset X^+$. Moreover, 
we have 
\[
S^+\sim(\chi_{X' X^+})_*\left(\sigma^*(-2K_X)-5F'\right).
\]
\item\label{thm:double-projection-from-line33}
If $g=9$, then $\tau$ is birational, $\X$ is equal to the $3$-dimensional projective 
space $\pr^3$, and $\tau$ is the blowup of $\pr^3$ along a smooth non-hyperelliptic 
curve $\Gamma\subset \pr^3$ of genus $3$ and 
$\left(\sO_{\pr^3}(1)\cdot \Gamma\right)=7$ 
with the exceptional divisor $S^+\subset X^+$. Moreover, 
we have 
\[
S^+\sim(\chi_{X' X^+})_*\left(\sigma^*(-3K_X)-7F'\right).
\]
\item\label{thm:double-projection-from-line34}
If $g=8$, then $\X=\pr^2$ and $\tau$ is a conic bundle with 
$\Delta_\tau\in\left|\sO_{\pr^2}(5)\right|$. 
\item\label{thm:double-projection-from-line35}
If $g=7$, then $\X=\pr^1$ and a general fiber of $\tau$ is a smooth del Pezzo 
surface of anti-canonical degree $5$. 
\end{enumerate}
\end{enumerate}
\end{thm}

\begin{proof}
See \cite[\S 1, Proposition 3 and Main Theorem]{Isk89}. See also 
\cite[Lemma 8.13, Corollary 8.16 and Theorem 8.3]{Pr25}. 
\end{proof}

If $g\geq 9$, then we have the converse of the above link \eqref{equation:iskovskikh}. 

\begin{thm}[{\cite{Pr92, KPS18, Pr25}}]\label{thm:double-projection-from-line-converse}
Let $\X$ be a smooth Fano threefold and let $\Gamma\subset\X$ be an 
irreducible curve in $\X$ such that one of: 
\begin{enumerate}
\renewcommand{\theenumi}{\roman{enumi}}
\renewcommand{\labelenumi}{(\theenumi)}
\item\label{thm:double-projection-from-line-converse-1}
$\X$ is the del Pezzo threefold $V$ of degree $5$ and $\Gamma$ is a smooth 
rational quintic curve, 
\item\label{thm:double-projection-from-line-converse-2}
$\X$ is the smooth $3$-dimensional hyperquadric $Q$ and $\Gamma$ is a smooth 
curve of genus $2$ and $\left(\sO_Q(1)\cdot \Gamma\right)=7$, or 
\item\label{thm:double-projection-from-line-converse-3}
$\X=\pr^3$ and $\Gamma$ is a smooth non-hyperelliptic curve of genus $3$ and 
$\left(\sO_{\pr^3}(1)\cdot\Gamma\right)=7$. 
\end{enumerate}
Let $\tau\colon X^+\to \X$ be the blowup of $\X$ along $\Gamma$ with the 
exceptional divisor $S^+\subset X^+$. Then $X^+$ is not a smooth Fano threefold 
but a smooth weak Fano threefold such that the anti-canonical model 
$\beta^+\colon X^+\to\bar{X}$ of $X^+$ is small. Let $\beta\colon X'\to\bar{X}$ be 
the flop of the elementary flopping contraction $\beta^+$ and let us set 
$\chi_{X^+ X'}:=\beta^{-1}\circ\beta^+\colon X^+\dashrightarrow X'$. 
(We note that $X'$ is a smooth weak Fano threefold by Proposition 
\ref{proposition:MDS}.) The unique $K_{X'}$-negative elementary contraction morphism 
$\sigma\colon X'\to X$ is the blowup of a prime Fano threefold $X$ along 
a line $Z\subset X$ with the exceptional divisor $F'\subset X'$. We summarize the 
diagram: 
\begin{equation}\label{equation:iskovskikh-converse}
\xymatrix{
& S^+ \ar@{}[r]|{\subset}  & X^+  \ar@{-->}[rr]^{\chi_{X^+ X'}} 
\ar@{->}[ld]_-{\tau} \ar@{->}[rd]_-{\beta^+} 
& & X' \ar@{->}[rd]^-{\sigma} \ar@{->}[ld]^-{\beta} \ar@{}[r]|{\supset} & F' & \\
\Gamma \ar@{}[r]|{\subset} & \X && \bar{X} && X \ar@{}[r]|{\supset}& Z
}
\end{equation}
Moreover, we have the following: 
\begin{enumerate}
\renewcommand{\theenumi}{\arabic{enumi}}
\renewcommand{\labelenumi}{(\theenumi)}
\item\label{thm:double-projection-from-line-converse1}
If $(\X, \Gamma)$ satisfies \eqref{thm:double-projection-from-line-converse-1}, 
then the genus of $X$ is equal to $12$ and the diagram 
\eqref{equation:iskovskikh-converse} is the converse of the diagram 
\eqref{equation:iskovskikh} for the case $g=12$. 
Moreover, the set of $(\beta^+)$-exceptional irreducible curves in $X^+$ is equal to 
the set consisting of the strict transforms $\tau^{-1}_*Z''$ of lines $Z''\subset V$ with 
$\operatorname{length}\left(\sO_{\Gamma\cap Z''}\right)=2$. 
The prime divisor $F^V:=\tau_*(\chi_{X^+ X'})^{-1}_*F'$ in $V$ is 
the hyperplane section in $V$ containing $\Gamma$. The divisor $F^V$ is normal 
if and only if $\sN_{Z/X}\cong\sO_{\pr^1}\oplus\sO_{\pr^1}(-1)$ holds. 
\item\label{thm:double-projection-from-line-converse2}
If $(\X, \Gamma)$ satisfies \eqref{thm:double-projection-from-line-converse-2}, 
then the genus of $X$ is equal to $10$ and the diagram 
\eqref{equation:iskovskikh-converse} is the converse of the diagram 
\eqref{equation:iskovskikh} for the case $g=10$. 
\item\label{thm:double-projection-from-line-converse3}
If $(\X, \Gamma)$ satisfies \eqref{thm:double-projection-from-line-converse-3}, 
then the genus of $X$ is equal to $9$ and the diagram 
\eqref{equation:iskovskikh-converse} is the converse of the diagram 
\eqref{equation:iskovskikh} for the case $g=9$. 
\end{enumerate}
\end{thm}

\begin{proof}
The possibility of flopping curves in 
\eqref{thm:double-projection-from-line-converse1}
is proved in \cite[Lemma 5.2.5]{KPS18}. 
For the normality of $F^V$, see \cite[Theorem 1.2 and Proposition 2.1]{Pr92} or 
\cite[Theorem 18]{DFK}. 
The other assertions can be found in \cite[Theorem 11.5 and Proposition 11.9]{Pr25}. 
\end{proof}

We will prove the following proposition in \S \ref{section:flop}. 

\begin{proposition}\label{proposition:disjoint-lines}
Under the assumption in Theorem 
\ref{thm:double-projection-from-line}, assume moreover that $g=12$ (resp., $g=10$). 
Then the number of the flopping curve of $\beta$ is at most $3$ (resp., 
at most $4$), 
and all two of flopping curves of $\beta$ in $X'$ are disjoint. 
\end{proposition}

We will use the following Takeuchi's link only in \S \ref{section:parameter}. 

\begin{thm}[{\cite{Tak89, KP18}}]\label{thm:takeuchi}
\begin{enumerate}
\renewcommand{\theenumi}{\arabic{enumi}}
\renewcommand{\labelenumi}{(\theenumi)}
\item\label{thm:takeuchi1}
Let $X$ be a prime Fano threefold of genus $12$ and let $\sC\subset X$ 
be a conic in $X$. 
Let $\sigma_{\sC}\colon \tilde{Q}^+\to X$ be the blowup of $X$ along $\sC$ 
with the exceptional divisor $\mathbb{S}\subset \tilde{Q}^+$. 
Then the variety $\tilde{Q}^+$ is not a smooth Fano threefold 
but a smooth weak Fano threefold, 
its anti-canonical model $\beta^+\colon\tilde{Q}^+\to \bar{Q}$ is small. 
The flop $\beta\colon\tilde{Q}\to \bar{Q}$ of $\beta^+$ satisfies that, the other 
elementary $K_{\tilde{Q}}$-negative contraction morphism $\sigma_\Lambda\colon
\tilde{Q}\to Q$ is birational, its image $Q$ is the $3$-dimensional smooth hyperquadric 
in $\pr^4$, and the morphism $\sigma_\Lambda$ is the blowup of $Q$ along 
a smooth rational sextic curve $\Lambda\subset Q$ satisfying that 
the restriction homomorphism 
\begin{equation}\label{equation:QN}
H^0\left(\pr^4,\sO_{\pr^4}(2)\right)\to H^0\left(\Lambda, \sO_{\pr^4}(2)|_\Lambda\right)
\end{equation}
is surjective. Let $\E\subset\tilde{Q}$ be the exceptional divisor of 
$\sigma_\Lambda$, and let us set 
$\E^+:=(\beta^{-1}\circ\beta^+)^{-1}_*\E\subset\tilde{Q}^+$ and 
$\mathbb{S}':=(\beta^{-1}\circ\beta^+)_*\mathbb{S}\subset\tilde{Q}$. Then we have 
\[
\mathbb{S}'\sim \sigma_\Lambda^*\sO_Q(2)-\E, \quad
\mathbb{E}^+\sim\sigma_{\sC}^*\sO_X(-2K_X)-5\mathbb{S}.
\]
We summarize the diagram: 
\begin{equation}\label{equation:takeuchi}
\xymatrix{
& \mathbb{S} \ar@{}[r]|{\subset}  & \tilde{Q}^+  \ar@{-->}[rr]^-{\beta^{-1}\circ\beta^+} 
\ar@{->}[ld]_-{\sigma_{\sC}} \ar@{->}[rd]_-{\beta^+} 
& & \tilde{Q} \ar@{->}[rd]^-{\sigma_{\Lambda}} 
\ar@{->}[ld]^-{\beta} \ar@{}[r]|{\supset} 
& \E & \\
\sC \ar@{}[r]|{\subset} & X && \bar{Q} && 
Q \ar@{}[r]|{\supset}& \Lambda
}
\end{equation}
\item\label{thm:takeuchi2}
Let $Q$ be the $3$-dimensional smooth hyperquadric in $\pr^4$ and let $\Lambda
\subset Q$ be a smooth sextic rational curve such that the restriction homomorphism 
\eqref{equation:QN} is surjective. Let $\sigma_\Lambda\colon\tilde{Q}\to Q$ be 
the blowup of $Q$ along $\Lambda$. Then we can run the Sarkisov link of $\tilde{Q}$ 
which is the inverse of the link \eqref{equation:takeuchi}. 
\end{enumerate}
\end{thm}

\begin{proof}
See \cite[(2.8.2), (2.13.2)]{Tak89} and \cite[Theorems 2.2 and 2.6]{KP18}. 
\end{proof}

We can observe that the locus 
\[
\left\{\left(\left[Z_1\right],\left[Z_2\right]\right)\in \Sigma(X) \times \Sigma(X)
\mid Z_1, Z_2\text{ are totally disjoint}\right\}
\subset \Sigma(X) \times \Sigma(X)
\]
is dense whenever $X$ is a prime Fano threefold of genus $g\geq 7$. 
In fact, we can see more: 

\begin{lemma}\label{lemma:totally-disjoint}
Let $X$ be a prime Fano threefold of genus $g\geq 7$. 
Take any line $\left[Z_1\right]\in \Sigma(X)$ 
in $X$. Then there exists a $0$-dimensional closed subset $\D_{Z_1}
\subset\Sigma(X)$ such that any line $\left[Z_2\right]\in\Sigma(X)
\setminus\D_{Z_1}$ satisfies that 
$Z_1$, $Z_2$ is an absolutely disjoint pair of lines in $X$. 
(In particular, totally disjoint.) 
\end{lemma}

\begin{proof}
By Theorem \ref{thm:double-projection-from-line}, 
there are at most finite numbers of lines $Z_{11},\dots,Z_{1l}$ in $X$ 
intersecting with $Z_1$. Moreover, for any $1\leq i\leq l$, there are at most 
finite numbers of lines $Z_{1i1},\dots,Z_{1im_i}$ intersecting with $Z_{1i}$; 
for any $1\leq i\leq l$ and $1\leq j\leq m_i$, there are at most finite numbers of lines 
$Z_{1ij1},\dots,Z_{1ijn_{ij}}$ intersecting with $Z_{1ij}$. 
If we take
\[
\D_{Z_1}:=\left\{\left[Z_1\right]\right\}
\cup\left\{\left[Z_{1i}\right]\right\}_{1\leq i\leq l}\cup
\left\{\left[Z_{1ij}\right]\right\}_{\substack{1\leq i\leq l\\ 1\leq j\leq m_i}}\cup
\left\{\left[Z_{1ijk}\right]\right\}_{\substack{1\leq i\leq l\\ 1\leq j\leq m_i\\ 
1\leq k\leq n_{ij}}}, 
\]
then the $\D_{Z_1}$ satisfies the desired property. 
\end{proof}

We consider the blowups of prime Fano threefolds along totally disjoint pairs of lines 
throughout the article. 

\begin{proposition}\label{proposition:totally-blowup}
Let $X$ be a prime Fano threefold of genus $g\geq 5$ and let $Z_1,Z_2\subset X$ 
be a totally disjoint pair of lines in $X$. 
Let $\sigma\colon X_0\to X$ be the blowup of $X$ along $Z_1\cup Z_2$, and let $F_1$ 
and $F_2$ be the exceptional divisors of $\sigma$ with $\sigma(F_1)=Z_1$ and 
$\sigma(F_2)=Z_2$. 
\begin{enumerate}
\renewcommand{\theenumi}{\arabic{enumi}}
\renewcommand{\labelenumi}{(\theenumi)}
\item\label{proposition:totally-blowup1}
Under the anti-canonical embedding $X\subset\pr^{g+1}$, the linear span 
$\langle Z_1\cup Z_2\rangle$ of the skew lines $Z_1\cup Z_2$ satisfies 
\[
\langle Z_1\cup Z_2\rangle\cap X=Z_1\cup Z_2
\]
scheme-theoretically. In particular, the anti-canonical divisor $-K_{X_0}$ is 
globally generated. 
\item\label{proposition:totally-blowup2}
Assume that $g\geq 7$. Then $X_0$ is not a smooth Fano threefold but a smooth 
weak Fano threefold with $(-K_{X_0})^{\cdot 3}=2g-10$. 
Moreover, the anti-canonical model 
$\alpha\colon X_0\to \bar{X}_0$ satisfies that the Picard number $\rho(\bar{X}_0)$ 
of $\bar{X}_0$ is equal to $1$.
\item\label{proposition:totally-blowup3}
Assume either $g\geq 9$, or $g=8$ and 
$h^0\left(X_0,\sigma^*(-K_X)-2F_1-2F_2\right)=0$. 
Then the above $\alpha$ is a small morphism. 
\end{enumerate}
\end{proposition}

\begin{proof}
\eqref{proposition:totally-blowup1}
Set $P:=\langle Z_1\cup Z_2\rangle\cong\pr^3$. The subscheme $P\cap X\subset P$ 
contains $Z_1\cup Z_2$ and is defined by the intersections of some quadric surfaces 
in $P$ by Theorem \ref{thm:iskovskikh} \eqref{thm:iskovskikh2}. 
Let $\tilde{P}\to P$ be the blowup of $P$ along $Z_1\cup Z_2$. 
Observe that there is an isomorphism 
\[
\tilde{P}\cong\pr_{\pr^1\times\pr^1}\left(\sO(1,0)\oplus\sO(0,1)\right)
\]
and there is a natural bijection between the set of quadrics in $P$ containing $Z_1\cup 
Z_2$ and the complete linear system $|\sO_{\pr^1\times\pr^1}(1,1)|$; 
for any quadric $Q$ in $P$ with $Z_1\cup Z_2\subset Q$ satisfies that its strict 
transform in $\tilde{P}$ is given by the pullback of an member of 
$|\sO_{\pr^1\times\pr^1}(1,1)|$. 
Thus, if $P\cap X\subset P$ is not scheme-theoretically equal to $Z_1\cup Z_2$, 
then $P\cap X$ must contains the image of a fiber of the $\pr^1$-bundle 
$\pr_{\pr^1\times\pr^1}\left(\sO(1,0)\oplus\sO(0,1)\right)\to \pr^1\times\pr^1$. 
The image is a line in $X$ intersecting with both $Z_1$ and $Z_2$. 
%Take any quadric surface $Q\subset P$ containing $Z_1\cup Z_2$. Then 
%there is an isomorphism $Q\cong\pr^1\times\pr^1$ such that 
%$Z_1$, $Z_2\in|\sO(1,0)|$ under the isomorphism. For any other quadric 
%surface $Q'\subset P$ containing $Z_1\cup Z_2$, the intersection $Q'|_Q$ can be 
%written as $Z_1+Z_2+Z'_1+Z'_2$, where $Z'_1$, $Z'_2\in|\sO(0,1)|$. 
%Obviously, any irreducible 
%component of $P\cap X$ is of dimension $\leq 1$. 
%Assume that $Z_1\cup Z_2\subsetneq P\cap X$. Then there must be a line 
%$Z'\subset X$ with $Z'\cap Z_1\neq \emptyset$ and $Z'\cap Z_2\neq \emptyset$. 
It contradicts with the total disconnectedness. Thus we get the equality
$P\cap X=Z_1\cup Z_2$. 
Let us consider the 
blowup of $\pr^{g+1}$ along the linear subspace $P=\langle Z_1\cup Z_2\rangle$. 
As above, $X_0$ is obtained as the strict transform of the blowup 
and 
$-K_{X_0}=\sigma^*(-K_X)-F_1-F_2$ is the restriction of a base point free divisor. 

\eqref{proposition:totally-blowup2}
We can check that $(F_i)^{\cdot 3}=1$, $(\sigma^*(-K_X)\cdot F_i^{\cdot 2})=-1$ 
and $(\sigma^*(-K_X)^{\cdot 2}\cdot F_i)=0$, which implies that 
$(-K_{X_0})^{\cdot 3}=2g-10>0$. By Theorem \ref{thm:double-projection-from-line} and 
$Z_1$, $Z_2$ are totally disjoint, there exist curves $l_1, l_2\subset X_0$ such that 
both are contracted by $\alpha$ but $(F_i\cdot l_j)=-\delta_{ij}$ holds. 
Thus the morphism $\alpha$ is not elementary. This implies that $\rho(\bar{X}_0)=1$. 

\eqref{proposition:totally-blowup3}
Assume that $\alpha$ contracts a prime divisor $D$ on $X_0$. We can write 
\[
D\sim a\sigma^*(-K_X)-b_1 F_1-b_2 F_2
\]
for some $a$, $b_1$, $b_2\in\mathbb{Z}$. Note that $((-K_{X_0})^{\cdot 2}\cdot D)=0$. 
This immediately implies that $D\neq F_1$ and $D\neq F_2$. In particular, we have 
$\sigma_*D\neq 0$, i.e., $a\geq 1$. Moreover, we have 
\[
0\leq \big{(} \sigma^*(-K_X)\cdot F_i\cdot D \big{)} = b_i. 
\]
By Theorem \ref{thm:iskovskikh}, 
we must have $b_i\leq r a$ for $i=1$, $2$, where 
\[
r:=\begin{cases}
3 & \text{if }g=12,\\
\frac{5}{2} & \text{if }g=10, \\
\frac{7}{3} & \text{if }g=9, \\
2 & \text{if }g=8. 
\end{cases}
\] 
Observe that
\[
0=\big{(} (\sigma^*(-K_X)-F_1-F_2)^{\cdot 2}\cdot D \big{)}
=(2g-4)a-3 b_1-3 b_2\geq 0, 
\]
and the equality holds only when $g=9$ and $b_1=b_2=\frac{7}{3}a$, 
or, $g=8$ and $b_1=b_2=2a$. 

Assume that $g=9$. Then $\sigma_*D\subset X$ must be $\sigma_*D\in|-3K_X|$ and 
of multiplicity $7$ along both $Z_1$ and $Z_2$. 
Consider the Sarkisov link 
\[\xymatrix{
 & F'_1 \ar@{}[r]|{\subset}  & X'_1  \ar@{-->}[rr]^-{\chi_1} 
\ar@{->}[ld]_-{\sigma_1} \ar@{->}[rd]_-{\beta_1} 
& & X^+_1 \ar@{->}[rd]^-{\tau_1} \ar@{->}[ld]^-{\beta^+_1}& \\
 Z_1 \ar@{}[r]|{\subset} & X && \bar{X}_1 && \pr^3 
}\]
from the blowup of $X$ along $Z_1$ as in \eqref{equation:iskovskikh} for the case $g=10$. 
Since $Z_1$ and $Z_2$ are totally disjoint, the map $\chi_1$ is an isomorphism around 
the neighborhood of curve $(\sigma_1)^{-1}_*Z_2$. Thus the image of 
$(\sigma_1)^{-1}_*Z_2$ to $\pr^3$ is a line since 
$\left(\sigma_1^*(-K_X)\cdot(\sigma_1)^{-1}_*Z_2\right)=1$ and $\left(F'_1\cdot
(\sigma_1)^{-1}_*Z_2\right)=0$. On the other hand, the center of the 
prime divisor $D$ on $\pr^3$ is a smooth curve of genus $3$. Since $\sigma_*D$ 
contains $Z_2$, this leads to a contradiction. 

Assume the case $g=8$. 
Let us consider the Sarkisov link 
\begin{equation}\label{equation:g8-proof}
\xymatrix{
 & F'_1 \ar@{}[r]|{\subset}  & X'_1  \ar@{-->}[rr]^-{\chi_1} 
\ar@{->}[ld]_-{\sigma_1} \ar@{->}[rd]_-{\beta_1} 
& & X^+_1 \ar@{->}[rd]^-{\tau_1} \ar@{->}[ld]^-{\beta^+_1}& \\
 Z_1 \ar@{}[r]|{\subset} & X && \bar{X}_1 && \pr^2 
}
\end{equation}
from the blowup of $X$ along $Z_1$ as in \eqref{equation:iskovskikh} for the case $g=8$. 
Since $Z_1$ and $Z_2$ are totally disjoint, the map $\chi_1$ is an isomorphism around 
the neighborhood of curve $(\sigma_1)^{-1}_*Z_2$. Thus the image of 
$(\sigma_1)^{-1}_*Z_2$ to $\pr^2$ is a line $l_2\subset\pr^2$ since 
$\left(\sigma_1^*(-K_X)\cdot(\sigma_1)^{-1}_*Z_2\right)=1$ and $\left(F'_1\cdot
(\sigma_1)^{-1}_*Z_2\right)=0$. Since the struct transform of the prime divisor 
$D$ to $X_1^+$ must be obtained by the pullback of an effective divisor in $\pr^2$, 
it must be equal to $\tau_1^*l_2$. Thus, $\sigma_*D\in|-K_X|$ and of multiplicity $2$ 
along $Z_1$. By considering the same argument from $Z_2$, we can also show that 
the multiplicity of $\sigma_*D$ along $Z_2$ is also $2$. 
However, this implies that $h^0\left(X_0,\sigma^*(-K_X)-2F_1-2F_2\right)\geq 1$, 
a contradiction. 

As a consequence, the morphism $\alpha$ is a small morphism. 
\end{proof}

\begin{remark}\label{remark:g8}
\begin{enumerate}
\renewcommand{\theenumi}{\arabic{enumi}}
\renewcommand{\labelenumi}{(\theenumi)}
\item\label{remark:g81}
The assumption ``$g \geq 7$'' in Proposition \ref{proposition:totally-blowup} 
\eqref{proposition:totally-blowup2} can be weakened to ``$g\geq 6$''. 
See \cite[Lemma 8.13]{Pr25} and \cite[Proposition 4.10]{MM83}.
\item\label{remark:g82}
Under the assumption in Proposition \ref{proposition:totally-blowup}, assume 
moreover that 
$g=8$ and 
\[
h^0\left(X_0,\sigma^*(-K_X)-2F_1-2F_2\right)\neq 0.
\] 
Then, from the proof of Proposition \ref{proposition:totally-blowup} 
\eqref{proposition:totally-blowup3}, we have 
$h^0\left(X_0,\sigma^*(-K_X)-2F_1-2F_2\right)=1$. Let us set 
$\{D\}:=|\sigma^*(-K_X)-2F_1-2F_2|$. The image $D_1\subset X_1^+$ of $D$ to $X_1^+$ 
under the link \eqref{equation:g8-proof} is singular along $(\chi_1)_*(\sigma_1)^{-1}_*Z_2$
and is equal to $\tau_1^*l_2$. Therefore, we have $l_2\subset\Delta_{\tau_1}$ and 
the curve $(\chi_1)_*(\sigma_1)^{-1}_*Z_2$ is the closure of the set of singular points of 
the fibers of $\tau_1$ at general points in $l_2$. 
\end{enumerate}
\end{remark}

\part{Construction of links}\label{part:construction-links}

\section{Links from the blowups of prime Fano threefolds along two 
lines}\label{section:go}

The purpose of this section is to show Theorem \ref{thm:main}. 
In \S \ref{section:go}, 
we assume that $X$ is a prime Fano threefold of genus $g\in\{12,10,9\}$ and 
$Z_1$, $Z_2\subset X$ is a totally disjoint pair of lines in $X$. 
For $\{i,j\}=\{1,2\}$, consider 
the Sarkisov link
\[
\xymatrix{
 & F'_i \ar@{}[r]|{\subset}  & X_i  \ar@{-->}[rr]^{\chi_i} 
\ar@{->}[ld]_-{\sigma_i} \ar@{->}[rd]_-{\beta_i} 
& & X^+_i \ar@{->}[rd]^-{\tau_i} \ar@{->}[ld]^-{\beta_i^+} \ar@{}[r]|{\supset} & S^+_i & \\
 Z_i \ar@{}[r]|{\subset} & X && \bar{X}_i && \X_i \ar@{}[r]|{\supset}& \Gamma_i
}
\]
from the blowup of $X$ along $Z_i$ as in \eqref{equation:iskovskikh}. 
In particular, the morphism $\tau_i$ is the blowup of $\X_i$ along a smooth curve 
$\Gamma_i\subset\X_i$ with the exceptional divisor $S_i^+\subset X_i^+$. We set 
\[
\X_i=:\begin{cases}
V_i & \text{if }g=12, \\
Q_i & \text{if }g=10, \\
\pr^3_i & \text{if }g=9.
\end{cases}
\]
Moreover, let $\sigma\colon X_0\to X$ be the blowup along $Z_1\cup Z_2$ and 
let $\sigma'_j\colon X_0\to X_i$ be the naturally induced morphism. The exceptional 
divisor $F_j\subset X_0$ of $\sigma'_j$ is the strict transform of $F'_j\subset X_j$. 
Let $S_1$, $S_2\subset X_0$ be the strict transform of $S_1^+\subset X_1^+$, 
$S_2^+\subset X_2^+$, respectively. 
By Proposition \ref{proposition:totally-blowup}, $X_0$ is a smooth weak Fano threefold 
of $(-K_{X_0})^{\cdot 3}=2g-10$ and the anti-canonical model $\alpha\colon X_0\to
\bar{X}_0$ of $X_0$ satisfies that $\alpha$ is small and  $\rho(\bar{X}_0)=1$. 

\begin{lemma}\label{lemma:1st-flop}
The flop $\chi_i$ is an isomorphism around a neighborhood of 
$Z_j^{X_i}:=(\sigma_i)^{-1}_*Z_j\subset X_i$. In particular, we 
can define the strict transform $Z_j^{X^+_i}\subset X^+_i$ 
(resp., $Z_j^{\bar{X}_i}\subset \bar{X}_i$, $Z_j^{\X_i}\subset \X_i$) of 
$Z_j\subset X$ to $X^+_i$ (resp., to $\bar{X}_i$, to $\X_i$). Moreover, 
the curve $Z_j^{\X_i}$ in $\X_i$ is a line (i.e., $\left(\sO_{\X_i}(1)\cdot Z_j^{\X_j}\right)=1$
holds) with 
\[
\operatorname{length}\left(\sO_{\Gamma_i\cap Z_j^{\X_i}}\right)=\begin{cases}
1 & \text{if }g=12, \\
2 & \text{if }g=10, \\
3 & \text{if }g=9.
\end{cases}
\]
Moreover, the strict transform $F_i^{\X_i}\subset\X_i$ of $F_i\subset X_0$ satisfies that 
$Z_j\not\subset F_i^{\X_i}$. 
(We remark that, if $g=12$, then the prime divisor $F_i^{V_i}\subset V_i$ is the 
intersection $V_i\cap\langle\Gamma_i\rangle\subset V_i$ of the linear span 
$\langle \Gamma_i\rangle\subset\pr^6$ by Theorem 
\ref{thm:double-projection-from-line-converse} 
\eqref{thm:double-projection-from-line-converse1}.)
\end{lemma}

\begin{proof}
From our assumption of the total disconnectedness, the rational map 
$X\dashrightarrow X_i^+$ is an isomorphism 
around a neighborhood of $Z_j$. Thus we have 
$\left(-K_{X_i^+}\cdot Z_j^{X_i^+}\right)=\left(-K_{X_i}\cdot Z_j^{X_i}\right)=1$ and 
\[
\left(S_i^+\cdot Z_j^{X_i^+}\right)=\left((\chi_i)_*^{-1}S_i^+\cdot Z_j^{X_i}\right)=
\begin{cases}
1 & \text{if }g=12, \\
2 & \text{if }g=10, \\
3 & \text{if }g=9
\end{cases}
\] 
by Theorem \ref{thm:double-projection-from-line}. 
In particular, 
we get $\left(\sO_{\X_i}(1)\cdot Z_j^{\X_i}\right)=1$. 
The remaining assertion follows from Lemma \ref{lemma:lengths}. 
\end{proof}

Let $\phi_i^+\colon W_i\to X_i^+$ (resp., $\gamma_{i1}\colon \bar{X}'_i\to \bar{X}_i$) 
be the blowup along $Z_j^{X_i^+}\subset X_i^+$ (resp., along $Z_j^{\bar{X}_i}\subset
\bar{X}_i$). Then the small elementary contractions $\beta_i$ and $\beta_i^+$ naturally 
lift to the small elementary flopping contraction morphisms 
$\beta_{i1}\colon X_0\to\bar{X}'_i$ and $\beta_{i1}^+\colon W_i\to\bar{X}'_i$. 
Set $\chi_{i1}:=(\beta_{i1}^+)^{-1}\circ\beta_{i1}\colon X_0\dashrightarrow W_i$. 
The rational map $\chi_{i1}$ is a small $\Q$-factorial modification of $X_0$ and is 
not an isomorphism. This immediately implies that $\beta_{i1}^+$ is the flop of 
$\beta_{i1}$. (Indeed, under the identification $\ND(X_0)=\ND(W_i)$, the nef cones 
$\Nef(X_0)$ and $\Nef(W_i)$ share the facet $\Nef(\bar{X}'_i)$, and the interiors of 
those cones are disjoint.)
We get the following commutative diagram: 
\begin{equation}\label{equation:1st-flop}
\xymatrix{
& X_0 \ar@{-->}[rr]^-{\chi_{i1}} \ar[dr]_-{\beta_{i1}} \ar[dd]^(.6){\sigma'_j} 
\ar[dddl]_-{\sigma}& & 
W_i \ar[dd]^(.6){\phi^+_i} 
\ar[dl]^-{\beta_{i1}^+} & \\
&&\bar{X}'_i \ar[dd]^(.25){\gamma_{i1}} &&\\
& X_i \ar@{-->}[rr]^(.3){\chi_i} \ar[dl]^-{\sigma_i} \ar[dr]_-{\beta_i} & & 
X_i^+ \ar[dl]^-{\beta_i^+} \ar[dr]_-{\tau_i} & \\
X&&\bar{X}_i&&\X_i
}
\end{equation}

Let $\phi_i\colon Y_i\to\X_i$ be the blowup of $\X_i$ along $Z_j^{\X_i}\subset \X_i$ 
with the exceptional divisor $F_j^{Y_i}\subset Y_i$, and let us set 
$\Gamma_i^{Y_i}:=(\phi_i)^{-1}_*\Gamma_i\subset Y_i$. Moreover, let 
$\tau_i^+\colon W_i^+\to Y_i$ be the blowup of $Y_i$ along $\Gamma_i^{Y_i}\subset Y_i$ 
with the exceptional divisor $S_i^{W_i^+}\subset W_i^+$. As in Example 
\ref{example:3flop}, we get an elementary flop 
\begin{equation}\label{equation:2nd-flop}
\xymatrix{
W_i \ar@{-->}[rr]^-{\chi_{i2}} \ar[dr]^-{\beta_{i2}} \ar[d]_-{\phi_i^+}&& 
W_i^+ \ar[dl]_-{\beta_{i2}^+} \ar[d]^-{\tau_i^+}\\
X_i^+ \ar[dr]_-{\tau_i} & \X'_i \ar[d]_(.4){\gamma_{i2}} & Y_i \ar[dl]^-{\phi_i} \\
 &\X_i &
}
\end{equation}
We set 
\[
\X'_i=:\begin{cases}
V'_i & \text{if } g=12, \\
Q'_i & \text{if } g=10, \\
P'_i & \text{if } g=9. \\
\end{cases}
\]
In any case, $Y_i$ is a smooth Fano threefold. Let $\psi_i\colon Y_i\to\Y_i$ 
be the elementary contraction morphism other than $\phi_i$. 
The morphism $\psi_i$ can be described as follows: 
\begin{enumerate}
\renewcommand{\theenumi}{\arabic{enumi}}
\renewcommand{\labelenumi}{(\theenumi)}
\item\label{construction-2nd1}
If $g=12$, then we set $Q_i:=\Y_i$; the $3$-dimensional smooth hyperquadric 
(see Theorem \ref{thm:dP5}). Let $C_i\subset Q_i$ be the center of the 
blowup $\psi_i$ and let $E^{Y_i}\subset Y_i$ be the exceptional divisor of $\psi_i$. 
We summarize the diagram: 
\[\xymatrix{
 & F_j^{Y_i} \ar@{}[r]|{\subset} & Y_i 
\ar@{}[r]|{\supset} \ar@{->}[ld]_{\phi_i} \ar@{->}[rd]^{\psi_i}
& E^{Y_i} & \\
 Z_j^{V_i} \ar@{}[r]|{\subset} & V_i && Q_i \ar@{}[r]|{\supset} & C_i
}
\]
\item\label{construction-2nd2}
If $g=10$, then we set $\pr^2_i:=\Y_i$; the projective plane (see Lemma 
\ref{lemma:2-31}). The morphism $\psi_i\colon Y_i\to \pr^2_i$ is a $\pr^1$-bundle 
on $\pr^2_i$. 
\item\label{construction-2nd3}
If $g=9$, then we set $\pr^1_i:=\Y_i$; the projective line. The morphism 
$\psi_i\colon Y_i\to\pr^1_i$ is isomorphic to the projective space bundle 
$\pr_{\pr^1_i}\left(\sO^{\oplus 2}\oplus\sO(1)\right)\to\pr^1_i$. 
\end{enumerate}

\begin{lemma}\label{lemma:3rd-flop}
\begin{enumerate}
\renewcommand{\theenumi}{\arabic{enumi}}
\renewcommand{\labelenumi}{(\theenumi)}
\item\label{lemma:3rd-flop1}
Assume that $g=12$. We can define the strict transform $\Gamma_i^{Q_i}\subset Q_i$ 
of $\Gamma_i\subset V_i$ to $Q_i$, and 
the curve $\Gamma_i^{Q_i}\subset Q_i$ is a twisted quartic curve with 
$\operatorname{length}\left(\sO_{C_i\cap
\Gamma_i^{Q_i}}\right)=3$. 
\item\label{lemma:3rd-flop2}
Assume that $g=10$. Then the curve $\Gamma_i^{Y_i}\subset Y_i$ is a smooth curve 
of genus $2$ with 
\[
\left(\phi_i^*\sO_{Q_i}(1)\cdot \Gamma_i^{Y_i}\right)=7, \quad
\left(\psi_i^*\sO_{\pr^2_i}(1)\cdot \Gamma_i^{Y_i}\right)=5. 
\]
The restriction $\psi_i|_{\Gamma_i^{Y_i}}\colon\Gamma_i^{Y_i}
\to\psi_i\left(\Gamma_i^{Y_i}\right)$ is birational onto a plane quintic curve. 
Moreover, for any point $p_i\in\psi_i\left(\Gamma_i^{Y_i}\right)$, we have 
$\mult_{p_i}\left(\psi_i\left(\Gamma_i^{Y_i}\right)\right)\leq 2$. 
\item\label{lemma:3rd-flop3}
Assume that $g=9$. Then the curve  $\Gamma_i^{Y_i}\subset Y_i$ is a smooth 
non-hyperelliptic curve of genus $3$ with 
$\left(\psi_i^*\sO_{\pr^1_i}(1)\cdot \Gamma_i^{Y_i}\right)=4$. Moreover, 
the composition $\psi_i\circ\tau_i^+\colon W_i^+\to\pr^1_i$ satisfies that, 
the anti-canonical divisor $-K_{W_i^+}$ is $(\psi_i\circ\tau_i^+)$-nef and 
$(\psi_i\circ\tau_i^+)$-big, and the anti-canonical model of $W_i^+$ 
over $\pr^1_i$ is small. 
\end{enumerate}
\end{lemma}

\begin{proof}
\eqref{lemma:3rd-flop1}
By Theorem \ref{thm:dP5} \eqref{thm:dP54} and Lemma \ref{lemma:1st-flop}, we have 
$(E^{Y_i}\cdot \Gamma_i^{Y_i})=3$ and 
$(\psi_i^*\sO_{Q_i}(1)\cdot \Gamma_i^{Y_i})=4$. Thus 
$\Gamma_i^{Q_i}$ is a rational curve of degree $4$ in $Q_i$ with 
$\Gamma_i^{Q_i}\neq C_i$. In particular, we have $\Gamma_i^{Y_i}\not\subset E^{Y_i}$. 
Note that, by Theorem \ref{thm:dP5} \eqref{thm:dP56}, there is a commutative diagram 
\[\xymatrix{
V_i \ar@{-->}[r]^-{\psi_i\circ\phi_i^{-1}} \ar@{^{(}->}[d] & Q_i  \ar@{^{(}->}[d] \\
\pr^6 \ar@{-->}[r]& \pr^4,
}\]
where $\pr^6\dashrightarrow\pr^4$ is the linear projection from the line 
$Z_j^{V_i}\subset\pr^6$. 
By Lemma \ref{lemma:1st-flop}, the linear span $\langle\Gamma_i\rangle$ of 
$\Gamma_i\subset\pr^6$ is of codimension $1$ and does not contain 
$Z_j^{V_i}$. This implies that the linear span $\langle\Gamma_i^{Q_i}
\rangle\subset\pr^4$ of $\Gamma_i^{Q_i}$ is equal to $\pr^4$. 
Thus, together with Lemma \ref{lemma:lengths}, the curve is a (smooth) twisted quartic 
curve with $\operatorname{length}\left(\sO_{C_i\cap
\Gamma_i^{Q_i}}\right)=3$. 

\eqref{lemma:3rd-flop2}
By Proposition \ref{proposition:MDS}, the variety $W_i^+$ is a smooth weak Fano 
threefold with the small anti-canonical model. 
Thus $\psi_i|_{\Gamma_i^{Y_i}}$ maps birationally onto a plane curve with the 
multiplicity condition by Lemma \ref{lemma:conic-bundle}. 
The other assertions follow immediately from Lemma \ref{lemma:1st-flop}. 

\eqref{lemma:3rd-flop3}
Again by By Proposition \ref{proposition:MDS}, the variety $W_i^+$ is a smooth weak Fano 
threefold with the small anti-canonical model. 
Thus the assertion follows from Lemma \ref{lemma:1st-flop}. 
\end{proof}

From now on, we will construct the following elementary flop: 
\begin{equation}\label{equation:3rd-flop}
\xymatrix{
W_i^+ \ar@{-->}[rr]^-{\chi_{i3}} \ar[dr]^-{\beta_{i3}} \ar[d]_-{\tau_i^+}&& 
X_0^{+,i} \ar[dl]_-{\beta_{i3}^+} \ar[d]^-{\tau^i}\\
Y_i \ar[dr]_-{\psi_i} & \Y''_i \ar[d]_(.4){\gamma_{i3}} & \Y^i \ar[dl]^-{\rho_i} \\
 &\Y_i &
}
\end{equation}
We set 
\[
\Y''_i=:\begin{cases}
Q''_i & \text{if } g=12, \\
P''_i & \text{if } g=10 \text{ or }9. \\
\end{cases}
\]
(We will also see later that $(\Y:=)\Y^1=\Y^2$, $(X_0^+:=)X_0^{+,1}=X_0^{+,2}$ and 
$(\tau:=)\tau^1=\tau^2$.)

\begin{enumerate}
\renewcommand{\theenumi}{\arabic{enumi}}
\renewcommand{\labelenumi}{(\theenumi)}
\item\label{lemma:3rd-flop1}
Assume that $g=12$. By Lemma \ref{lemma:3rd-flop} \eqref{lemma:3rd-flop1}, 
the curve $\Gamma_i^{Q_i}\subset Q_i$ is a smooth twisted quartic curve. Let 
$\rho_i\colon \Y^i\to Q_i(=\Y_i)$ be the blowup of $Q_i$ along $\Gamma_i^{Q_i}$, 
and let $\tau^i\colon X_0^{+,i}\to \Y^i$ be the blowup of $\Y^i$ along 
$\hat{C}^i:=(\rho_i)^{-1}_*C_i\subset \Y^i$. The diagram \eqref{equation:3rd-flop} 
is taken to be as in Example \ref{example:3flop}. Note that the rational map $\chi_{i3}$ 
is an elementary flop since 
$\operatorname{length}\left(\sO_{C_i\cap\Gamma_i^{Q_i}}\right)=3$. 
\item\label{lemma:3rd-flop2}
Assume that $g=10$. By Lemma \ref{lemma:3rd-flop} \eqref{lemma:3rd-flop2}, 
we can take the diagram \eqref{equation:3rd-flop} from the blowup of the $\pr^1$-bundle 
$Y_i\to \pr^2_i(=\Y_i)$ along $\Gamma_i^{Y_i}\subset Y_i$ 
as in Proposition \ref{proposition:conic-bundle}. 
(By Lemma \ref{lemma:conic-bundle}, the anti-canonical divisor $-K_{W_i^+}$ of $W_i^+$ 
is $(\psi_i\circ\tau_i^+)$-nef and $(\psi_i\circ\tau_i^+)$-big.)
In particular, the morphism $\rho_i\colon\Y^i\to\pr^2_i$ is a $\pr^1$-bundle and 
the morphism $\tau^i\colon X_0^{+,i}\to \Y^i$ is the blowup of $\Y^i$ along 
a smooth curve $\Gamma^i\subset\Y^i$ of genus $2$. Moreover, the restriction 
$\rho_i|_{\Gamma^i}\colon\Gamma^i\twoheadrightarrow 
\psi_i\left(\Gamma_i^{Y_i}\right)$ is birational onto the plane curve 
$\psi_i\left(\Gamma_i^{Y_i}\right)$. Note that the arithmetic genus of the plane 
quintic curve $\psi_i\left(\Gamma_i^{Y_i}\right)$ is $6$. Thus the curve is a singular 
curve since the genus of $\Gamma^i$ is $2$. Therefore, the rational map 
$\chi_{i3}$ is an elementary flop by Remark \ref{remark:conic-bundle}. 
\item\label{lemma:3rd-flop3}
Assume that $g=9$. By Lemma \ref{lemma:3rd-flop} \eqref{lemma:3rd-flop3}, 
we can take the diagram \eqref{equation:3rd-flop} as in Proposition 
\ref{proposition:dP-fibration}. As in Example \ref{example:dP-fibration}, 
the morphism $\rho_i\colon \Y^i\to\pr^1_i(=\Y_i)$ is a Hirzebruch surface, 
the morphism $\tau^i$ is 
$(\chi_{i3})_*\left(-K_{W_i^+}-\frac{1}{2}S_i^{W_i^+}\right)$-trivial, and the morphism 
$\tau^i\colon X_0^{+,i}\to \Y^i$ is a conic bundle with 
$\left(\Delta_{\tau^i}\cdot\rho_i^*\sO_{\pr^1_i}(1)\right)=3$. 
A priori, the rational map $\chi_{i3}$ may be an isomorphism. We will see later that 
$\chi_{i3}$ is in fact an elementary flop. 
\end{enumerate}

The following proposition is important \S \ref{section:go}. 

\begin{proposition}\label{proposition:merge}
\begin{enumerate}
\renewcommand{\theenumi}{\arabic{enumi}}
\renewcommand{\labelenumi}{(\theenumi)}
\item\label{proposition:merge1}
The rational maps 
\begin{eqnarray*}
\tau^1\circ\chi_{13}\circ\chi_{12}\circ\chi_{11}\colon X_0 &\dashrightarrow& \Y^1, \\
\tau^2\circ\chi_{23}\circ\chi_{22}\circ\chi_{21}\colon X_0 &\dashrightarrow& \Y^2
\end{eqnarray*}
give the same rational map as rational contraction maps from $X_0$. 
(We set $\Y:=\Y^1=\Y^2$.)
\item\label{proposition:merge2}
The birational maps 
\begin{eqnarray*}
\chi_{13}\circ\chi_{12}\circ\chi_{11}\colon X_0 &\dashrightarrow& X_0^{+,1}, \\
\chi_{23}\circ\chi_{22}\circ\chi_{21}\colon X_0 &\dashrightarrow& X_0^{+,2}
\end{eqnarray*}
give the same birational map as birational contraction maps from $X_0$. 
(We set $X_0^+:=X_0^{+,1}=X_0^{+,2}$.)
In particular, the morphisms $\tau^1$ and $\tau^2$ give the same contraction 
morphism $\tau\colon X_0^+\to \Y$. 
\item\label{proposition:merge3}
We have the following: 
\begin{enumerate}
\renewcommand{\theenumii}{\roman{enumii}}
\renewcommand{\labelenumii}{(\theenumii)}
\item\label{proposition:merge31}
Assume that $g=12$. Then $\hat{Q}:=\Y$ is a Fano threefold of type 2.21, and 
$\rho_i\colon\hat{Q}\to Q_i(=\Y_i)$ $(i=1,2)$ are pairwise distinct elementary 
contraction morphisms. Moreover, we have $(\hat{C}:=)\hat{C}^1=\hat{C}^2
\subset\hat{Q}$, and the curve $\hat{C}$ is a bi-cubic curve in $\hat{Q}$. 
We remark that $\tau\colon X_0^+\to\hat{Q}$ is the blowup along $\hat{C}$. 
Let $E^+\subset X_0^+$ be the exceptional divisor of the blowup $\tau$. 
\item\label{proposition:merge32}
Assume that $g=10$. Then $U:=\Y$ is the del Pezzo threefold of degree $6$ and 
rank $2$, and $\rho_i\colon U\to \pr^2_i(=\Y_i)$ $(i=1,2)$ are pairwise distinct 
elementary contraction morphisms. Moreover, we have $(\Gamma:=)
\Gamma^1=\Gamma^2$, and the curve $\Gamma\subset U$ is a smooth 
bi-quintic curve of genus $2$ satisfying that $\rho_i(\Gamma)\subset\pr^2_i$ is a 
plane quintic curve with $\mult_{p_i}\left(\rho_i(\Gamma)\right)\leq 2$ for any 
$p_i\in\rho_i(\Gamma)$ and for any $i=1,2$. 
We remark that $\tau\colon X_0^+\to U$ is the blowup along $\Gamma$. 
Let $E^+\subset X_0^+$ be the exceptional divisor of the blowup $\tau$. 
\item\label{proposition:merge33}
Assume that $g=9$. Then $\Y=\pr^1_1\times\pr^1_2$ and the morphisms 
$\rho_i\colon \pr^1_1\times\pr^1_2\to\pr^1_i(=\Y_i)$ are the projections. 
Moreover, the morphism $\tau\colon X_0^+\to\pr^1_1\times\pr^1_2$ is a 
conic bundle with $\Delta_\tau\in\left|\sO(3,3)\right|$, and $\chi_{13}$ and 
$\chi_{23}$ are elementary flops. 
\end{enumerate}
\end{enumerate}
\end{proposition}

\begin{proof}
The proof is divided into 4 numbers of steps. 

\noindent\underline{\textbf{Step 1}}\\
We firstly remark that the assertion \eqref{proposition:merge2} follows from 
the assertion \eqref{proposition:merge1}. Indeed, assume that 
\eqref{proposition:merge1} is true. Under the identifications 
$\ND(X_0)=\ND(X_0^{+,1})=\ND(X_0^{+,2})$, both the nef cones $\Nef(X_0^{+,1})$ and 
$\Nef(X_0^{+,2})$ contain $\R_{\geq 0}\left[-K_{X_0}\right]$ and $\Nef(\Y)$ 
from the assertion \eqref{proposition:merge1}. 
Since $\tau^i$ is a $K_{X_0^{+,i}}$-negative elementary contraction ($i=1,2$), 
we must have 
\[
\left[-K_{X_0}\right]\not\in\R\Nef(\Y). 
\]
Thus $\R_{\geq 0}\left[-K_{X_0}\right]+\Nef(\Y)$ is a $3$-dimensional cone in 
$\ND(X_0)$. This implies that the interiors of $\Nef(X_0^{+,1})$ and 
$\Nef(X_0^{+,2})$ intersect. Thus we get the assertion \eqref{proposition:merge2} 
from \eqref{proposition:merge1}. Therefore, it is enough to show the assertions 
\eqref{proposition:merge1} and \eqref{proposition:merge3}. 

\noindent\underline{\textbf{Step 2}}\\
Assume that $g=12$. 
Under the natural isomorphism $\Pic X_0\cong \Pic X_0^{+,1}$, 
we set 
\[
\left[a_0,a_1,a_2\right]:=(\chi_{13}\circ\chi_{12}\circ\chi_{11})_*
\left(a_0\sigma^*(-K_X)+a_1F_1+a_2F_2\right)\in\Pic X_0^{+,1}
\]
for any $a_0,a_1,a_2\in\Z$. 
By Theorem \ref{thm:double-projection-from-line}, the pullback of 
$\sO_{V_1}(1)$ (resp., the pullback of $\sO_{Q_1}(1)$, 
the strict transform of $S_1$) corresponds to 
\[
\left[1,-2,0\right] \quad(\text{resp.,}\quad
\left[1,-2,-1\right], \quad\left[1,-3,0\right]).
\]
Thus the pullback of $-K_{\Y^1}$ on $X_0^{+,1}$ corresponds to 
$\left[2,-3,-3\right]$, and the exceptional divisor 
$E^{+,1}$ of $\tau^1$ corresponds to 
$\left[1,-2,-2\right]$. Recall that the anti-canonical divisor $-K_{\Y^1}$ of $\Y^1$
is ample by Proposition \ref{proposition:2-21}. Thus the birational contraction map 
$X_0\dashrightarrow \Y^1$ is the ample model of $\left[2,-3,-3\right]$. 
This implies that $\tau^1\circ\chi_{13}\circ\chi_{12}\circ\chi_{11}=\tau^2\circ\chi_{23}
\circ\chi_{22}\circ\chi_{21}$. The morphisms 
$\rho_1\circ\tau$ and $\rho_2\circ\tau$ are the ample models of 
$\left[1,-2,-1\right]$ and 
$\left[1,-1,-2\right]$, so that $\rho_1$ and $\rho_2$ 
are mutually distinct contraction morphisms. (We remark that the Picard rank of 
$\Y^i=Q_i$ is equal to one.)
Moreover, the center $\hat{C}$ of the blowup $\tau=\tau^1=\tau^2$ is 
the image of $E^{+,1}=E^{+,2}$, and it must be a bi-cubic curve since $\rho_i(\hat{C})=C_i$
holds.

\noindent\underline{\textbf{Step 3}}\\
Assume that $g=10$. 
Under the natural isomorphism $\Pic X_0\cong \Pic X_0^{+,1}$, 
we set 
\[
\left[a_0,a_1,a_2\right]:=(\chi_{13}\circ\chi_{12}\circ\chi_{11})_*
\left(a_0\sigma^*(-K_X)+a_1F_1+a_2F_2\right)\in\Pic X_0^{+,1}
\]
for any $a_0,a_1,a_2\in\Z$. The strict transform of $S_1^+$ corresponds to 
$\left[2,-5,0\right]$ and 
$(\rho_1\circ\tau^1)^*\sO_{\pr^2_1}(1)\sim\left[1,-2,-1\right]$ by 
Theorem \ref{thm:double-projection-from-line}. 
Thus we have $E^{+,1}\sim\left[3,-5,-5\right]$ by Proposition 
\ref{proposition:conic-bundle}, where $E^{+,1}$ is the exceptional divisor of $\tau^1$. 
Since we have $(\tau^1)^*(-K_{\Y^1})=-K_{X_0^{+,1}}+E^{+,1}\sim\left[4,-6,-6\right]$ and 
$\Pic X_0^{+,1}=(\tau^1)^*\Pic \Y^1\oplus\Z\left[E^{+,1}\right]$, 
there exists $L^1\in\Pic \Y^1$ such that 
\[(\tau^1)^*L^1\sim\left[2,-3,-3\right], \quad
-K_{\Y^1}\sim 2L^1,\quad 
\Pic \Y^1\cong\Z\left[\rho_1^*\sO_{\pr^2_1}(1)\right]\oplus\Z\left[L^1\right].
\]
Note that 
\begin{eqnarray*}
h^0\left(\Y^1,L^1-\rho_1^*\sO_{\pr^2_1}(1)\right)&=&h^0\left(X_0^{+,1},
\left[1,-1,-2\right]\right)
=h^0\left(X_0, \sigma^*(-K_X)-F_1-2F_2\right)\\
&=&h^0\left(\pr_2^2, \sO_{\pr_2^2}(1)\right)=3, \\
h^0\left(\Y^1,L^1-\rho_1^*\sO_{\pr^2_1}(2)\right)&=&
h^0\left(X_0, F_1-F_2\right)=0. 
\end{eqnarray*}
By Proposition \ref{proposition:MDS} and Lemma \ref{lemma:image-weak}, 
the variety $\Y^1$ is a weak Fano threefold. 
Since $\rho_1\colon \Y^1\to\pr^2_1$ is a $\pr^1$-bundle over $\pr^2_1$, 
there exists a vector bundle $\sE^1$ of rank $2$ on $\pr^2_1$ with 
$c_1(\sE^1)=0$ or $-1$ such that $\rho_1$ is equal to the projective space bundle 
$\pr_{\pr_1^2}(\sE^1)\to\pr^2_1$. Let $\xi_1\in\Pic \Y^1$ be the tautological 
line bundle with respects to the projective space bundle. Observe that 
\[
2L^1\sim -K_{\Y^1}\sim 2 \xi_1+\rho_1^*\sO_{\pr^2_1}(3-c_1(\sE^1)). 
\]
This immediately implies that $c_1(\sE^1)=-1$ and 
$L^1\sim \xi_1+\rho_1^*\sO_{\pr^2_1}(2)$. Moreover, from above, we have 
\begin{equation}\label{equation:langer}
h^0\left(\pr^2_1,\sE^1\otimes\sO_{\pr^2_1}(1)\right)=3, \quad
h^0\left(\pr^2_1,\sE^1\right)=0.
\end{equation}
Together with Langer's classification result \cite[Theorem 3.2]{langer}, the vector bundle 
$\sE^1$ must be isomorphic to $T_{\pr^2_1}\otimes\sO_{\pr^2_1}(-2)$. 
In particular, $L^1$ is ample. Moreover, the composition 
\[
\tau^1\circ\chi_{13}\circ\chi_{12}\circ\chi_{11}\colon X_0\dashrightarrow \Y^1
\]
is the ample model of $2\sigma^*(-K_X)-3F_1-3F_2$. This implies that 
$\tau^1\circ\chi_{13}\circ\chi_{12}\circ\chi_{11}=\tau^2\circ\chi_{23}
\circ\chi_{22}\circ\chi_{21}$. 
Moreover, under the the identification $U:=\Y^1=\Y^2$, the morphism 
$\rho_2\colon U\to \pr^2_2$ is given by the semiample divisor 
$L^1-\rho_1^*\sO_{\pr^2_1}(1)$. Thus two morphisms $\rho_1$, $\rho_2$ are 
mutually distinct. The remaining assertion can be proved as in Step 2. 

\noindent\underline{\textbf{Step 4}}\\
Assume that $g=9$. 
There exists $m\in\Z_{\geq 0}$ such that $\rho_1$ is equal to 
$\pr_{\pr^1_1}\left(\sO\oplus\sO(m)\right)\to\pr^1_1$. Let $\xi_1\in\Pic \Y^1$ be the 
tautological line bundle with respects to the projective bundle. 
Under the natural isomorphism $\Pic X_0\cong \Pic X_0^{+,1}$, 
we set 
\[
\left[a_0,a_1,a_2\right]:=(\chi_{13}\circ\chi_{12}\circ\chi_{11})_*
\left(a_0\sigma^*(-K_X)+a_1F_1+a_2F_2\right)\in\Pic X_0^{+,1}
\]
for any $a_0,a_1,a_2\in\Z$. The strict transform of $S_1^+$ corresponds to 
$\left[3,-7,0\right]$ and 
$(\rho_1\circ\tau^1)^*\sO_{\pr^1_1}(1)\sim\left[1,-2,-1\right]$ by 
Theorem \ref{thm:double-projection-from-line}. 
By the contraction theorem \cite[\S 3]{Mo82}, there exists an exact sequence 
\[
0\to \Pic \Y^1 \xrightarrow{(\tau^1)^*}\Pic X_0^{+,1}\to \Z\to 0. 
\]
Moreover, $(\tau^1)^*\Pic \Y^1\otimes\Q$ contains 
\[
\left[-\frac{1}{2},\frac{5}{2}, -1\right]\sim_\Q-K_{X_0^{+,1}}-\frac{1}{2}S_1^+,\quad
\left[1,-2,-1\right]\sim(\tau^1)^*\rho_1^*\sO_{\pr^1_1}(1)
\]
by Proposition \ref{proposition:dP-fibration}. This implies that 
\[
(\tau^1)^*\Pic \Y^1=\Z\left[1,-2,-1\right]\oplus\Z\left[1,-1,-2\right]. 
\]
Since $h^0\left(X_0^+,\left[1,-1,-2\right]\right)
=h^0\left(\pr^1_2,\sO_{\pr^1_2}(1)\right)=2>0$, 
there uniquely exists $b\in\Z$ such that 
$(\tau^1)^*\left(\xi_1+\rho_1^*\sO_{\pr^1_1}(b)\right)\sim\left[1,-1,-2\right]$. 
For any $k\in\Z$, we have 
\begin{eqnarray*}
&&H^0\left(\pr^1_1,\sO(k+b)\right)\oplus H^0\left(\pr^1_1,\sO(k+b+m)\right)\\
&\cong& H^0\left(\Y^1,\xi_1+\rho_1^*\sO_{\pr^1_1}(k+b)\right)\\
&=&H^0\left(X_0^{+,1},\left[1,-1,-2\right]+k\left[1,-2,-1\right]\right)\\
&=&H^0\left(X_0, (k+1)\sigma^*(-K_X)-(2k+1)F_1-(k+2)F_2\right). 
\end{eqnarray*}
In particular, we get 
\[
h^0\left(\pr^1_1,\sO(k+b)\right)+h^0\left(\pr^1_1,\sO(k+b+m)\right)
=\begin{cases}
2 &\text{if }k=0, \\
0 &\text{if }k\leq -1.
\end{cases}
\]
This immediately implies that $(m,b)=(0,0)$. 
Thus $\Y^1$ is isomorphic to $\pr^1\times\pr^1$. Moreover, the rational map 
$\tau^1\circ\chi_{13}\circ\chi_{12}\circ\chi_{11}\colon X_0\dashrightarrow \Y^1$ 
is the ample model of $\left[2,-3,-3\right]$. 
This implies that $\tau^1\circ\chi_{13}\circ\chi_{12}\circ\chi_{11}=
\tau^2\circ\chi_{23}\circ\chi_{22}\circ\chi_{21}$ and 
$(\Y:=)\Y^1=\Y^2=\pr^1\times\pr^1$. Moreover, 
$\rho_1$ and $\rho_2$ are mutually distinct projections since $\left[1,-2,-1\right]$ and 
$\left[1,-1,-2\right]$ are not proportional. 

In order to show that $\chi_{13}$ is an elementary flop, note that 
\begin{eqnarray*}
\left(\sO_{X_0^+}(\left[1,-1,-2\right])\right)^{\cdot 3}&=&0, \\ 
\left(\sO_{W_1^+}(\left[1,-1,-2\right])\right)^{\cdot 3}&=&
\left((\tau_1^+)^*\left(\phi_1^*\sO_{\pr^3_1}(2)+\psi_1^*\sO_{\pr^1_1}(2)\right)
-S_1^{W_1^+}\right)^{\cdot 3}=-5 
\end{eqnarray*}
holds (see \cite[Lemma 2.1]{MM85}). Thus $\beta_{13}$ is not an isomorphism. 
\end{proof}

By construction, $X_0$ itself, together with $W_1$, $W_1^+$, $X_0^+$, 
$W_2^+$, $W_2$ are small $\Q$-factorial modifications of $X_0$. 
By Proposition \ref{proposition:MDS}, all of them are smooth weak Fano threefolds. 
Let $\chi\colon X_0\dashrightarrow X_0^+$ be the natural small birational map. 

We summarize the diagram: 
\begin{equation}\label{equation:big}
\xymatrix{
&&\X_1&&\\
&\bar{X}_1&X_1^+ \ar[l]_-{\beta_1^+} \ar[u]^{\tau_1}&Y_1\ar[lu]_{\phi_1} \ar[rd]^{\psi_1}&\\
&X_1\ar[u]^-{\beta_1} \ar@{-->}[ru]^{\chi_1} \ar[ld]_{\sigma_1}&W_1\ar[u]^-{\phi_1^+}
\ar@{-->}[r] ^-{\chi_{12}}&W_1^+\ar[u]^{\tau_1^+} \ar@{-->}[d]^-{\chi_{13}} &\Y_1\\
X&&X_0\ar[lu]^{\sigma_2'} \ar[ll]_{\sigma} \ar[ld]_{\sigma_1'}
\ar@{-->}[r]^{\chi} \ar@{-->}[u]_-{\chi_{11}} \ar@{-->}[d]^-{\chi_{21}} &X_0^+
\ar[r]^{\tau}
&\Y \ar[u]_{\rho_1}\ar[d]^{\rho_2}\\
&X_2\ar[d]_-{\beta_2} 
\ar[lu]^{\sigma_2} \ar@{-->}[rd]_{\chi_2}&W_2 \ar@{-->}[r]_-{\chi_{22}} \ar[d]_-{\phi_2^+}
&W_2^+\ar[d]_{\tau_2^+} \ar@{-->}[u]_-{\chi_{23}} &\Y_2\\
&\bar{X}_2&X_2^+ \ar[l]^-{\beta_2^+} \ar[d]_{\tau_2}&Y_2\ar[ru]_{\psi_2} \ar[ld]^{\phi_2}&\\
&&\X_2&&
}
\end{equation}
where, 
\begin{enumerate}
\renewcommand{\theenumi}{\arabic{enumi}}
\renewcommand{\labelenumi}{(\theenumi)}
\item\label{X-Y:big1}
if $g=12$, then $\X_i=V_i$, $\Y_i=Q_i$, $\Y=\hat{Q}$ and $\tau$ is the blowup of 
$\hat{Q}$ along a bi-cubic curve $\hat{C}\subset\hat{Q}$ with the exceptional divisor 
$E^+\subset X_0^+$, 
\item\label{X-Y:big2}
if $g=10$, then $\X_i=Q_i$, $\Y_i=\pr^2_i$, $\Y=U$ and $\tau$ is the blowup of 
$U$ along a smooth bi-quintic curve $\Gamma$ of genus $2$ such that 
$\mult_{p_i}\left(\rho_i(\Gamma)\right)\leq 2$ for any $p_i\in\rho_i(\Gamma)$ 
with the exceptional divisor $E^+\subset X_0^+$, and 
\item\label{X-Y:big1}
if $g=9$, then $\X_i=\pr^3_i$, $\Y_i=\pr^1_i$, $\Y=\pr^1_1\times\pr^1_2$ 
and $\tau$ is a conic bundle with $\Delta_\tau\in\left|\sO(3,3)\right|$.
\end{enumerate}

The detail of the diagram \eqref{equation:big} for $\{i,j\}=\{1,2\}$, which is 
the combination of the diagrams \eqref{equation:1st-flop}, \eqref{equation:2nd-flop} 
and \eqref{equation:3rd-flop}, 
is the following: 
\begin{equation}\label{equation:big2}
\xymatrix{
& X_0 \ar@{-->}[rr]^-{\chi_{i1}} \ar[dr]_-{\beta_{i1}} \ar[dd]^(.6){\sigma'_j} 
\ar[dddl]_-{\sigma}& & 
W_i \ar@{-->}[rr]^-{\chi_{i2}} \ar[dr]_-{\beta_{i2}} \ar[dd]^(.6){\phi^+_i} 
\ar[dl]^-{\beta_{i1}^+} & & 
W_i^+ \ar@{-->}[rr]^-{\chi_{i3}} \ar[dr]_-{\beta_{i3}} \ar[dd]^(.6){\tau^+_i} 
\ar[dl]^-{\beta_{i2}^+} & & 
X_0^+ \ar[dd]^(.6){\tau} \ar[dl]^-{\beta_{i3}^+}\\
&&\bar{X}'_i \ar[dd]^(.25){\gamma_{i1}} &&\X'_i \ar[dd]^(.4){\gamma_{i2}} &&
\Y''_i \ar[dd]^(.4){\gamma_{i3}} &\\
& X_i \ar@{-->}[rr]^(.3){\chi_i} \ar[dl]^-{\sigma_i} \ar[dr]_-{\beta_i} & & 
X_i^+ \ar[dl]^-{\beta_i^+} \ar[dr]_-{\tau_i} & & Y_i \ar[dl]^-{\phi_i} \ar[dr]_-{\psi_i} 
& & \Y \ar[dl]^-{\rho_i}\\
X&&\bar{X}_i&&\X_i&&\Y_i&
}
\end{equation}

For $g\in\{10, 12\}$, let $E\subset X_0$ be the strict transform of $E^+\subset X_0^+$. 
Moreover, for any $\W\in\{X_0,W_1,W_1^+, X_0^+, W_2^+, W_2\}$, 
we set the following: 
\begin{itemize}
\item
For any $\E\in\left\{F_1, F_2, S_1, S_2\right\}$ 
(resp., for any $\E\in\left\{E, F_1, F_2, S_1, S_2\right\}$ when $g\in\{12, 10\}$), 
let $\E^\W\subset \W$ be the strict transform of $\E$ on $\X$. 
For example, the prime divisor 
$S_1^{W_1^+}\subset W_1^+$ coincides with the previous definition. 
\item
For any $\V\in\{\X_1,\X_2, \Y_1,\Y_2\}$, let $\sO_{\V}^{\W}(1)\in\Pic \W$ be the 
strict transform to $\W$ of the pullback of $\sO_{\V}(1)$ to some 
$\W'\in\{X_0,W_1,W_1^+, X_0^+, W_2^+, W_2\}$ with $\W'\to \V$ a morphism. 
This definition does not depend on the choice of $\W'$. 
For example, $\sO_{\X_1}^{X_0}(1)$ is the strict transform of 
$(\tau_1\circ\phi_1^+)^*\sO_{\X_1}(1)$ to $X_0$. 
\item
Similarly, for any $\V\in\{X, \bar{X}_0, \bar{X}_1, \bar{X}_2\}$, 
let $\sO_{\V}^{\W}(1)\in\Pic \W$ be the 
strict transform to $\W$ of the pullback of $\sO_{\V}(-K_{\V})$ to some 
$\W'\in\{X_0,W_1,W_1^+, X_0^+, W_2^+, W_2\}$ with $\W'\to \V$ a morphism. 
Recall that $\bar{X}_0$ is the anti-canonical model of $X_0$, (hence, of $W_1, W_1^+, 
X_0^+, W_2^+, W_2$ also). 
\end{itemize}

The following is trivial from the construction. 

\begin{lemma}\label{lemma:pic}
The classes $\sO_X^{X_0}(1)\left(=\sigma^*\sO_X(-K_X)\right)$, $F_1$, $F_2$ 
form a basis of $\Pic X_0$. 
Let us set 
\[
\left[a_0,a_1,a_2\right]:=a_0\sO_X^{X_0}(1)+a_1F_1+a_2F_2. 
\]
\begin{enumerate}
\renewcommand{\theenumi}{\arabic{enumi}}
\renewcommand{\labelenumi}{(\theenumi)}
\item\label{lemma:pic1}
If $g=12$, then we have
\begin{align*}
\sO_{\bar{X}_1}^{X_0}(1)&\sim\left[1,-1,0\right], \quad
&\sO_{\bar{X}_2}^{X_0}(1)&\sim\left[1,0,-1\right], \\
\sO_{V_1}^{X_0}(1)&\sim\left[1,-2,0\right], \quad
&\sO_{V_2}^{X_0}(1)&\sim\left[1,0,-2\right], \\
\sO_{Q_1}^{X_0}(1)&\sim\left[1,-2,-1\right], \quad
&\sO_{Q_2}^{X_0}(1)&\sim\left[1,-1,-2\right], \\
S_1&\sim\left[1,-3,0\right], \quad
&S_2&\sim\left[1,0,-3\right], \\
E&\sim\left[1,-2,-2\right], \quad
&\sO_{\bar{X}_0}^{X_0}(1)&\sim\left[1,-1,-1\right].
\end{align*}
\item\label{lemma:pic2}
If $g=10$, then we have
\begin{align*}
\sO_{\bar{X}_1}^{X_0}(1)&\sim\left[1,-1,0\right], \quad
&\sO_{\bar{X}_2}^{X_0}(1)&\sim\left[1,0,-1\right], \\
\sO_{Q_1}^{X_0}(1)&\sim\left[1,-2,0\right], \quad
&\sO_{Q_2}^{X_0}(1)&\sim\left[1,0,-2\right], \\
\sO_{\pr^2_1}^{X_0}(1)&\sim\left[1,-2,-1\right], \quad
&\sO_{\pr^2_2}^{X_0}(1)&\sim\left[1,-1,-2\right], \\
S_1&\sim\left[2,-5,0\right], \quad
&S_2&\sim\left[2,0,-5\right], \\
E&\sim\left[3,-5,-5\right], \quad
&\sO_{\bar{X}_0}^{X_0}(1)&\sim\left[1,-1,-1\right].
\end{align*}
\item\label{lemma:pic3}
If $g=9$, then we have
\begin{align*}
\sO_{\bar{X}_1}^{X_0}(1)&\sim\left[1,-1,0\right], \quad
&\sO_{\bar{X}_2}^{X_0}(1)&\sim\left[1,0,-1\right], \\
\sO_{\pr^3_1}^{X_0}(1)&\sim\left[1,-2,0\right], \quad
&\sO_{\pr^3_2}^{X_0}(1)&\sim\left[1,0,-2\right], \\
\sO_{\pr^1_1}^{X_0}(1)&\sim\left[1,-2,-1\right], \quad
&\sO_{\pr^1_2}^{X_0}(1)&\sim\left[1,-1,-2\right], \\
S_1&\sim\left[3,-7,0\right], \quad
&S_2&\sim\left[3,0,-7\right], \\
\sO_{\bar{X}_0}^{X_0}(1)&\sim\left[1,-1,-1\right]. &&
\end{align*}
\end{enumerate}
\end{lemma}

Now we are ready to consider the movable cone $\Mov(X_0)$ and 
the pseudo-effective cone $\Psef(X_0)$ of $X_0$. 

\begin{lemma}\label{lemma:coneX0}
The Picard rank of $\bar{X}_0$ is 
equal to one. The nef cone $\Nef(X_0)$ of $X_0$ is spanned by 
exactly $4$ numbers of rays 
\[
\R_{\geq 0}\left[\sO_X^{X_0}(1)\right], \quad
\R_{\geq 0}\left[\sO_{\bar{X}_1}^{X_0}(1)\right], \quad
\R_{\geq 0}\left[\sO_{\bar{X}_0}^{X_0}(1)\right], \quad
\R_{\geq 0}\left[\sO_{\bar{X}_2}^{X_0}(1)\right].
\] 
The facets of $\Nef(X_0)$ can be described as follows: 
\begin{eqnarray*}
\Nef(X_1)
&=&\R_{\geq 0}\left[\sO_X^{X_0}(1)\right]
+\R_{\geq 0}\left[\sO_{\bar{X}_1}^{X_0}(1)\right], \\
\Nef(\bar{X}'_1)=\Nef(X_0)\cap\Nef(W_1)
&=&\R_{\geq 0}\left[\sO_{\bar{X}_1}^{X_0}(1)\right]
+\R_{\geq 0}\left[\sO_{\bar{X}_0}^{X_0}(1)\right], \\
\Nef(\bar{X}'_2)=\Nef(X_0)\cap\Nef(W_2)
&=&\R_{\geq 0}\left[\sO_{\bar{X}_0}^{X_0}(1)\right]
+\R_{\geq 0}\left[\sO_{\bar{X}_2}^{X_0}(1)\right], \\
\Nef(X_2)
&=&\R_{\geq 0}\left[\sO_{\bar{X}_2}^{X_0}(1)\right]
+\R_{\geq 0}\left[\sO_X^{X_0}(1)\right].
\end{eqnarray*}
\end{lemma}

\begin{proof}
Firstly, note that the Picard rank of $X_0$ is equal to three and the Picard rank of 
$\bar{X}_0$ is equal to one by Proposition \ref{proposition:totally-blowup} 
\eqref{proposition:totally-blowup2}. 
Obviously, 
\[
\Nef(X_i)=\R_{\geq 0}\left[\sO_X^{X_0}(1)\right]
+\R_{\geq 0}\left[\sO_{\bar{X}_i}^{X_0}(1)\right]
\] 
is a two dimensional face 
of $\Nef(X_0)$. Moreover, since both $3$-dimensional cones 
$\Nef(X_0)$ and $\Nef(W_i)$ contain distinct rays 
\[
\R_{\geq 0}\left[\sO_{\bar{X}_0}^{X_0}(1)\right]
\quad\text{and}\quad
\R_{\geq 0}\left[\sO_{\bar{X}_i}^{X_0}(1)\right],
\] 
the intersection $\Nef(X_0)\cap\Nef(W_i)$ 
must be equal to the $2$-dimensional face 
\[
\R_{\geq 0}\left[\sO_{\bar{X}_0}^{X_0}(1)\right]
+\R_{\geq 0}\left[\sO_{\bar{X}_i}^{X_0}(1)\right],
\] 
which must be equal to the $2$-dimensional cone $\Nef(\bar{X}'_i)$. 
Thus we get the assertion. 
\end{proof}

Similarly, we get the following. We omit the proof since the strategy of the proof 
is completely same as the proof of Lemma \ref{lemma:coneX0}. 

\begin{lemma}\label{lemma:cone-hoka}
\begin{enumerate}
\renewcommand{\theenumi}{\arabic{enumi}}
\renewcommand{\labelenumi}{(\theenumi)}
\item\label{lemma:cone-hoka1}
The nef cone $\Nef(W_i)$ of $W_i$ is spanned by exactly three rays 
\[
\R_{\geq 0}\left[\sO_{\bar{X}_0}^{W_i}(1)\right], \quad 
\R_{\geq 0}\left[\sO_{\bar{X}_i}^{W_i}(1)\right], \quad 
\R_{\geq 0}\left[\sO_{\X_i}^{W_i}(1)\right].
\]
The facets of $\Nef(W_i)$ can be described as follows: 
\begin{eqnarray*}
\Nef(X_i^+)&=&
\R_{\geq 0}\left[\sO_{\X_i}^{W_i}(1)\right]
+\R_{\geq 0}\left[\sO_{\bar{X}_i}^{W_i}(1)\right], \\
\Nef(\bar{X}'_i)=\Nef(W_i)\cap\Nef(X_0)&=&
\R_{\geq 0}\left[\sO_{\bar{X}_i}^{W_i}(1)\right]
+\R_{\geq 0}\left[\sO_{\bar{X}_0}^{W_i}(1)\right], \\
\Nef(\X'_i)=\Nef(W_i)\cap\Nef(W_i^+)&=&
\R_{\geq 0}\left[\sO_{\bar{X}_0}^{W_i}(1)\right]
+\R_{\geq 0}\left[\sO_{\X_i}^{W_i}(1)\right].
\end{eqnarray*}
\item\label{lemma:cone-hoka2}
The nef cone $\Nef(W_i^+)$ of $W_i^+$ is spanned by exactly three rays 
\[
\R_{\geq 0}\left[\sO_{\bar{X}_0}^{W_i^+}(1)\right], \quad
\R_{\geq 0}\left[\sO_{\Y_i}^{W_i^+}(1)\right],\quad 
\R_{\geq 0}\left[\sO_{\X_i}^{W_i^+}(1)\right].
\]
The facets of $\Nef(W_i^+)$ can be described as follows: 
\begin{eqnarray*}
\Nef(Y_i)&=&
\R_{\geq 0}\left[\sO_{\X_i}^{W_i^+}(1)\right]
+\R_{\geq 0}\left[\sO_{\Y_i}^{W_i^+}(1)\right], \\
\Nef(\Y''_i)=\Nef(W_i^+)\cap\Nef(X_0^+)&=&
\R_{\geq 0}\left[\sO_{\Y_i}^{W_i^+}(1)\right]
+\R_{\geq 0}\left[\sO_{\bar{X}_0}^{W_i^+}(1)\right], \\
\Nef(\X'_i)=\Nef(W_i^+)\cap\Nef(W_i)&=&
\R_{\geq 0}\left[\sO_{\bar{X}_0}^{W_i^+}(1)\right]
+\R_{\geq 0}\left[\sO_{\X_i}^{W_i^+}(1)\right].
\end{eqnarray*}
\item\label{lemma:cone-hoka3}
The nef cone $\Nef(X_0^+)$ of $X_0^+$ is spanned by exactly three rays 
\[
\R_{\geq 0}\left[\sO_{\bar{X}_0}^{X_0^+}(1)\right], \quad 
\R_{\geq 0}\left[\sO_{\Y_1}^{X_0^+}(1)\right], \quad 
\R_{\geq 0}\left[\sO_{\Y_2}^{X_0^+}(1)\right].
\]
The facets of $\Nef(X_0^+)$ can be described as follows: 
\begin{eqnarray*}
\Nef(\Y)&=&
\R_{\geq 0}\left[\sO_{\Y_1}^{X_0^+}(1)\right]
+\R_{\geq 0}\left[\sO_{\Y_2}^{X_0^+}(1)\right], \\
\Nef(\Y''_1)=\Nef(X_0^+)\cap\Nef(W_1^+)&=&
\R_{\geq 0}\left[\sO_{\Y_1}^{X_0^+}(1)\right]
+\R_{\geq 0}\left[\sO_{\bar{X}_0}^{X_0^+}(1)\right], \\
\Nef(\Y''_2)=\Nef(X_0^+)\cap\Nef(W_2^+)&=&
\R_{\geq 0}\left[\sO_{\Y_2}^{X_0^+}(1)\right]
+\R_{\geq 0}\left[\sO_{\bar{X}_0}^{X_0^+}(1)\right].
\end{eqnarray*}
\end{enumerate}
\end{lemma}

As a consequence, we get the following result: 

\begin{thm}\label{thm:diamond}
The set of small $\Q$-factorial modifications of $X_0$ is equal to the set: 
\[
\left\{X_0, W_1, W_1^+, X_0^+, W_2^+, W_2\right\}.
\]
Moreover, we have
\begin{eqnarray*}
\Mov(X_0)&=&
\R_{\geq 0}\left[\sO_X^{X_0}(1)\right]
+\R_{\geq 0}\left[\sO_{\X_1}^{X_0}(1)\right]
+\R_{\geq 0}\left[\sO_{\Y_1}^{X_0}(1)\right]\\
&&+\R_{\geq 0}\left[\sO_{\Y_2}^{X_0}(1)\right]
+\R_{\geq 0}\left[\sO_{\X_2}^{X_0}(1)\right].
\end{eqnarray*}
\begin{enumerate}
\renewcommand{\theenumi}{\arabic{enumi}}
\renewcommand{\labelenumi}{(\theenumi)}
\item\label{lemma:pic1}
If $g\in\{12,10\}$, then we have
\[
\Psef(X_0)=\R_{\geq 0}[E]+\R_{\geq 0}[S_1]+\R_{\geq 0}[F_2]
+\R_{\geq 0}[F_1]+\R_{\geq 0}[S_2].
\]
\item\label{lemma:pic2}
If $g=9$, then we have
\[
\Psef(X_0)=\R_{\geq 0}[\sO_{\pr^1_2}^{X_0}(1)]+\R_{\geq 0}[\sO_{\pr^1_1}^{X_0}(1)]+
\R_{\geq 0}[S_1]+\R_{\geq 0}[F_2]
+\R_{\geq 0}[F_1]+\R_{\geq 0}[S_2].
\]
\end{enumerate}
\end{thm}

\begin{proof}
We can directly check that the $5$ rays 
\[
\R_{\geq 0}[S_i],\quad 
\R_{\geq 0}\left[\sO_{\X_i}^{X_0}(1)\right], \quad
\R_{\geq 0}\left[\sO_{\bar{X}_i}^{X_0}(1)\right],\quad
\R_{\geq 0}\left[\sO_X^{X_0}(1)\right],\quad 
\R_{\geq 0}[F_i]
\] 
lie in a same $2$-dimensional subspace; 
the $4$ rays  
\[
\R_{\geq 0}\left[\sO_{\Y_i}^{X_0}(1)\right], \quad 
\R_{\geq 0}\left[\sO_{\bar{X}_0}^{X_0}(1)\right], \quad
\R_{\geq 0}\left[\sO_{\bar{X}_j}^{X_0}(1)\right], \quad 
\R_{\geq 0}[F_i]
\]
lie in a same $2$-dimensional subspace; 
the $3$ rays  
\[
\R_{\geq 0}\left[\sO_{\X_1}^{X_0}(1)\right],\quad
\R_{\geq 0}\left[\sO_{\bar{X}_0}^{X_0}(1)\right],\quad
\R_{\geq 0}\left[\sO_{\X_2}^{X_0}(1)\right]
\]
lie in a same $2$-dimensional subspace. 
If $g=12$, then the $4$ rays 
\[
\R_{\geq 0}[S_1],\quad
\R_{\geq 0}\left[\sO_{Q_1}^{X_0}(1)\right],\quad 
\R_{\geq 0}\left[\sO_{Q_2}^{X_0}(1)\right], \quad 
\R_{\geq 0}[S_2]
\]
lie in a same $2$-dimensional subspace;
the $4$ rays 
\[
\R_{\geq 0}[E],\quad 
\R_{\geq 0}\left[\sO_{Q_i}^{X_0}(1)\right], \quad
\R_{\geq 0}\left[\sO_{V_i}^{X_0}(1)\right], \quad
\R_{\geq 0}[F_j]
\]
lie in a same $2$-dimensional subspace. If $g\in\{10,9\}$, then the $3$ rays 
\[
\R_{\geq 0}\left[\sO_{\Y_i}^{X_0}(1)\right], \quad
\R_{\geq 0}\left[\sO_{\X_i}^{X_0}(1)\right], \quad
\R_{\geq 0}[F_j]
\]
lie in a same $2$-dimensional subspace.
Thus the assertion follows from Proposition \ref{proposition:MDS} and 
Lemmas \ref{lemma:coneX0} and \ref{lemma:cone-hoka}. 
Indeed, if there is another small $\Q$-factorial modification $X_0^-$ of $X_0$ 
other than $X_0$, $W_1$, $W_1^+$, $X_0^+$, $W_2^+$, $W_2$, then the nef cone 
$\Nef(X_0^-)$ must contains point $[-K_{X_0}]$ and 
the interior of $\Nef(X_0^-)$ must be disjoint from the union 
$\Nef(X_0)\cup \Nef(W_1)\cup \Nef(W_1^+)\cup \Nef(X_0^+)\cup \Nef(W_2^+)\cup 
\Nef(W_2)$. However, the point $[-K_{X_0}]$ lies in the interior of the union from the 
above observation, a contradiction. 
(Note that the description of $\Psef(X_0)$ follows from 
\cite[Proposition 1.11 (2)]{HK}.)
\end{proof}

We illustrate a slice of the chamber decomposition of the cone 
$\Mov(X_0)$. For $\V\in\left\{\bar{X}_0, X,\bar{X}_1, \X_1, \Y_1, \Y_2, 
\X_2, \bar{X}_2\right\}$, we write 
$\R_{\geq 0}\sO_{\V}(1):=\R_{\geq 0}\left[\sO_{\V}^{X_0}(1)\right]$
for simplicity. 

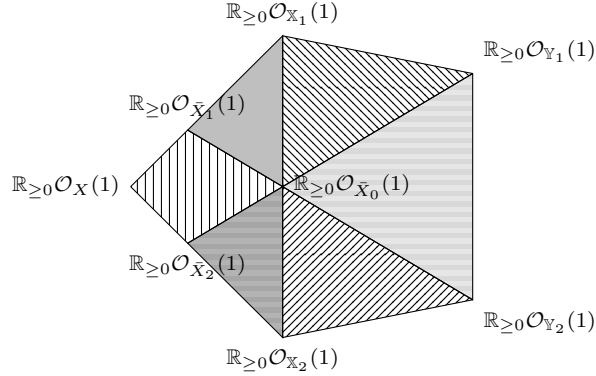
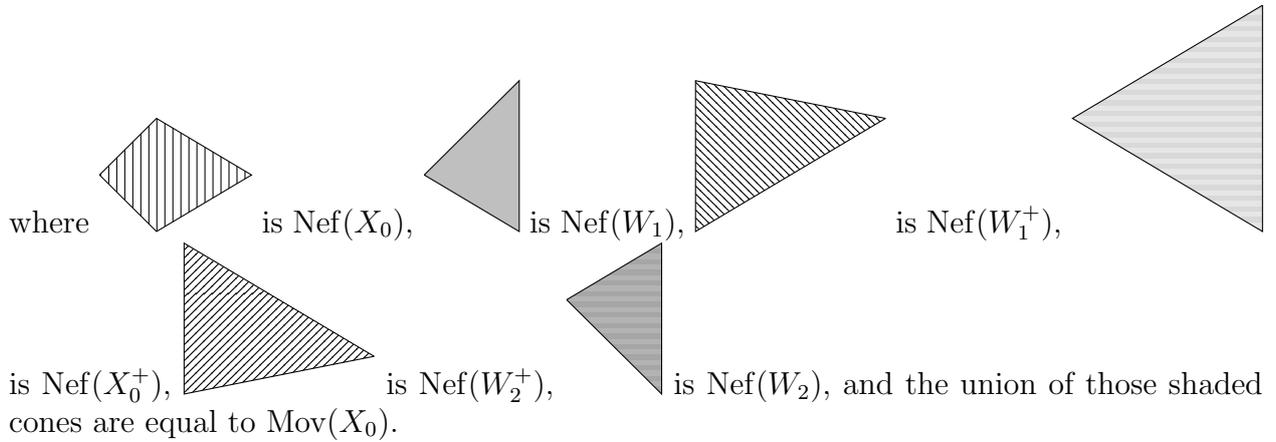
\[  \begin{tikzpicture}
\path[pattern=vertical lines] (0,0) -- (0.75,0.75) -- (2,0) --(0.75,-0.75)--(0,0);
\fill[lightgray] (0.75,0.75) -- (2,0) --(2,2)--(0.75,0.75);
\path[pattern=horizontal lines gray]  (0.75,-0.75) -- (2,0) --(2,-2)--(0.75,-0.75);
\path[pattern=north west lines]  (4.5,1.5) -- (2,0) --(2,2)--(4.5,1.5);
\path[pattern=north east lines]  (4.5,-1.5) -- (2,0) --(2,-2)--(4.5,-1.5);
\path[pattern=horizontal lines light gray]  (4.5,1.5) -- (2,0) --(4.5,-1.5)--(4.5,1.5);
\draw (0,0) -- (0.75,0.75) -- (2,0) --(0.75,-0.75)--(0,0);
\draw (0.75,0.75) -- (2,0) --(2,2)--(0.75,0.75);
\draw  (4.5,1.5) -- (2,0) --(2,2)--(4.5,1.5);
\draw  (4.5,1.5) -- (2,0) --(4.5,-1.5)--(4.5,1.5);
\draw  (4.5,-1.5) -- (2,0) --(2,-2)--(4.5,-1.5);
\draw  (0.75,-0.75) -- (2,0) --(2,-2)--(0.75,-0.75);
    \draw (0,0)node[left]{{\tiny $\R_{\geq 0}\sO_X(1)$}};
    \draw (0.75,0.75)node[above ]{{\tiny $\R_{\geq 0}\sO_{\bar{X}_1}(1)$}};
    \draw (0.75,-0.75)node[below]{{\tiny $\R_{\geq 0}\sO_{\bar{X}_2}(1)$}};
    \draw (2,2)node[above]{{\tiny $\R_{\geq 0}\sO_{\X_1}(1)$}};
    \draw (2,-2)node[below]{{\tiny $\R_{\geq 0}\sO_{\X_2}(1)$}};
    \draw (2,0)node[right]{{\tiny $\R_{\geq 0}\sO_{\bar{X}_0}(1)$}};
    \draw (4.5,1.5)node[above  right]{{\tiny $\R_{\geq 0}\sO_{\Y_1}(1)$}};
    \draw (4.5,-1.5)node[below  right]{{\tiny $\R_{\geq 0}\sO_{\Y_2}(1)$}};
\end{tikzpicture}\]
where 
\begin{tikzpicture}
\path[pattern=vertical lines] (0,0) -- (0.75,0.75) -- (2,0) --(0.75,-0.75)--(0,0);
\draw (0,0) -- (0.75,0.75) -- (2,0) --(0.75,-0.75)--(0,0);
\end{tikzpicture}
is $\operatorname{Nef}(X_0)$, 
\begin{tikzpicture}
\fill[lightgray] (0.75,0.75) -- (2,0) --(2,2)--(0.75,0.75);
%\path[pattern=horizontal lines gray] (0.75,0.75) -- (2,0) --(2,2)--(0.75,0.75);
\draw (0.75,0.75) -- (2,0) --(2,2)--(0.75,0.75);
\end{tikzpicture}
is $\operatorname{Nef}(W_1)$, 
\begin{tikzpicture}
\path[pattern=north west lines]  (4.5,1.5) -- (2,0) --(2,2)--(4.5,1.5);
\draw  (4.5,1.5) -- (2,0) --(2,2)--(4.5,1.5);
\end{tikzpicture}
is $\operatorname{Nef}(W_1^+)$, 
\begin{tikzpicture}
\path[pattern=horizontal lines light gray]  (4.5,1.5) -- (2,0) --(4.5,-1.5)--(4.5,1.5);
\draw  (4.5,1.5) -- (2,0) --(4.5,-1.5)--(4.5,1.5);
\end{tikzpicture}
is $\operatorname{Nef}(X_0^+)$, 
\begin{tikzpicture}
\path[pattern=north east lines]  (4.5,-1.5) -- (2,0) --(2,-2)--(4.5,-1.5);
\draw  (4.5,-1.5) -- (2,0) --(2,-2)--(4.5,-1.5);
\end{tikzpicture}
is $\operatorname{Nef}(W_2^+)$, 
\begin{tikzpicture}
\path[pattern=horizontal lines gray] (0.75,-0.75) -- (2,0) --(2,-2)--(0.75,-0.75);
\draw  (0.75,-0.75) -- (2,0) --(2,-2)--(0.75,-0.75);
\end{tikzpicture}
is $\operatorname{Nef}(W_2)$, and the  union of those shaded cones are equal to 
$\Mov(X_0)$. 

We also illustrate a slice of the chamber decomposition of the cone $\Psef(X_0)$. 
If $g=12$, then we have: 

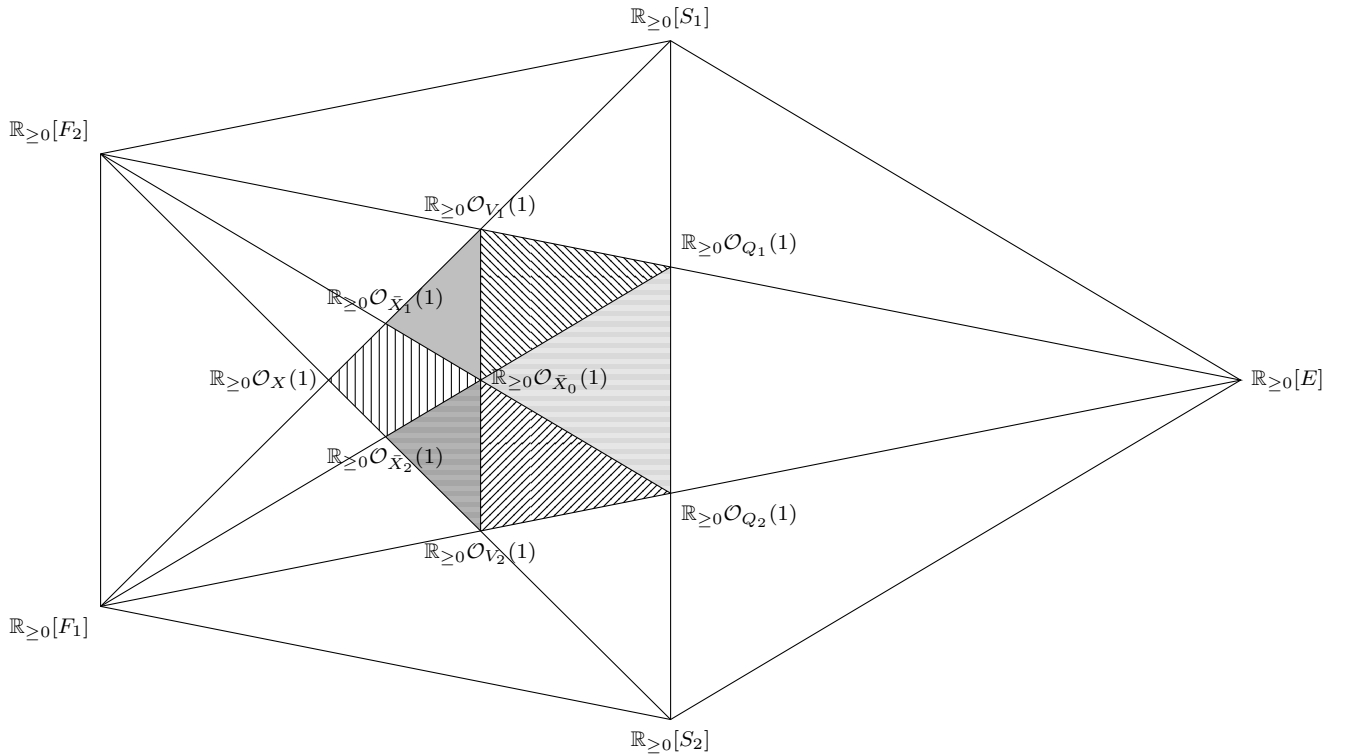
\[  \begin{tikzpicture}
\path[pattern=vertical lines] (0,0) -- (0.75,0.75) -- (2,0) --(0.75,-0.75)--(0,0);
\fill[lightgray] (0.75,0.75) -- (2,0) --(2,2)--(0.75,0.75);
%\path[pattern=horizontal lines gray]  (0.75,-0.75) -- (2,0) --(2,-2)--(0.75,-0.75);
\path[pattern=horizontal lines gray]  (0.75,-0.75) -- (2,0) --(2,-2)--(0.75,-0.75);
\path[pattern=north west lines]  (4.5,1.5) -- (2,0) --(2,2)--(4.5,1.5);
\path[pattern=north east lines]  (4.5,-1.5) -- (2,0) --(2,-2)--(4.5,-1.5);
\path[pattern=horizontal lines light gray]  (4.5,1.5) -- (2,0) --(4.5,-1.5)--(4.5,1.5);
    \draw (-3,3) -- (4.5,4.5) -- (12,0)--(4.5,-4.5)--(-3,-3)--(-3,3);
    \draw (-3,3) -- (12,0)--(-3,-3);
    \draw (-3,3) --(4.5,-1.5); 
    \draw (-3,-3) --(4.5,1.5); 
    \draw (-3,3) --(4.5,-4.5); 
    \draw (-3,-3) --(4.5,4.5); 
    \draw (4.5,4.5) --(4.5,-4.5); 
    \draw (2,2) --(2,-2); 
    \draw (0,0)node[left]{{\tiny $\R_{\geq 0}\sO_X(1)$}};
    \draw (-3,3)node[above  left]{{\tiny $\mathbb{R}_{\geq 0}[F_2]$}};
    \draw (-3,-3)node[below  left]{{\tiny $\mathbb{R}_{\geq 0}[F_1]$}};
    \draw (4.5,4.5)node[above]{{\tiny $\mathbb{R}_{\geq 0}[S_1]$}};
    \draw (4.5,-4.5)node[below]{{\tiny $\mathbb{R}_{\geq 0}[S_2]$}};
    \draw (12,0)node[right]{{\tiny $\mathbb{R}_{\geq 0}[E]$}};
    \draw (0.75,0.75)node[above ]{{\tiny $\R_{\geq 0}\sO_{\bar{X}_1}(1)$}};
    \draw (0.75,-0.75)node[below]{{\tiny $\R_{\geq 0}\sO_{\bar{X}_2}(1)$}};
    \draw (2,2)node[above]{{\tiny $\R_{\geq 0}\sO_{V_1}(1)$}};
    \draw (2,-2)node[below]{{\tiny $\R_{\geq 0}\sO_{V_2}(1)$}};
    \draw (2,0)node[right]{{\tiny $\R_{\geq 0}\sO_{\bar{X}_0}(1)$}};
    \draw (4.5,1.5)node[above  right]{{\tiny $\R_{\geq 0}\sO_{Q_1}(1)$}};
    \draw (4.5,-1.5)node[below  right]{{\tiny $\R_{\geq 0}\sO_{Q_2}(1)$}};
\end{tikzpicture}\]

If $g=10$, then we have: 

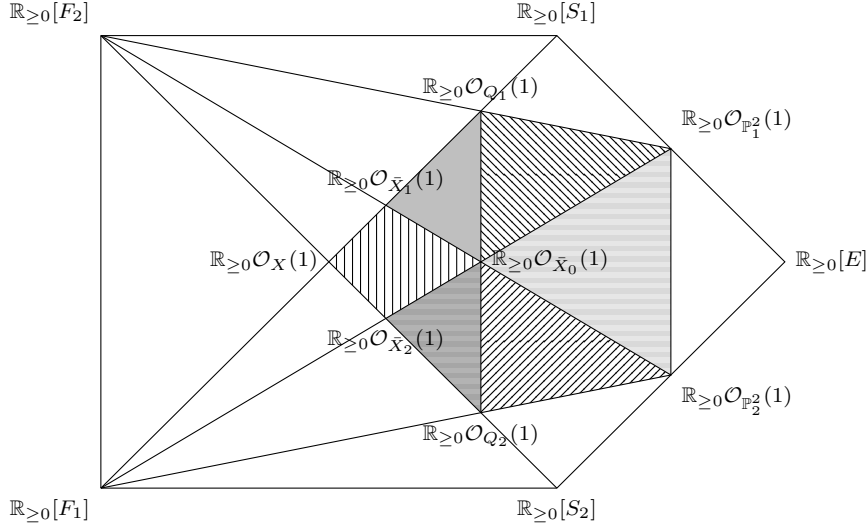
\[  \begin{tikzpicture}
\path[pattern=vertical lines] (0,0) -- (0.75,0.75) -- (2,0) --(0.75,-0.75)--(0,0);
\fill[lightgray] (0.75,0.75) -- (2,0) --(2,2)--(0.75,0.75);
%\path[pattern=horizontal lines gray]  (0.75,-0.75) -- (2,0) --(2,-2)--(0.75,-0.75);
\path[pattern=horizontal lines gray]  (0.75,-0.75) -- (2,0) --(2,-2)--(0.75,-0.75);
\path[pattern=north west lines]  (4.5,1.5) -- (2,0) --(2,2)--(4.5,1.5);
\path[pattern=north east lines]  (4.5,-1.5) -- (2,0) --(2,-2)--(4.5,-1.5);
\path[pattern=horizontal lines light gray]  (4.5,1.5) -- (2,0) --(4.5,-1.5)--(4.5,1.5);
    \draw (-3,3) -- (3,3) -- (6,0)--(3,-3)--(-3,-3);
    \draw (-3,3) -- (4.5,-1.5)--(4.5,1.5)--(-3,-3);
    \draw (-3,3) --(4.5,1.5); 
    \draw (-3,-3) --(4.5,-1.5); 
    \draw (-3,3) --(3,-3); 
    \draw (-3,-3) --(3,3); 
    \draw (2,2) --(2,-2); 
    \draw (-3,-3)--(-3,3);
    \draw (0,0)node[left]{{\tiny $\R_{\geq 0}\sO_X(1)$}};
    \draw (-3,3)node[above  left]{{\tiny $\mathbb{R}_{\geq 0}[F_2]$}};
    \draw (-3,-3)node[below  left]{{\tiny $\mathbb{R}_{\geq 0}[F_1]$}};
    \draw (3,3)node[above]{{\tiny $\mathbb{R}_{\geq 0}[S_1]$}};
    \draw (3,-3)node[below]{{\tiny $\mathbb{R}_{\geq 0}[S_2]$}};
    \draw (6,0)node[right]{{\tiny $\mathbb{R}_{\geq 0}[E]$}};
    \draw (0.75,0.75)node[above ]{{\tiny $\R_{\geq 0}\sO_{\bar{X}_1}(1)$}};
    \draw (0.75,-0.75)node[below]{{\tiny $\R_{\geq 0}\sO_{\bar{X}_2}(1)$}};
    \draw (2,2)node[above]{{\tiny $\R_{\geq 0}\sO_{Q_1}(1)$}};
    \draw (2,-2)node[below]{{\tiny $\R_{\geq 0}\sO_{Q_2}(1)$}};
    \draw (2,0)node[right]{{\tiny $\R_{\geq 0}\sO_{\bar{X}_0}(1)$}};
    \draw (4.5,1.5)node[above  right]{{\tiny $\R_{\geq 0}\sO_{\pr^2_1}(1)$}};
    \draw (4.5,-1.5)node[below  right]{{\tiny $\R_{\geq 0}\sO_{\pr^2_2}(1)$}};
\end{tikzpicture}\]

If $g=9$, then we have: 

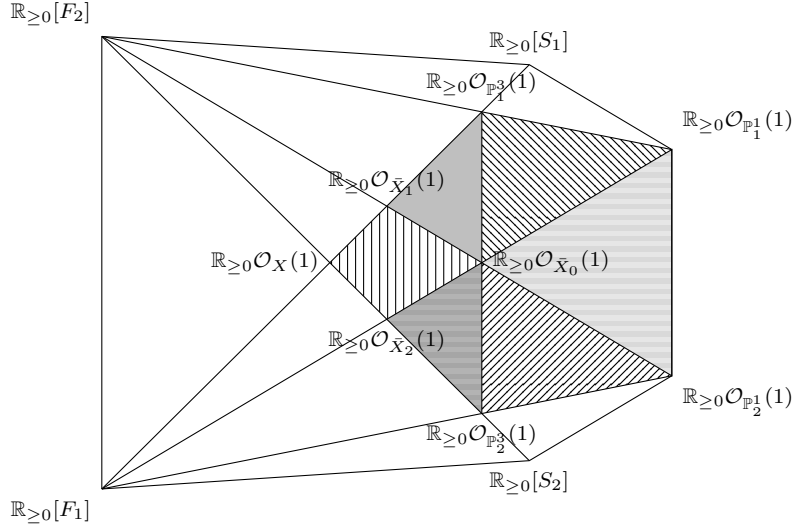
\[  \begin{tikzpicture}
\path[pattern=vertical lines] (0,0) -- (0.75,0.75) -- (2,0) --(0.75,-0.75)--(0,0);
\fill[lightgray] (0.75,0.75) -- (2,0) --(2,2)--(0.75,0.75);
%\path[pattern=horizontal lines gray]  (0.75,-0.75) -- (2,0) --(2,-2)--(0.75,-0.75);
\path[pattern=horizontal lines gray]  (0.75,-0.75) -- (2,0) --(2,-2)--(0.75,-0.75);
\path[pattern=north west lines]  (4.5,1.5) -- (2,0) --(2,2)--(4.5,1.5);
\path[pattern=north east lines]  (4.5,-1.5) -- (2,0) --(2,-2)--(4.5,-1.5);
\path[pattern=horizontal lines light gray]  (4.5,1.5) -- (2,0) --(4.5,-1.5)--(4.5,1.5);
    \draw (-3,3) -- (2.625,2.625) -- (4.5,1.5)--(4.5,-1.5)--(2.625,-2.625)--(-3,-3);
    \draw (-3,3) -- (4.5,-1.5)--(4.5,1.5)--(-3,-3);
    \draw (-3,3) --(4.5,1.5); 
    \draw (-3,-3) --(4.5,-1.5); 
    \draw (-3,3) --(2.625,-2.625); 
    \draw (-3,-3) --(2.625,2.625); 
    \draw (2,2) --(2,-2); 
    \draw (-3,-3)--(-3,3);
    \draw (0,0)node[left]{{\tiny $\R_{\geq 0}\sO_X(1)$}};
    \draw (-3,3)node[above  left]{{\tiny $\mathbb{R}_{\geq 0}[F_2]$}};
    \draw (-3,-3)node[below  left]{{\tiny $\mathbb{R}_{\geq 0}[F_1]$}};
    \draw (2.625,2.625)node[above]{{\tiny $\mathbb{R}_{\geq 0}[S_1]$}};
    \draw (2.625,-2.625)node[below]{{\tiny $\mathbb{R}_{\geq 0}[S_2]$}};
    \draw (0.75,0.75)node[above ]{{\tiny $\R_{\geq 0}\sO_{\bar{X}_1}(1)$}};
    \draw (0.75,-0.75)node[below]{{\tiny $\R_{\geq 0}\sO_{\bar{X}_2}(1)$}};
    \draw (2,2)node[above]{{\tiny $\R_{\geq 0}\sO_{\pr^3_1}(1)$}};
    \draw (2,-2)node[below]{{\tiny $\R_{\geq 0}\sO_{\pr^3_2}(1)$}};
    \draw (2,0)node[right]{{\tiny $\R_{\geq 0}\sO_{\bar{X}_0}(1)$}};
    \draw (4.5,1.5)node[above  right]{{\tiny $\R_{\geq 0}\sO_{\pr^1_1}(1)$}};
    \draw (4.5,-1.5)node[below  right]{{\tiny $\R_{\geq 0}\sO_{\pr^1_2}(1)$}};
\end{tikzpicture}\]

As an immediate corollary of Theorem \ref{thm:diamond}, we get the following 
concluding result. 

\begin{corollary}\label{corollary:go}
Let $X$ be a prime Fano threefold of genus $g\in\{12,10,9\}$, 
let $Z_1$, $Z_2$ be a pair of 
totally disjoint lines in $X$, let $\sigma\colon X_0\to X$ be the blowup along 
$Z_1\cup Z_2$, and let $F_1$, $F_2\subset X_0$ be the exceptional divisor of $\sigma$ 
over $Z_1$, $Z_2$, respectively. Then 
$X_0$ is a smooth weak Fano threefold and the anti-canonical model 
$\alpha\colon X_0\to \bar{X}_0$ of $X_0$ is small with $\rho(\bar{X}_0)=1$. 
Moreover, we have the following link: 
\begin{equation}\label{equation:go-conclusion}
\xymatrix{
&F_1\cup F_2 \ar@{}[r]|-{\subset} &
X_0 \ar[dl]_-{\sigma} \ar@{-->}[rr]^{\chi} \ar[dr]_-{\alpha} &&
X_0^+  \ar[dl]^-{\alpha^+} \ar[dr]^-{\tau} &  \\
Z_1\cup Z_2 \ar@{}[r]|-{\subset}  &X && \bar{X}_0 && 
\Y, 
}
\end{equation}
where $X_0^+$ is the $(\sigma^*K_X)$-flop of $\alpha$, i.e., 
\[
X_0^+=\Proj_{\bar{X}_0}\bigoplus_{m\in\Z_{\geq 0}}\alpha_*\sO_{X_0}
(\sigma^*(mK_X)),
\]
together with the structure morphism $\alpha^+\colon X_0^+\to \bar{X}_0$, 
$\chi:=(\alpha^+)^{-1}\circ\alpha$, and 
\[
\Y=\Proj\bigoplus_{m\in\Z_{\geq 0}}H^0\left(X_0, 
m\left(\sigma^*(-2K_X)-3(F_1+F_2)\right)\right).
\]
Moreover, we have the following: 
\begin{enumerate}
\renewcommand{\theenumi}{\arabic{enumi}}
\renewcommand{\labelenumi}{(\theenumi)}
\item\label{corollary:go1}
If $g=12$, then the variety $\hat{Q}=\Y$ is a Fano threefold of type 2.21, 
the morphism $\tau$ 
is obtained by the blowup along a bi-cubic curve $\hat{C}$ in $\hat{Q}$ and 
the exceptional divisor $E^+$ of $\tau$ is the strict transform of the unique member 
of $|\sigma^*(-K_X)-2(F_1+F_2)|$. 
\item\label{corollary:go2}
If $g=10$, then the variety $U=\Y$ is the del Pezzo threefold of degree $6$ and 
rank $2$, the morphism $\tau$ 
is obtained by the blowup along a smooth bi-quintic curve $\Gamma$ in $U$ of 
genus $2$ 
such that the multiplicity of $\rho_i(\Gamma)$ at any point is of multiplicity at most $2$ 
for $i\in\{1,2\}$, and 
the exceptional divisor $E^+$ of $\tau$ is the strict transform of the unique member 
of $|3\sigma^*(-K_X)-5(F_1+F_2)|$. 
\item\label{corollary:go3}
If $g=9$, then the variety $\Y$ is equal to $\pr^1\times\pr^1$, and the morphism 
$\tau$ is a conic bundle with $\Delta_\tau\in\left|\sO(3,3)\right|$. 
\end{enumerate}
\end{corollary}

\begin{remark}\label{remark:equivariant} 
The above link \eqref{equation:go-conclusion} 
is $\Aut(X; Z_1\cup Z_2)$-equivariant. In particular, 
if $Z_1\cup Z_2\subset X$ is $\operatorname{Aut}(X)$-invariant, then 
the above link is $\Aut(X)$-equivariant. 
See \S \ref{section:special} for a special case. 
\end{remark}

\section{Links to the blowups of 
prime Fano threefolds along two lines}\label{section:back}

The purpose of this section is to prove Theorem \ref{thm:main-back}. 
More precisely, we see the converse of the link constructed in \S 
\ref{section:go} when $g\in\{12,10\}$. 
In \S \ref{subsection:back-g12}, we consider the inverse link for the case $g=12$. 
In \S \ref{subsection:back-g10}, we consider the inverse link for the case $g=10$.

\subsection{The case $g=12$}\label{subsection:back-g12}

In \S \ref{subsection:back-g12}, we assume that 
$\hat{Q}$ is a Fano threefold of type 2.21, let $\rho_i\colon\hat{Q}\to Q_i$ be 
the distinct contractions $(i=1,2)$, and let $\Gamma_i^{Q_i}\subset Q_i$ be 
the center of the blowup $\rho_i$. 
We also assume that $\hat{C}\subset\hat{Q}$ is a bi-cubic curve in $\hat{Q}$, 
let us set $C_i:=(\rho_i)_*\hat{C}\subset Q_i$, and let 
$\tau\colon X_0^+\to \hat{Q}$ be the blowup along $\hat{C}\subset\hat{Q}$ 
with the exceptional divisor $E^+\subset X_0^+$. 

\begin{lemma}\label{lemma:nefnef}
The variety $X_0^+$ is a smooth weak Fano threefold with 
$(-K_{X_0^+})^{\cdot 3}=14$. 
\end{lemma}

\begin{proof}
Since $\hat{C}\subset\hat{Q}$ is a smooth rational curve and 
$\left(-K_{\hat{Q}}\cdot\hat{C}\right)=6$, we get
\[
(-K_{X_0^+})^{\cdot 3}=(-K_{\hat{Q}})^{\cdot 3}-14=14>0 
\]
(see \cite[Lemma 2.1]{MM85} for example). 
Thus it is enough to show that $-K_{X_0^+}$ is nef. Note that 
\[
-2K_{X_0^+}\sim \left((\rho_1\circ\tau)^*\sO_{Q_1}(2)-E^+\right)
+\left((\rho_2\circ\tau)^*\sO_{Q_2}(2)-E^+\right). 
\]
For $i\in\{1,2\}$, let us set 
\[
\left\{p'_{i1},\dots,p'_{im'_i}\right\}:=\left(\Gamma_i^{Q_i}\cap 
C_i\right)_{\operatorname{red}}, 
\]
and let us set ${\hat{B}}_{p'_{ik}}:=\rho_i^{-1}(p'_{ik})$ with the reduced structure. 
Since $C_i\subset Q_i$ is a twisted cubic curve, the base locus of the complete 
linear system 
\[
\left|(\rho_i\circ\tau)^*\sO_{Q_i}(2)-E^+\right|
\]
is contained in the subset 
\[
\bigcup_{1\leq k\leq m'_i}(\rho_i\circ\tau)^{-1}(p_{ik}). 
\]
Therefore, in order to show the nefness of $-K_{X_0^+}$, it is enough to show 
the inequality
$\left(-K_{X_0^+}\cdot\tau^{-1}_*{\hat{B}}_{p'_{ik}}\right)\geq 0$. 
By Lemma \ref{lemma:lengths}, we have 
\[
\left(-K_{X_0^+}\cdot\tau^{-1}_*{\hat{B}}_{p'_{ik}}\right)
=\left(-K_{\hat{Q}}\cdot {\hat{B}}_{p'_{ik}}\right)
-\operatorname{length}\left(\sO_{\hat{C}\cap {\hat{B}}_{p'_{ik}}}\right)=1-1=0.
\]
Thus we get the assertion. 
\end{proof}

Consider the link 
\[\xymatrix{
& E^{Y_1} \ar@{}[r]|{\subset} & Y_1 
\ar@{}[r]|{\supset} \ar@{->}[ld]_{\psi_1} \ar@{->}[rd]^{\phi_1}
& F_2^{Y_1} & \\
C_1  \ar@{}[r]|{\subset} & Q_1 && V_1 \ar@{}[r]|{\supset} & Z_2^{V_1}
}\]
from the blowup $Q_1$ along $C_1$ 
as in \eqref{equation:fujita-converse}. 
In particular, the variety $V_1$ is the del Pezzo threefold of degree $5$, 
$Z_2^{V_1}\subset V_1$ is a line and $\phi_1$ is the blowup of $V_1$ along $Z_2^{V_1}$
with the exceptional divisor $F_2^{V_1}\subset V_1$. 
By Theorem \ref{thm:dP5} \eqref{thm:dP55}, the image 
$(\psi_1)_*F_2^{Y_1}\subset Q_1$ is the hyperplane section containing $C_1$. 
Moreover, the linear span $\langle\Gamma_1^{Q_1}\rangle$
of the curve $\Gamma_1^{Q_1}\subset Q_1\subset\pr^4$ 
is the whole space $\pr^4$. 
Therefore, we can define the strict transform $\Gamma_1\subset V_1$ of 
$\Gamma_1^{Q_1}\subset Q_1$ 
since $\Gamma_1^{Q_1}\not\subset(\psi_1)_*F_2^{Y_1}$. 

\begin{lemma}\label{lemma:twisted-quintic}
\begin{enumerate}
\renewcommand{\theenumi}{\arabic{enumi}}
\renewcommand{\labelenumi}{(\theenumi)}
\item\label{lemma:twisted-quintic1}
We have 
\[
\operatorname{length}\left(\sO_{\Gamma_1\cap Z_2^{V_1}}\right)=1. 
\]
\item\label{lemma:twisted-quintic2}
Under the half-anti-canonical embedding $V_1\subset\pr^6$, the curve 
$\Gamma_1\subset V_1$ is a twisted quintic curve. 
Moreover, the linear span $\langle\Gamma_1\rangle\subset \pr^6$ does not 
contain $Z_2^{V_1}$. 
\end{enumerate}
\end{lemma}

\begin{proof}
Since $\operatorname{length}\left(\sO_{C_1\cap \Gamma_1^{Q_1}}\right)=3$, 
by Lemma \ref{lemma:lengths}, we have 
\begin{eqnarray*}
\left(\sO_{V_1}(1)\cdot\Gamma_1\right)
&=&\left(\psi_1^*\sO_{Q_1}(2)-E^{Y_1}\cdot (\psi_1)^{-1}_*\Gamma_1^{Q_1}\right)
=8-3=5, \\
\left(F_2^{Y_1}\cdot(\phi_1)^{-1}_*\Gamma_1\right)
&=&\left(\psi_1^*\sO_{Q_1}(1)-E^{Y_1}\cdot (\psi_1)^{-1}_*\Gamma_1^{Q_1}\right)
=4-3=1.
\end{eqnarray*}
Assume that there exists a singular point $p_1\in\Gamma_1$ of $\Gamma_1$. 
Then we must have $p_1\in Z_2^{V_1}$ since $(\phi_1)^{-1}_*\Gamma_1$ is 
smooth. Set $l_{p_1}:=\phi_1^{-1}(p_1)$ with the reduced structure. 
Then, by Lemma \ref{lemma:lengths}, we must have 
\[
\operatorname{length}\left(\sO_{l_{p_1}\cap(\phi_1)^{-1}_*\Gamma_1}\right)\geq 2. 
\]
This leads to a contradiction since $\left(F_2^{Y_1}\cdot
(\phi_1)^{-1}_*\Gamma_1\right)=1$. Thus $\Gamma_1$ is a smooth rational curve. 
This implies that $\Gamma_1\subset V_1\subset\pr^6$ is a twisted quintic curve 
by Lemma \ref{lemma:dP5quintic}. 
The assertion \eqref{lemma:twisted-quintic1} follows 
from Lemma \ref{lemma:lengths}. 

Assume that $Z_2^{V_1}\subset\langle\Gamma_1\rangle$. Since the rational map 
$V_1\dashrightarrow Q_1$ is the restriction of the projection $\pr^6\dashrightarrow 
\pr^4$ from the line $Z_2^{V_1}$, the image of $\langle\Gamma_1\rangle$ in $\pr^4$ 
must be a hyperplane. 
Since the linear span of $\Gamma_1^{Q_1}$ is the whole space $\pr^4$, 
this leads to a contradiction. Thus we have proved $Z_2^{V_1}
\not\subset\langle\Gamma_1\rangle$. 
\end{proof}

Let us consider the Sarkisov link 
\[\xymatrix{
& S_1^+ \ar@{}[r]|{\subset}  & X_1^+  \ar@{-->}[rr]^{\chi_1^{-1}} 
\ar@{->}[ld]_-{\tau_1} \ar@{->}[rd]_-{\beta_1^+} 
& & X_1 \ar@{->}[rd]^-{\sigma_1} \ar@{->}[ld]^-{\beta_1} \ar@{}[r]|{\supset} & F'_1 & \\
\Gamma_1 \ar@{}[r]|{\subset} & V_1 && \bar{X}_1 && X \ar@{}[r]|{\supset}& Z_1
}\]
from the blowup of $V_1$ along $\Gamma_1$ as in \eqref{equation:iskovskikh-converse}. 
In particular, the variety $X$ is a prime Fano threefold of genus $12$ and 
$Z_1\subset X$ is a line. 

\begin{proposition}\label{proposition:goback}
\begin{enumerate}
\renewcommand{\theenumi}{\arabic{enumi}}
\renewcommand{\labelenumi}{(\theenumi)}
\item\label{proposition:goback1}
The strict transform $(\tau_1)_*^{-1}\left(Z_2^{V_1}\right)\subset X_1^+$ of $Z_2^{V_1}
\subset V_1$ is not contained in the strict transform $(\chi_1)_*F'_1$. 
Moreover, the rational map $\chi_1^{-1}$ is an isomorphism around the 
neighborhood of $(\tau_1)_*^{-1}\left(Z_2^{V_1}\right)$. 
Thus we can consider the strict transform $Z_2\subset X$ of $Z_2^{V_1}
\subset V_1$. 
\item\label{proposition:goback2}
The curve $Z_2\subset X$ is a line in $X$. Moreover, $Z_1$, $Z_2$ is 
a totally disjoint pair of lines in $X$. 
\end{enumerate}
\end{proposition}

\begin{proof}
\eqref{proposition:goback1} 
The strict transform of $F'_1$ in $V_1$ is the hyperplane section 
of $V_1$ containing $\Gamma_1$ by Theorem 
\ref{thm:double-projection-from-line-converse} 
\eqref{thm:double-projection-from-line-converse1}. Since $Z_2^{V_1}
\not\subset\langle\Gamma_1\rangle$, we get the first assertion of 
\eqref{proposition:goback1}. 

We remark that the curve $(\tau_1)_*^{-1}\left(Z_2^{V_1}\right)$ intersects 
$-K_{X_1^+}$ with the intersection number $1$ by 
Lemma \ref{lemma:twisted-quintic}. Thus the curve 
$(\tau_1)_*^{-1}\left(Z_2^{V_1}\right)$ cannot be contracted by the flopping 
contraction $\beta_1^+\colon X_1^+\to\bar{X}_1$. 
Let us consider the blowup $\phi_1^+\colon W_1\to X_1^+$ of $X_1^+$ along 
$(\tau_1)_*^{-1}\left(Z_2^{V_1}\right)$. The variety $W_1$ is 
a small $\Q$-factorial modification of $X_0^+$. The variety $X_0^+$ is 
a smooth weak Fano threefold by Lemma \ref{lemma:nefnef}. 
Thus, by Proposition \ref{proposition:MDS}, the anti-canonical divisor $-K_{W_1}$ 
of $W_1$ is nef and big. In particular, the curve $(\tau_1)_*^{-1}\left(Z_2^{V_1}\right)$
cannot intersect with any flopping curve of $\chi_1^{-1}$. 
(Indeed, if a flopping curve $B\subset X_1^+$ of $\chi_1^{-1}$ intersects with 
$(\tau_1)_*^{-1}\left(Z_2^{V_1}\right)$, then its struct transform $B^{W_1}\subset W_1$ 
on $W_1$ satisfies that $0=\left(-K_{X_1^+}\cdot B\right)>
\left(-K_{W_1}\cdot B^{W_1}\right)$, a contradiction.)
Thus we get the assertion \eqref{proposition:goback1}. 

\eqref{proposition:goback2} 
 Let $Z_2^{X_1}\subset X_1$ be the strict transform of $Z_2^{V_1}\subset V_1$. 
By (1) (and by Lemma \ref{lemma:twisted-quintic}), 
we have $(F'_1\cdot Z_2^{X_1})=0$ and $(-K_{X_1}\cdot Z_2^{X_1})=1$. 
Thus $Z_2\subset X$ is a line disjoint from $Z_1$. Let $\sigma_2'\colon X_0\to X_1$
be the blowup of $X_1$ along $Z_2^{X_1}$. Note that $X_0$ is nothing but 
the blowup along $Z_1\cup Z_2$. Moreover, the variety $X_0$ is a small $\Q$-factorial 
modification of $X_0^+$. Therefore, again by Proposition \ref{proposition:MDS}, 
the variety $X_0$ is a smooth weak Fano threefold. Therefore, there is no line 
$Z\subset X$ with $Z\cap Z_1\neq\emptyset$ and $Z\cap Z_2\neq\emptyset$. 
Thus we get the assertion \eqref{proposition:goback2}. 
\end{proof}

Therefore, from any Fano threefold $\hat{Q}$ of type 2.21 and a bi-cubic curve 
$\hat{C}$ in $\hat{Q}$, we can construct the inverse of the link in 
\S \ref{section:go} for the case $g=12$. 

\begin{corollary}\label{corollary:back-g12}
Let $\hat{Q}$ be a Fano threefold of type 2.21, let $\hat{C}\subset\hat{Q}$ be 
a bi-cubic curve in $\hat{Q}$, and let $\tau\colon X_0^+\to \hat{Q}$ be the 
blowup along $\hat{C}$ with the $\tau$-exceptional divisor $E^+\subset X_0^+$. 
Then $X_0^+$ is a smooth weak Fano threefold and the anti-canonical model 
$\alpha^+\colon X_0^+\to \bar{X}_0$ of $X_0^+$ is small with $\rho(\bar{X}_0)=1$. 
Moreover, we have the following link: 
\begin{equation}\label{equation:back-g12}
\xymatrix{
&E^+\ar@{}[r]|-{\subset} & 
X_0^+\ar[dl]_-{\tau} \ar@{-->}[rr]^{\chi^{-1}} \ar[dr]_-{\alpha^+} &&
X_0\ar@{}[r]|-{\supset} \ar[dl]^-{\alpha} \ar[dr]^-{\sigma} & F_1\cup F_2 & \\
\hat{C} \ar@{}[r]|-{\subset} &\hat{Q}  && \bar{X}_0 && X 
\ar@{}[r]|-{\supset} & Z_1\cup Z_2, 
}
\end{equation}
where $X_0$ is the $(\tau^*K_{\hat{Q}})$-flop of $\alpha^+$, i.e., 
\[
X_0=\Proj_{\bar{X}_0}\bigoplus_{m\in\Z_{\geq 0}}(\alpha^+)_*\sO_{X_0^+}
\left(\tau^*(mK_{\hat{Q}})\right),
\]
together with the structure morphism $\alpha\colon X_0\to\bar{X}_0$, 
$\chi^{-1}:=\alpha^{-1}\circ\alpha^+$, and 
\[
X=\Proj\bigoplus_{m\in\Z_{\geq 0}}H^0\left(X_0^+, 
m\left(\tau^*(-2K_{\hat{Q}})-3E^+\right)\right).
\]
Moreover, the variety $X$ is a prime Fano threefold of genus $12$, the morphism 
$\sigma$ is obtained by the blowup along a totally disjoint pair of lines $Z_1$, $Z_2$ 
in $X$ and the exceptional divisor $F_1+F_2$ of $\sigma$ is the strict transform 
of the unique member of $|\tau^*(-K_{\hat{Q}})-2E^+|$. 
\end{corollary}

\begin{proof}
By Proposition \ref{proposition:totally-blowup} \eqref{proposition:totally-blowup3}, the anti-canonical model $\alpha\colon 
X_0\to \bar{X}_0$ of $X_0$ is small. Hence so is $\alpha^+$. 
The remaining assertions are trivial from the link \eqref{equation:big} 
in \S \ref{section:go} for the case $g=12$. 
\end{proof}

\begin{remark}\label{remark:aut}
Let $X$ be a prime Fano threefold of genus $12$ and let $Z_1$, $Z_2\subset X$ be 
a totally disjoint pair of lines on $X$. Consider the link as in \S \ref{section:go}. 
From the structure of the cone of divisors, we have 
\[
\Aut(X_0)\cong \Aut(X; Z_1\cup Z_2), \quad 
\Aut(X_0^+)\cong\Aut(\hat{Q}; \hat{C}). 
\]
Moreover, since $X_0^+$ (resp., $X_0$) is canonically defined from $X_0$ (resp., 
$X_0^+$), we have the natural isomorphism 
\[
\Aut(X_0)\cong\Aut(X_0^+). 
\]
Indeed, for any $\theta\in\Aut(X_0)$, we have $\chi\circ\theta\circ\chi^{-1}
\in\Aut(X_0^+)$ since 
\[
(\chi\circ\theta\circ\chi^{-1})_*\left(\tau^*(-K_{\hat{Q}})-(1+\varepsilon)E^+\right)
\sim_\Q\tau^*(-K_{\hat{Q}})-(1+\varepsilon)E^+
\]
for any $\varepsilon\in\Q_{>0}$. Similarly, for any $\theta^+\in\Aut(X_0^+)$, we can 
show that $\chi^{-1}\circ\theta^+\circ\chi\in\Aut(X_0)$. 
\end{remark}

\begin{proof}[Proof of Theorem \ref{thm:main-back} \eqref{thm:main-back1}]
Follows immediately from Corollaries \ref{corollary:go} and \ref{corollary:back-g12}. 
\end{proof}

\subsection{The case $g=10$}\label{subsection:back-g10}

In \S \ref{subsection:back-g10}, we assume that 
$U$ is the del Pezzo threefold of degree $6$ and rank $2$ 
together with the projections $\rho_1\colon U\to\pr^2_1$ and 
$\rho_2\colon U\to \pr^2_2$, and let $\Gamma\subset U$ be a smooth bi-quintic 
curve of genus $2$ such that the multiplicity of the plane curve $\rho_i(\Gamma)$ at any 
point is at most $2$ for each $i\in\{1,2\}$. Let $\tau\colon X_0^+\to U$ be 
the blowup of $U$ along $\Gamma$ and let $E^+\subset X_0^+$ be 
the exceptional divisor 
of $\tau$. We set $\sO_{X_0^+}(a_1,a_2):=\tau^*\sO_U(a_1,a_2)$ for any $a_1,a_2\in\Z$. 
The following lemma is trivial. 

\begin{lemma}\label{lemma:trivial-quintic}
For each $i\in\{1,2\}$, the plane curve $\rho_i(\Gamma)\subset\pr^2_i$ is a quintic 
curve. In other words, the restriction morphism 
$\rho_i|_\Gamma\colon \Gamma\to \rho_i(\Gamma)$ 
is birational. 
\end{lemma}

\begin{proof}
Assume that $\rho_1|_\Gamma\colon \Gamma\to \rho_1(\Gamma)$ is 
not birational. Then the curve $\rho_1(\Gamma)$ must be a line. 
Then $\Gamma$ must be contained in the pullback 
$\rho_1^{-1}\left(\rho_1(\Gamma)\right)$, which is isomorphic to 
$\pr_{\pr^1}(\sO\oplus\sO(1))$. 
Since $\Gamma$ is a bi-quintic curve, $\Gamma$ does not intersects with the 
$(-1)$-curve of the surface. Thus $\Gamma$ is isomorphic to a smooth plane 
quintic curve. However, since the genus of $\Gamma$ is equal to $2$, 
this leads to a contradiction. 
\end{proof}

The following two propositions are important in this section. 

\begin{proposition}\label{proposition:non-existence-surface}
We have $H^0\left(X_0^+, \sO_{X_0^+}(1,1)-E^+\right)=0$. 
In other words, any member in $|\sO_U(1,1)|$ does not contain the curve $\Gamma$. 
\end{proposition}

\begin{proof}
The proof is divided into 4 numbers of steps. 

\noindent\underline{\textbf{Step 1}}\\
Assume that there exists a surface $S\in|\sO_U(1,1)|$ with $\Gamma\subset S$. 
Obviously, such $S$ must be irreducible. By Lemma \ref{lemma:dP6}, 
the surface $S$ has only du Val singularities. Let $\nu\colon\tilde{S}\to S$ be 
the minimal resolution. Again by Lemma \ref{lemma:dP6}, the dual graph of the 
configuration of all negative curves of $\tilde{S}$ is one of \eqref{lemma:dP621}, 
\eqref{lemma:dP622} or \eqref{lemma:dP622} in Lemma \ref{lemma:dP6}. 

\noindent\underline{\textbf{Step 2}}\\
Let us consider the case \eqref{lemma:dP621} in Lemma \ref{lemma:dP6}. 
We may assume that the curves 
$\be_1,\be_2,\be_3$ are contracted by $\rho_1$, and the curves $\Bf_1,\Bf_2,\Bf_3$ 
are contracted by $\rho_2$. The Picard group of $S$ is generated by the classes 
of $\be_1,\Bf_3, \be_2, \Bf_1$. 
We can take $a,a',b,b'\in\Z$ such that 
$\Gamma\sim a\be_1+b'\Bf_3+a'\be_2+b\Bf_1$. 
Since $\Gamma$ is a bi-quintic curve, we have 
$a+a'=5$ and $b+b'=5$. Moreover, since $\left(\be_1\cdot \Gamma\right)\geq 0$, 
$\left(\Bf_2\cdot \Gamma\right)\geq 0$ and $\left(\be_3\cdot \Gamma\right)\geq 0$, 
we have 
$a+b\leq 5$ and $a,b\geq 0$. Moreover, since the genus of $\Gamma$ is equal to $2$, 
we have 
\[
2=\left((K_S+\Gamma)\cdot\Gamma\right)=-(a+b)^2+7(a+b)-10-a^2-b^2+3(a+b). 
\]
Therefore, 
\[
(a,b)=(0,2),\quad (0,3), \quad (2,0), \quad (3,0), \quad (2,3)\quad\text{ or } \quad (3,2). 
\]
In any case, there exists a $(-1)$-curve $\be$ in $S$ such that 
$(\Gamma\cdot\be)=3$. 
For example, if $(a,b)=(0,2)$, then we can take $\be=\be_1$. 
However, this implies that $\mult_{\rho_i(\be)}(\rho_i(\Gamma))=3$ for $i\in\{1,2\}$ such 
that $\be$ is contracted by $\rho_i$. 
This leads to a contradiction. Thus $S$ cannot be smooth. 

\noindent\underline{\textbf{Step 3}}\\
Let us consider the case \eqref{lemma:dP622} in Lemma \ref{lemma:dP6}. 
In this case,  the surface $\tilde{S}$ is 
a toric variety. We may assume that the morphism 
$\rho_1\circ\nu\colon\tilde{S}\to\pr^2_1$ 
is the contraction of the curves $\be_1$, $\Bc_0,\be_3$, and the morphism 
$\rho_2\circ\nu\colon\tilde{S}\to\pr^2_2$ 
is the contraction of the curves $\be_4$, $\Bc_0,\be_2$. 
The Picard group of $\tilde{S}$ is generated by the classes of 
$\be_1, \be_2,\Bc_0,\be_3$. 
We can take $a,a',a'',b\in\Z$ such that the strict transform $\tilde{\Gamma}
\subset\tilde{S}$ 
of $\Gamma\subset S$ is linearly equivalent to $a\be_1+a'\be_2+b\Bc_0+a''\be_3$. 
Since $\Gamma$ is a bi-quintic curve, we have $a'=5$ and $a+a''=5$. 
Moreover, since $\left(\tilde{\Gamma}\cdot\be_1\right)\geq 0$, 
$\left(\tilde{\Gamma}\cdot\be_2\right)\geq 0$ and 
$\left(\tilde{\Gamma}\cdot\Bc_0\right)\geq 0$, we have 
$a\leq 5$, $a+b\geq 5$ and $a+2b\leq 10$. 
Moreover, since the genus of $\tilde{\Gamma}$ is equal to $2$, we have 
\[
b^2+(a-10)b+a^2-10a+31=0.
\]
Therefore, we have 
\[
(a,b)=(2,3), \quad (3,2)\quad\text{or}\quad(5,2).  
\]
If $(a,b)=(2,3)$, then $\left(\tilde{\Gamma}\cdot\be_1\right)=3$ holds. This implies that 
$\mult_{\rho_1(\be_1)}(\rho_1(\Gamma))=3$, a contradiction. 
If $(a,b)=(3,2)$ or $(5,2)$, then $\left(\tilde{\Gamma}\cdot\Bc_0+\be_3\right)=3$ holds. 
This implies that $\mult_{\rho_1(\be_3)}(\rho_1(\Gamma))=3$ since the pullback of 
the exceptional divisor of the ordinary blowup of $\pr^2$ along the point 
$\rho_1(\be_3)$ is the Cartier divisor $\Bc_0+\be_3$ on $\tilde{S}$. 
This also leads to a contradiction. Thus the case \eqref{lemma:dP622} cannot occur. 

\noindent\underline{\textbf{Step 4}}\\
Let us consider the case \eqref{lemma:dP623} in Lemma \ref{lemma:dP6}. 
We may assume that the morphism $\rho_1\circ\nu\colon\tilde{S}\to\pr^2_1$ 
contracts the curves $\Bc_0,\Bc_1,\be_1$, and the morphism 
$\rho_2\circ\nu\colon\tilde{S}\to\pr^2_2$ 
contracts the curves $\Bc_0,\Bc_1,\be_2$. 
The Picard group of $\tilde{S}$ is generated by the classes of 
$\Bc_0, \Bc_1, \be_1, \be_2$. 
We can take $a_1,a_2,b_0, b_1\in\Z$ such that the strict transform 
$\tilde{\Gamma}\subset\tilde{S}$ of $\Gamma\subset S$ is linearly equivalent to 
$b_0\Bc_0+b_1\Bc_1+a_1\be_1+a_2\be_2$. 
Since $\Gamma$ is a bi-quintic curve, we have $a_1=a_2=5$. 
Moreover, since $\tilde{\Gamma}$ is an irreducible curve and different from 
$\Bc_0, \Bc_1, \be_1, \be_2$, we have $b_1\geq 2b_0$, $10+b_0-2b_1\geq 0$ and 
$b_1\geq 5$. Since the genus of $\tilde{\Gamma}$ is equal to $2$, we get 
\[
b_0^2-10b_1-b_0b_1+b_1^2+31=0. 
\]
Thus we have $(b_0,b_1)=(2,5)$. However, since 
\[
\mult_{\rho_1(\be_1)}(\rho_1(\Gamma))=\left(\tilde{\Gamma}\cdot 
(\Bc_0+\Bc_1+\be_1)\right)=3, 
\]
this leads to a contradiction. Thus we have completed the proof of Proposition 
\ref{proposition:non-existence-surface}. 
\end{proof}

\begin{proposition}\label{proposition:vanishing}
\begin{enumerate}
\renewcommand{\theenumi}{\arabic{enumi}}
\renewcommand{\labelenumi}{(\theenumi)}
\item\label{proposition:vanishing1}
We have 
\[
h^0\left(X_0^+,\sO_{X_0^+}(1,2)-E^+\right)
=h^0\left(X_0^+,\sO_{X_0^+}(2,1)-E^+\right)=1. 
\]
\item\label{proposition:vanishing2}
For any $j\geq 1$, we have 
\[
h^j\left(X_0^+,\sO_{X_0^+}(2-j,2)-E^+\right)
=h^j\left(X_0^+,\sO_{X_0^+}(2,2-j)-E^+\right)=0. 
\]
\end{enumerate}
\end{proposition}

\begin{proof}
Let $I_\Gamma\subset\sO_U$ be the coherent ideal sheaf corresponds to 
$\Gamma\subset U$. 
As is well-known (see \cite[Lemma 4.3.16]{L1}) that 
\[
H^j\left(X_0^+,\sO_{X_0^+}(a,b)-E^+\right)\cong H^j\left(U, \sO_U(a,b)\otimes 
I_\Gamma\right)
\]
holds for any $a,b\in\Z$ and for any $j\geq 0$. 
From the exact sequence
\[
0=H^1\left(\Gamma,\sO_U(0,2)|_\Gamma\right)
\to H^2\left(U, \sO_U(0,2)\otimes I_\Gamma\right)
\to H^2\left(U, \sO_U(0,2)\right)
\]
and the Kodaira vanishing theorem, we have $H^2\left(U, \sO_U(0,2)\otimes 
I_\Gamma\right)
=0$. Similarly, the equality 
$H^3\left(U, \sO_U(-1,2)\otimes I_\Gamma\right)=0$ is trivial. 

Let us consider the following exact sequence 
\begin{eqnarray*}
0&\to&H^0\left(U, \sO_U(1,2)\otimes I_\Gamma\right)\to
H^0\left(U, \sO_U(1,2)\right)\to H^0\left(\Gamma, \sO_U(1,2)|_\Gamma\right) \\
&\to&H^1\left(U, \sO_U(1,2)\otimes I_\Gamma\right)\to H^1\left(U, \sO_U(1,2)\right). 
\end{eqnarray*}
Since 
\begin{eqnarray*}
h^0\left(U, \sO_U(1,2)\right)
&=&h^0\left(\pr^2, T_{\pr^2}\otimes\sO_{\pr^2}(1)\right)=15,\\
h^1\left(U, \sO_U(1,2)\right)&=&0,\quad 
h^0\left(\Gamma, \sO_U(1,2)|_\Gamma\right)=15+1-2=14,
\end{eqnarray*}
we have $h^0\left(U, \sO_U(1,2)\otimes I_\Gamma\right)\geq 1$. 
Moreover, the equality $h^0\left(U, \sO_U(1,2)\otimes I_\Gamma\right)=1$ holds 
if and only if the equality $h^1\left(U, \sO_U(1,2)\otimes I_\Gamma\right)=0$ holds. 
Thus it is enough to show the inequality 
$h^0\left(U, \sO_U(1,2)\otimes I_\Gamma\right)\leq 1$. 

Assume that $h^0\left(U, \sO_U(1,2)\otimes I_\Gamma\right)\geq 2$. 
We can take mutually distinct divisors $F, F'\in|\sO_U(1,2)\otimes I_\Gamma|
\subset|\sO_U(1,2)|$. By Proposition \ref{proposition:non-existence-surface}, 
both $F$ and $F'$ are prime divisors. Let us write $F\cap F'=:\Gamma\cup B^1$. 
Since $\left(\sO_U(1,0)\cdot B^1\right)=3$ and $\left(\sO_U(0,1)\cdot B^1\right)=0$, 
any irreducible component of $B^1$ is a fiber of $\rho_2$. 

Let us take any irreducible component $B_1\subset B^1$ of $B^1$, set 
$p_1:=\rho_2(B_1)\in\pr^2_2$ and let $\theta_1\colon \mathbb{S}_1\to\pr^2_2$ 
be the blowup of $\pr^2_2$ at $p_1$ with the exceptional curve 
$\be_1\subset\mathbb{S}_1$. Moreover, let 
\[\xymatrix{
U \ar[d]_-{\rho_2} & \U_1 \ar[l]_-{\theta'_1} \ar[d]^-{\rho_{2,1}} \\
\pr^2_2 & \mathbb{S}_1 \ar[l]^-{\theta_1}
}\]
be the fiber product. Set $\E_1:=\rho_{2,1}^*\be_1$. The morphism $\theta'_1$ is 
nothing but the blowup of $U$ along $B_1$ with the exceptional divisor $\E_1$. 
Set  
\begin{eqnarray*}
\Gamma^1&:=&(\theta'_1)_*^{-1}\Gamma,\quad
n_1:=\left(\E_1\cdot \Gamma^1\right)=\operatorname{length}\left(\sO_{\Gamma
\cap B_1}\right)
\leq 2, \\
m_1&:=&\mult_{B_1}F, \quad m'_1:=\mult_{B_1}F', \quad
F_1:=(\theta'_1)^{-1}_*F, \quad F'_1:=(\theta'_1)^{-1}_*F'. 
\end{eqnarray*}
Let us write $F_1\cap F'_1=:\Gamma^1\cup B^2$. Then we have 
\begin{eqnarray*}
\left((\theta'_1)^*\sO_U(1,0)\cdot B^2\right)&=&3-m_1m'_1, \\
\left((\theta'_1)^*\sO_U(0,1)\cdot B^2\right)&=&0, \\
\left(\E_1\cdot B^2\right)&=&m_1+m'_1-n_1.
\end{eqnarray*}
Let us decompose $B^2=B^{2,1}+B^{2,2}$ as $1$-cycles, where $B^{2,1}$ is the sum 
of the components of $B^2$ containing $\E_1$. 
Note that $(\rho_{2,1})_*B^{2,2}$ is an effective $1$-cycle on $\mathbb{S}_1$ which 
does not contain $\be_1$ and 
\[
\left(\theta_1^*\sO_{\pr^2_2}(1)\cdot(\rho_{2,1})_*B^{2,2}\right)
=\left((\theta'_1)^*\sO_U(0,1)\cdot B^{2,2}\right)=0, 
\]
since the pullback of $\sO_{\pr_2^2}(1)$ is nef. 
This implies that $(\rho_{2,1})_*B^{2,2}=0$. In particular, we have 
$\left(\E_1\cdot B^{2,2}\right)=0$. 
There is an isomorphism $\E_1\cong\pr^1\times\pr^1$ with 
$\E_1|_{\E_1}\cong\sO(0,-1)$. Under the isomorphism, we can write 
$B^{2,1}\in|\sO(a_1,b_1)|$ with $a_1,b_1\geq 0$. Since 
\[
-a_1=\left(\E_1\cdot B^{2,1}\right)=m_1+m'_1-n_1, 
\]
we must have $a_1=0$, $m_1=m'_1=1$ and $n_1=2$.  In particular, any irreducible 
component of $B^2$ is a fiber of $\rho_{2,1}$. 

Let us take any irreducible component $B_2\subset B^2$ of $B^2$, set 
$p_2:=\rho_{2,1}(B_2)\in\mathbb{S}_1$ and let 
$\theta_2\colon \mathbb{S}_2\to\mathbb{S}_1$ 
be the blowup of $\mathbb{S}_1$ at $p_2$ with the exceptional curve 
$\be_2\subset\mathbb{S}_2$. Moreover, let 
\[\xymatrix{
\U_1 \ar[d]_-{\rho_{2,1}} & \U_2 \ar[l]_-{\theta'_2} \ar[d]^-{\rho_{2,2}} \\
\mathbb{S}_1 & \mathbb{S}_2 \ar[l]^-{\theta_2}
}\]
be the fiber product. Set $\E_2:=\rho_{2,2}^*\be_2$. The morphism $\theta'_2$ is 
nothing but the blowup of $\U_1$ along $B_2$ with the exceptional divisor $\E_2$. 
Set  
\begin{eqnarray*}
\Gamma^2&:=&(\theta'_2)_*^{-1}\Gamma^1,\quad
n_2:=\left(\E_2\cdot \Gamma^2\right)=\operatorname{length}\left(\sO_{\Gamma^1
\cap B_2}\right)
\leq 2, \\
m_2&:=&\mult_{B_2}F_1, \quad m'_2:=\mult_{B_2}F'_1, \quad
F_2:=(\theta'_2)^{-1}_*F_1, \quad F'_2:=(\theta'_2)^{-1}_*F'_1 
\end{eqnarray*}
as before. 
Let us write $F_2\cap F'_2=:\Gamma^2\cup B^3$. Then we have 
\begin{eqnarray*}
\left((\theta'_1\circ\theta'_2)^*\sO_U(1,0)\cdot B^3\right)&=&2-m_2m'_2, \\
\left((\theta'_1\circ\theta'_2)^*\sO_U(0,1)\cdot B^3\right)&=&0, \\
\left((\theta'_2)^*\E_1\cdot B^3\right)&=&0,\\
\left(\E_2\cdot B^3\right)&=&m_2+m'_2-n_2. 
\end{eqnarray*}
Let us decompose $B^3=B^{3,1}+B^{3,2}$ as $1$-cycles, where $B^{3,1}$ is the sum 
of the components of $B^3$ containing $\E_2$. 
Again, $(\rho_{2,2})_*B^{3,2}$ is an effective $1$-cycle on $\mathbb{S}_2$ which 
does not contain $\be_2$, and 
\[
\left(\theta_2^*\left(\theta_1^*\sO_{\pr^2_2}(2)-\be_1\right)\cdot
(\rho_{2,2})_*B^{3,2}\right)
=\left((\theta'_2)^*\left((\theta'_1)^*\sO_U(0,2)-\E_1\right)\cdot B^{3,2}\right)=0. 
\]
This implies that $(\rho_{2,2})_*B^{3,2}=0$. In particular, we have 
$\left(\E_1\cdot B^{3,2}\right)=0$. 
There is an isomorphism $\E_2\cong\pr^1\times\pr^1$ with 
$\E_2|_{\E_2}\cong\sO(0,-1)$. Under the isomorphism, we can write 
$B^{3,1}\in|\sO(a_2,b_2)|$ with $a_2,b_2\geq 0$. Since 
\[
-a_2=\left(\E_2\cdot B^{3,1}\right)=m_2+m'_2-n_2, 
\]
we must have $a_2=0$, $m_2=m'_2=1$ and $n_2=2$.  In particular, again, 
any irreducible component of $B^3$ is a fiber of $\rho_{2,2}$. 

Let us take any irreducible component $B_3\subset B^3$ of $B^3$, let 
$\theta'_3\colon \U_3\to \U_2$ be the blowup of $\U_2$ along $B_3$ with 
the exceptional divisor $\E_3$. Set 
\begin{eqnarray*}
\Gamma^3&:=&(\theta'_3)_*^{-1}\Gamma^2,\quad
n_3:=\left(\E_3\cdot \Gamma^3\right)=\operatorname{length}\left(\sO_{\Gamma^2
\cap B_3}\right)
\leq 2, \\
m_3&:=&\mult_{B_3}F_2, \quad m'_3:=\mult_{B_3}F'_2, \quad
F_3:=(\theta'_3)^{-1}_*F_2, \quad F'_3:=(\theta'_3)^{-1}_*F'_2 
\end{eqnarray*}
as before. Then, we can also compute in a same way that $m_3=m'_3=1$, $n_3=2$ and 
$F_3\cap F'_3=\Gamma^3$ as $1$-cycles on $\U_3$. On the other hand, since 
$F_3\cap F'_3$ is a complete intersection of divisors in $\U_3$, the intersection 
is scheme-theoretically equal to $\Gamma^3$. This implies that 
\[
2p_a(\Gamma^3)-2=\left(\left(K_{\U_3}+F_3+F'_3\right)\cdot F_3\cdot F'_3\right)
=4.
\]
However, this contradicts with the fact that the genus of $\Gamma$ is equal to $2$. 
Thus we get the assertion. 
\end{proof}

\begin{corollary}\label{corollary:g10converse-weakness}
The anti-canonical divisor $-K_{X_0^+}$ of $X_0^+$ is globally generated. 
In particular, the variety $X_0^+$ is a smooth weak Fano threefold with 
$(-K_{X_0^+})^{\cdot 3}=10$. 
\end{corollary}

\begin{proof}
The equality $(-K_{X_0^+})^{\cdot 3}=10$ can be obtained by 
\cite[Lemma 2.1]{MM85}. Thus the assertion is a direct consequence of 
Propositions \ref{proposition:CM} \eqref{proposition:CM2}, \ref{proposition:vanishing}
\eqref{proposition:vanishing2} and Lemma \ref{lemma:conic-bundle} 
\eqref{lemma:conic-bundle2}. 
\end{proof}

Let us consider the elementary flop 
\[\xymatrix{
X_0^+ \ar[d]_-{\tau} \ar@{-->}[rr]^-{\chi_{13}^{-1}} \ar[dr]^-{\beta_{13}^+} & & 
W_1^+ \ar[dl]_-{\beta_{13}} \ar[d]^-{\tau_1^+} \\
U \ar[dr]_-{\rho_1} & P''_1 \ar[d]_(.4){\gamma_{13}} & Y_1 \ar[dl]^-{\psi_1} \\
& \pr^2_1 &
}\]
as in Proposition \ref{proposition:conic-bundle}. 
The variety $Y_1$ is a $\pr^1$-bundle over $\pr^2_1$, and the morphism 
$\tau_1^+\colon W_1^+\to Y_1$ is the blowup of $Y_1$ along a smooth irreducible 
curve $\Gamma_1^{Y_1}\subset Y_1$ of genus $2$ with the exceptional divisor 
$S_1^{W_1^+}\subset W_1^+$. Note that the curve $\Gamma_1^{Y_1}$ birationally 
maps onto the singular curve $\rho_1(\Gamma)$, and the morphism $\beta_{13}^+$ is 
not an isomorphism. 
The variety $W_1^+$ is a smooth weak Fano threefold by Proposition 
\ref{proposition:MDS}. Thus so is $Y_1$ by Lemma \ref{lemma:image-weak}. 
Under the natural isomorphism $\Pic X_0^+\cong\Pic W_1^+$, we know that 
\[
S_1^{W_1^+}\sim(\chi_{13})^{-1}_*\left(\sO_{X_0^+}(5,0)-E^+\right). 
\] 
Thus we get 
\begin{equation}\label{equation:g10converse-Y}
(\tau_1^+)^*(-K_{Y_1})\sim(\chi_{13})^{-1}_*\left(\sO_{X_0^+}(7,2)-2E^+\right).
\end{equation}
Thus, the anti-canonical divisor $-K_{Y_1}$ of $Y_1$ is not divisible by $2$ in 
$\Pic Y_1$. Therefore, there exists a rank $2$ vector bundle $\sE$ on $\pr^2_1$
with $c_1(\sE)=0$ such that the morphism $\psi_1$ is equal to the projective 
space bundle $\pr_{\pr^2_1}(\sE)\to\pr^2_1$. Let $\xi$ be the tautological line bundle 
with respects to the projective space bundle. Then we have 
\[
(\tau_1^+)^*\xi\sim(\chi_{13})^{-1}_*\left(\sO_{X_0^+}(2,1)-E^+\right). 
\]
Note that
\[
h^0\left(\pr^2_1,\sE\right)=h^0\left(W_1^+, (\tau_1^+)^*\xi\right)
=h^0\left(X_0^+,\sO_{X_0^+}(2,1)-E^+\right)=1
\]
by Proposition \ref{proposition:vanishing}. In particular, $\sE$ is not a stable vector 
bundle. Moreover, we get 
\[
h^0\left(\pr^2_1,\sE\otimes\sO_{\pr^2_1}(-1)\right)
=h^0\left(X_0^+,\sO_{X_0^+}(1,1)-E^+\right)=0
\]
by Proposition \ref{proposition:non-existence-surface}. 
By Yasutake's classification result \cite[Proposition 2.10]{Y} and by Lemma 
\ref{lemma:2-31}, there exists a birational morphism $\phi_1\colon Y_1\to Q_1$ 
and a line $Z_2^{Q_1}\subset Q_1$ 
such that $Q_1$ is the $3$-dimensional smooth hyperquadric and the morphism 
$\phi_1$ is the blowup of $Q_1$ along $Z_2^{Q_1}$ with the exceptional divisor 
$F_2^{Y_1}\subset Y_1$. 

\begin{lemma}\label{lemma:g10converse-easy-numerical}
The divisor $F_2^{Y_1}$ does not contain the curve $\Gamma_1^{Y_1}$. 
Moreover, we have 
\begin{eqnarray*}
(\tau_1^+)^*\sO_{Q_1}(1)&\sim&(\chi_{13})^{-1}_*\left(\sO_{X_0^+}(3,1)-E^+\right), \\
(\tau_1^+)^*F_2^{Y_1}&\sim&(\chi_{13})^{-1}_*\left(\sO_{X_0^+}(2,1)-E^+\right). 
\end{eqnarray*}
\end{lemma}

\begin{proof}
Since $-K_{Y_1}\sim\phi_1^*\sO_{Q_1}(2)+\psi_1^*\sO_{\pr^2_1}(1)$ and 
$F_2^{Y_1}\sim \phi_1^*\sO_{Q_1}(1)-\psi_1^*\sO_{\pr^2_1}(1)$ 
(see Lemma \ref{lemma:2-31}), the second assertion is trivial from 
\eqref{equation:g10converse-Y}. 
Assume that $\Gamma_1^{Y_1}\subset F_2^{Y_1}$. Then we have 
\[
0\neq h^0\left(W_1^+,(\tau_1^+)^*F_2^{Y_1}-S_1^{W_1^+}\right)
=h^0\left(U,\sO_U(-3,1)\right)=0, 
\]
a contradiction. 
\end{proof}

\begin{lemma}\label{lemma:g10converse-a}
Set $\Gamma_1^{Q_1}:=\phi_1(\Gamma_1^{Y_1})\subset Q_1$. 
Then the curve $\Gamma_1^{Q_1}$ is a smooth curve of genus $2$ with 
$\left(\sO_{Q_1}(1)\cdot \Gamma_1^{Q_1}\right)=7$ and 
$\operatorname{length}\left(\sO_{\Gamma_1^{Q_1}\cap Z_2^{Q_1}}\right)=2$. 
\end{lemma}

\begin{proof}
Since $\Gamma_1^{Y_1}\not\subset F_2^{Y_1}$, the curve $\Gamma_1^{Y_1}$ 
maps birationally onto its image by $\phi_1$. 
If there is a singular point $p\in\Gamma_1^{Q_1}$, then, since 
$\Gamma_1^{Y_1}$ is smooth, 
the fiber $l:=\phi_1^{-1}(p)$ is a smooth curve with $(-K_{Y_1}\cdot l)=1$ and 
$\operatorname{length}\left(\sO_{l\cap \Gamma_1^{Y_1}}\right)\geq 2$ by Lemma 
\ref{lemma:lengths}. However, since $W_1^+$ is a smooth weak Fano threefold, we have 
\[
0\leq\left(-K_{W_1^+}\cdot (\tau_1^+)_*^{-1}l\right)=1-
\operatorname{length}\left(\sO_{l\cap \Gamma_1^{Y_1}}\right)<0, 
\]
a contradiction. Thus $\Gamma_1^{Q_1}$ is a smooth curve of genus $2$. 
On the other hand, by \cite[Lemma 2.1]{MM85}, we have 
\[
10=(-K_{W_1^+})^{\cdot 3}=(-K_{Y_1})^{\cdot 3}-2\left(\left(-K_{Y_1}\cdot\Gamma_1^{Y_1}
\right)-2+1\right)=48-2\left(-K_{Y_1}\cdot \Gamma_1^{Y_1}\right), 
\]
we get the equality $\left(-K_{Y_1}\cdot \Gamma_1^{Y_1}\right)=19$. 
Since $\left(\psi_1^*\sO_{\pr^2_1}(1)\cdot \Gamma_1^{Y_1}\right)=5$ holds, 
we can get $\left(\sO_{Q_1}(1)\cdot \Gamma_1^{Q_1}\right)=7$ and $\left(F_2^{Y_1}
\cdot \Gamma_1^{Y_1}\right)=2$. 
\end{proof}

Let us consider the Sarkisov link 
\[\xymatrix{
& S_1^+ \ar@{}[r]|{\subset}  & X_1^+  \ar@{-->}[rr]^{\chi_1^{-1}} 
\ar@{->}[ld]_-{\tau_1} \ar@{->}[rd]_-{\beta_1^+} 
& & X_1 \ar@{->}[rd]^-{\sigma_1} \ar@{->}[ld]^-{\beta_1} \ar@{}[r]|{\supset} & F'_1 & \\
\Gamma_1^{Q_1} \ar@{}[r]|{\subset} & Q_1 && \bar{X}_1 && X \ar@{}[r]|{\supset}& Z_1
}\]
from the blowup of $Q_1$ along $\Gamma_1^{Q_1}$ as in Theorem 
\ref{thm:double-projection-from-line-converse} 
\eqref{thm:double-projection-from-line-converse2}; the link 
\eqref{equation:iskovskikh-converse} for the case $g=10$. 
In particular, the variety $X$ is a prime Fano threefold of genus $10$ and 
$Z_1\subset X$ is a line. The proof of the following proposition is 
similar to the proof of Proposition \ref{proposition:goback}. 

\begin{proposition}\label{proposition:goback-g10}
\begin{enumerate}
\renewcommand{\theenumi}{\arabic{enumi}}
\renewcommand{\labelenumi}{(\theenumi)}
\item\label{proposition:goback-g101}
The strict transform $(\tau_1)_*^{-1}\left(Z_2^{Q_1}\right)\subset X_1^+$ of 
$Z_2^{Q_1}\subset Q_1$ is not contained in the strict transform $(\chi_1)_*F'_1$. 
Moreover, the rational map $\chi_1^{-1}$ is an isomorphism around the 
neighborhood of $(\tau_1)_*^{-1}\left(Z_2^{Q_1}\right)$. 
Thus we can consider the strict transform $Z_2\subset X$ of 
$Z_2^{Q_1}\subset Q_1$. 
\item\label{proposition:goback-g102}
The curve $Z_2\subset X$ is a line in $X$. Moreover, $Z_1$, $Z_2$ is 
a totally disjoint pair of lines in $X$. 
\end{enumerate}
\end{proposition}

\begin{proof}
\eqref{proposition:goback-g101} 
The strict transform $F_1^{Q_1}:=(\tau_1)_*(\chi_1)_*F'_1\subset Q_1$ is the 
unique element of $|\sO_{Q_1}(2)\otimes I_{\Gamma_1^{Q_1}}|$, where 
$I_{\Gamma_1^{Q_1}}\subset 
\sO_{Q_1}$ is the ideal sheaf corresponds to $\Gamma_1^{Q_1}\subset Q_1$. 
If $Z_2^{Q_1}\subset F_1^{Q_1}$ holds, then 
\[
(\tau_1^+)^{-1}_*\left(\phi_1^*F_1^{Q_1}-F_2^{Y_1}\right)\sim(\chi_{13})_*\sO_{X_0^+}(-1,1),
\]
is effective. Since $h^0\left(U,\sO_U(-1,1)\right)=0$, this leads to a contradiction. 

Note that 
\[
\left(-K_{X_1^+}\cdot (\tau_1)^{-1}_*Z_2^{Q_1}\right)=\left(\sO_{Q_1}(3)\cdot 
Z_2^{Q_1}\right)-\operatorname{length}\left(\sO_{\Gamma_1^{Q_1}\cap Z_2^{Q_1}}\right)=1
\]
by Lemma \ref{lemma:g10converse-a}. 
Thus the curve 
$(\tau_1)_*^{-1}\left(Z_2^{Q_1}\right)$ cannot be contracted by the flopping 
contraction $\beta_1^+\colon X_1^+\to\bar{X}_1$. 
Let us consider the blowup $\phi_1^+\colon W_1\to X_1^+$ of $X_1^+$ along 
$(\tau_1)_*^{-1}\left(Z_2^{Q_1}\right)$. The variety $W_1$ is a smooth weak Fano 
threefold by Proposition \ref{proposition:MDS}. 
In particular, the curve $(\tau_1)_*^{-1}\left(Z_2^{Q_1}\right)$
cannot intersect with any flopping curve of $\chi_1^{-1}$. 
Thus we get the assertion \eqref{proposition:goback-g101}. 

\eqref{proposition:goback-g102} 
Let $Z_2^{X_1}\subset X_1$ be the strict transform of $Z_2^{Q_1}\subset Q_1$. 
By \eqref{proposition:goback-g101} 
(and by Theorem \ref{thm:double-projection-from-line}), 
we have $(F'_1\cdot Z_2^{X_1})=0$ and $(-K_{X_1}\cdot Z_2^{X_1})=1$. 
Thus $Z_2\subset X$ is a line disjoint from $Z_1$. Let $\sigma_2'\colon X_0\to X_1$
be the blowup of $X_1$ along $Z_2^{X_1}$. Note that $X_0$ is nothing but 
the blowup along $Z_1\cup Z_2$. Moreover, the variety $X_0$ is a small $\Q$-factorial 
modification of $X_0^+$. Therefore, again by Proposition \ref{proposition:MDS}, 
the variety $X_0$ is a smooth weak Fano threefold. Therefore, there is no line 
$Z\subset X$ with $Z\cap Z_1\neq\emptyset$ and $Z\cap Z_2\neq\emptyset$. 
Thus we get the assertion \eqref{proposition:goback-g102}. 
\end{proof}

As a consequence, we get the following corollary. 

\begin{corollary}\label{corollary:back-g10}
Let $U$ be the del Pezzo threefold of degree $6$ and rank $2$, 
let $\Gamma\subset U$ be 
a smooth bi-quintic curve of genus $2$ such that the multiplicity of the plane curve 
$\rho_i(\Gamma)$ at any point is at most $2$ for $i\in\{1,2\}$, and let 
$\tau\colon X_0^+\to U$ be the 
blowup along $\Gamma$ with the $\tau$-exceptional divisor $E^+\subset X_0^+$. 
Then $X_0^+$ is a smooth weak Fano threefold and the anti-canonical model 
$\alpha^+\colon X_0^+\to \bar{X}_0$ of $X_0^+$ is small with $\rho(\bar{X}_0)=1$. 
Moreover, we have the following link: 
\begin{equation}\label{equation:ux}
\xymatrix{
&E^+\ar@{}[r]|-{\subset} & 
X_0^+\ar[dl]_-{\tau} \ar@{-->}[rr]^{\chi^{-1}} \ar[dr]_-{\alpha^+} &&
X_0\ar@{}[r]|-{\supset} \ar[dl]^-{\alpha} \ar[dr]^-{\sigma} & F_1\cup F_2 & \\
\Gamma \ar@{}[r]|-{\subset} &U  && \bar{X}_0 && X 
\ar@{}[r]|-{\supset} & Z_1\cup Z_2, 
}
\end{equation}
where $X_0$ is the $(\tau^*K_U)$-flop of $\alpha^+$, i.e., 
\[
X_0=\Proj_{\bar{X}_0}\bigoplus_{m\in\Z_{\geq 0}}(\alpha^+)_*\sO_{X_0^+}
\left(\tau^*(mK_U)\right),
\]
together with the structure morphism $\alpha\colon X_0\to\bar{X}_0$, 
$\chi^{-1}:=\alpha^{-1}\circ\alpha^+$, and 
\[
X=\Proj\bigoplus_{m\in\Z_{\geq 0}}H^0\left(X_0^+, 
m\left(\tau^*\sO_U(5,5)-3E^+\right)\right).
\]
Moreover, the variety $X$ is a prime Fano threefold of genus $10$, the morphism 
$\sigma$ is obtained by the blowup along a totally disjoint pair of lines $Z_1$, $Z_2$ 
in $X$ and the exceptional divisor $F_1+F_2$ of $\sigma$ is the strict transform 
of the unique member of $|\tau^*\sO_U(3,3)-2E^+|$. 
\end{corollary}

\begin{proof}
By Proposition \ref{proposition:totally-blowup} \eqref{proposition:totally-blowup3}, 
the anti-canonical model $\alpha\colon 
X_0\to \bar{X}_0$ of $X_0$ is small. Hence so is $\alpha^+$. 
The remaining assertions are trivial from the link \eqref{equation:big} 
in \S \ref{section:go} for the case $g=10$. 
\end{proof}

\begin{proof}[Proof of Theorem \ref{thm:main-back} \eqref{thm:main-back2}]
Immediately follows from Corollaries \ref{corollary:back-g10} and 
\ref{corollary:go}. See also Remark \ref{remark:aut}. 
\end{proof}

\section{Flopping and flopped curves}\label{section:flop}

In \S \ref{section:flop}, we assume that 
$X$ is a prime Fano threefold of genus $g\in\{12, 10, 9\}$ 
(in \S \ref{subsection:flop1012}, we further assume that $g\in\{12,10\}$)
and let $Z_1$, $Z_2$ be a pair 
of totally disjoint lines in $X$. We follow the notations in \S \ref{section:go}. 
The purpose of this section is to analyze the flopping curves of the elementary flops 
$\chi_{i1}$, 
$\chi_{i2}$, $\chi_{i3}$ in the diagram \eqref{equation:big2}, especially when 
$g\in\{12,10\}$. In \S \ref{section:flop},  
\begin{itemize}
\item
let $B_{11},\dots,B_{1m_1}\subset X_1$ be the flopping curves of 
$\beta_1\colon X_1\to \bar{X}_1$, 
\item
let $B_{21},\dots,B_{2m_2}\subset X_2$ be the flopping curves of 
$\beta_2\colon X_2\to \bar{X}_2$, and 
\item
let $B_{1k}^{X_0}\subset X_0$ ($1\leq k\leq m_1$), $B_{2l}^{X_0}\subset X_0$ 
$(1\leq l\leq m_2)$ be the strict transform of $B_{1k}\subset X_1$, $B_{2l}\subset X_2$, 
respectively. 
\end{itemize}
Note that $B_{i1}^{X_0},\dots,B_{im_i}^{X_0}\subset X_0$ are nothing but the 
flopping curves of $\chi_{i1}$. The goal of \S \ref{section:flop} is to prove 
Theorem \ref{thm:flop-curve}. 

\subsection{General properties of flopping curves}\label{subsection:flop-general}

\begin{lemma}\label{lemma:conic-bridge}
Take $\{i,j\}=\{1,2\}$. Assume that there exists a (smooth) conic $\sC\subset X$ in $X$ 
satisfying $\sC\cap Z_1\neq\emptyset$ and $\sC\cap Z_2\neq\emptyset$. 
For $\V\in\{X_0,X_1,X_2\}$, let $\sC^{\V}\subset\V$ be the strict transform of $\sC$ 
to $\V$. 
\begin{enumerate}
\renewcommand{\theenumi}{\arabic{enumi}}
\renewcommand{\labelenumi}{(\theenumi)}
\item\label{lemma:conic-bridge1}
We have 
\[
\operatorname{length}\left(\sO_{\sC\cap Z_1}\right)=
\operatorname{length}\left(\sO_{\sC\cap Z_2}\right)=1, \quad
\left(-K_{X_0}\cdot \sC^{X_0}\right)=0.
\]
\item\label{lemma:conic-bridge2}
The curve $\sC^{X_0}$ and the locus $\Exc(\chi_{i1})$ in $X_0$ are disjoint. 
In other words, we have $\sC^{X_0}\cap B_{ik}^{X_i}=\emptyset$ for any $1\leq k\leq m_i$. 
In particular, we can define the strict transform $\sC^{W_i}\subset W_i$ of 
$\sC^{X_0}$ on $W_i$. 
\item\label{lemma:conic-bridge3}
The curve $\sC^{W_i}\subset W_i$ is a flopping curve of the elementary 
flop $\chi_{i2}$. In particular, since all flopping curves of $\chi_{i2}$ are mutually disjoint, 
for any other conic $\sC'\subset X$ with $\sC'\cap Z_1\neq\emptyset$ and 
$\sC'\cap Z_2\neq\emptyset$, we have $\sC^{X_0}\cap{\sC'}^{X_0}=\emptyset$. 
\end{enumerate}
\end{lemma}

\begin{proof}
\eqref{lemma:conic-bridge1}
By Lemma \ref{lemma:lengths}, we get 
\[
0\leq\left(-K_{X_0}\cdot\sC^{X_0}\right)=2
-\operatorname{length}\left(\sO_{\sC\cap Z_1}\right)
-\operatorname{length}\left(\sO_{\sC\cap Z_2}\right).
\]
Hence the assertion \eqref{lemma:conic-bridge1} is trivial. 

\eqref{lemma:conic-bridge2}
Since $\Exc(\chi_{i1})$ is equal to the union $\bigcup_{k=1}^{m_i}B_{i k}^{X_0}$, we can 
consider the strict transform $\sC^{W_i}\subset W_i$ of $\sC^{X_0}\subset X_0$. 
Since $\chi_{i1}$ is an $\sO_{\X_i}^{X_0}(1)$-negative elementary flop 
and $\sO_{\X_i}^{W_i}(1)$ is nef on $W_i$, we have 
\[
0=\left(\sO_{\X_i}^{X_0}(1)\cdot\sC^{X_0}\right)\geq
\left(\sO_{\X_i}^{W_i}(1)\cdot\sC^{W_i}\right)\geq 0
\]
by Lemma \ref{lemma:negativity}. 
Thus, again by Lemma \ref{lemma:negativity}, we have 
$\sC^{X_0}\cap\Exc(\chi_{i1})=\emptyset$ and the equality
$\left(\sO_{\X_i}^{W_i}(1)\cdot\sC^{W_i}\right)=0$. 

\eqref{lemma:conic-bridge3}
From the equality $\left(\sO_{\X_i}^{W_i}(1)\cdot\sC^{W_i}\right)=0$, 
the curve $\sC^{W_i}$ is a $K_{W_i}$-trivial curve over $\X_i$ on $W_i$. 
Thus $\sC^{W_i}$ is a flopping curve of $\chi_{i2}$. 
\end{proof}

\begin{lemma}\label{lemma:half-half}
Take $\{i,j\}=\{1,2\}$. 
\begin{enumerate}
\renewcommand{\theenumi}{\arabic{enumi}}
\renewcommand{\labelenumi}{(\theenumi)}
\item\label{lemma:half-half1}
Assume that there exists $1\leq l\leq m_j$ such that the curve $B_{jl}^{X_0}\subset X_0$ 
is disjoint from $\Exc(\chi_{i1})$. Let $B_{jl}^{W_i}\subset W_i$ be the strict transform 
of $B_{jl}^{X_0}\subset X_0$ to $W_i$. Then $\chi_{i2}\colon W_i\dashrightarrow W_i^+$ is 
an isomorphism around a neighborhood of $B_{jl}^{W_i}$. Let $B_{jl}^{W_i^+}\subset W_i^+$ 
be the strict transform of $B_{jl}^{W_i}\subset W_i$ to $W_i^+$. Then the curve 
$B_{jl}^{W_i^+}$ is a flopping curve of the elementary flop $\chi_{i3}\colon W_i^+
\dashrightarrow X_0^+$. 
\item\label{lemma:half-half2}
For any $1\leq l\leq m_j$, we have 
\[
\#\left\{1\leq k\leq m_i\mid B_{ik}^{X_0}\cap B_{jl}^{X_0}\neq\emptyset\right\}\leq 1.
\]
Moreover, if $B_{ik}^{X_0}\cap B_{jl}^{X_0}\neq\emptyset$, then the strict transform 
$B_{jl}^{W_i}\subset W_i$ of $B_{jl}^{X_0}\subset X_0$ to $W_i$ is a flopping curve 
of $\chi_{i2}$. 
\end{enumerate}
\end{lemma}

\begin{proof}
\eqref{lemma:half-half1}
From the assumption, we have 
\begin{eqnarray*}
\left(\sO_{\X_i}^{W_i}(1)\cdot B_{jl}^{W_i}\right)&=&
\left(\sO_{\X_i}^{X_0}(1)\cdot B_{jl}^{X_0}\right)=1, \\
\left(\sO_{\Y_i}^{W_i}(1)\cdot B_{jl}^{W_i}\right)&=&
\left(\sO_{\Y_i}^{X_0}(1)\cdot B_{jl}^{X_0}\right)=0.
\end{eqnarray*}
Thus $B_{jl}^{W_i}\subset W_i$ is not a flopping curve of $\chi_{i2}$. Let 
$B_{jl}^{W_i^+}\subset W_i^+$ be the strict transform of $B_{jl}^{W_i}$. 
Since $\chi_{i2}$ is an $\sO_{\Y_i}^{W_i}(1)$-negative elementary flop 
and $\sO_{\Y_i}^{W_i^+}(1)$ is nef on $W_i^+$, we have 
\[
0=\left(\sO_{\Y_i}^{W_i}(1)\cdot B_{jl}^{W_i}\right)\geq
\left(\sO_{\Y_i}^{W_i^+}(1)\cdot B_{jl}^{W_i^+}\right)\geq 0
\]
by Lemma \ref{lemma:negativity}. Thus, again by Lemma \ref{lemma:negativity}, 
the curve $B_{jl}^{W_i}$ is disjoint from $\Exc(\chi_{i2})$, and the curve $B_{jl}^{W_i^+}$ 
is a $K_{W_i^+}$-trivial 
curve over $\Y_i$. This implies that $B_{jl}^{W_i^+}$ is a flopping curve of 
$\chi_{i3}$. 

\eqref{lemma:half-half2}
Assume that there exist $1\leq k<k'\leq m_i$ such that 
$B_{ik}^{X_0}\cap B_{jl}^{X_0}\neq\emptyset$ and $B_{ik'}^{X_0}\cap B_{jl}^{X_0}
\neq\emptyset$. Set $p_k:=B_{ik}^{X_0}\cap B_{jl}^{X_0}$ and 
$p_{k'}:=B_{ik'}^{X_0}\cap B_{jl}^{X_0}$. If $p_k\neq p_{k'}$, then, by Lemma 
\ref{lemma:negativity}, we have 
\[
\left(\sO_{\X_i}^{W_i}(1)\cdot B_{jl}^{W_i}\right)\leq
\left(\sO_{\X_i}^{X_0}(1)\cdot B_{jl}^{X_0}\right)-2=-1.
\]
This leads to a contradiction since $\sO_{\X_i}^{W_i}(1)$ is nef on $W_i$. 
Thus we must have $p_k=p_{k'}$. However, the three numbers of lines 
$Z_i$, $\sigma_i(B_{ik})$, $\sigma_i(B_{ik'})$ in $X\subset \pr^{g+1}$ 
must lie in a $2$-dimensional linear 
subspace of $\pr^{g+1}$. This contradicts with Theorem \ref{thm:iskovskikh} 
\eqref{thm:iskovskikh2}. 

If $B_{ik}^{X_0}\cap B_{jl}^{X_0}\neq\emptyset$, then the same argument gives 
the equality $\left(\sO_{\X_i}^{W_i}(1)\cdot B_{jl}^{W_i}\right)=0$. 
Thus the curve $B_{jl}^{W_i}\subset W_i$ is a flopping curve of $\chi_{i2}$.
\end{proof}

\begin{lemma}\label{lemma:12goback}
Take $\{i,j\}=\{1,2\}$. 
Let $B^{W_i}\subset W_i$ be a flopping curve of the elementary flop 
$\chi_{i2}\colon W_i\dashrightarrow W_i^+$. 
\begin{enumerate}
\renewcommand{\theenumi}{\arabic{enumi}}
\renewcommand{\labelenumi}{(\theenumi)}
\item\label{lemma:12goback1}
We can define the strict transform $B^{X_0}\subset X_0$ of the curve 
$B^{W_i}\subset W_i$. 
\item\label{lemma:12goback2}
On of the following holds: 
\begin{enumerate}
\renewcommand{\theenumii}{\roman{enumii}}
\renewcommand{\labelenumii}{(\theenumii)}
\item\label{lemma:12goback21}
there exist $1\leq k\leq m_i$ and $1\leq l\leq m_j$ such that the curve $B^{X_0}$ 
in $X_0$ is equal to the curve $B_{jl}^{X_0}$, and we have 
$B_{ik}^{X_0}\cap B_{jl}^{X_0}\neq\emptyset$ (In particular, $Z_1$, $Z_2$ are \emph{not} 
absolutely disjoint.), or
\item\label{lemma:12goback22}
the image $\sigma\left(B^{X_0}\right)\subset X$ of $B^{X_0}$ in $X$ is 
a (smooth) conic intersecting with both $Z_1$ and $Z_2$. 
\end{enumerate}
\end{enumerate}
\end{lemma}

We remark that, for the case \eqref{lemma:12goback21}, the pair $(k,l)$ is uniquely 
determined. This follows from Lemma \ref{lemma:half-half} \eqref{lemma:half-half2}. 

\begin{proof}
Since 
\[
\left(\sO_{\X_i}^{W_i}(1)\cdot B^{W_i}\right)=0,\quad
\left(F_j^{W_i}\cdot B^{W_i}\right)=1,\quad
\left(-K_{W_i}\cdot B^{W_i}\right)=0,
\]
we get $\left(F_i^{W_i}\cdot B^{W_1}\right)=1$ and 
$\left(\sO_X^{W_i}(1)\cdot B^{W_i}\right)=2$. 
We remark that $\chi_{i1}$ is an isomorphism around a neighborhood of $F_j$, and 
$F_j^{W_i}$ and $B^{W_i}$ transversely meet at one point. 
Thus we can consider the strict transform $B^{X_0}\subset X_0$ of $B^{W_i}\subset W_i$, 
and we have $B^{X_0}\not\subset F_i$ since $F_1\cap F_2=\emptyset$. 
Since the elementary flop $\chi_{i1}^{-1}$ is $F_i^{W_i}$-negative, we have 
\[
0\leq\left(F_i\cdot B^{X_0}\right)\leq\left(F_i^{W_i}\cdot B^{W_i}\right)=1 
\]
by Lemma \ref{lemma:negativity}. 

Assume that $\left(F_i\cdot B^{X_0}\right)=0$. Then
we have $B^{X_0}\cap \Exc(\chi_{i1})\neq\emptyset$. Thus 
there exists $1\leq k\leq m_i$ such that $B^{X_0}\cap B_{ik}^{X_0}\neq\emptyset$. 
Since $\left(F_j\cdot B^{X_0}\right)=1$ and $\left(-K_{X_0}\cdot B^{X_0}\right)=0$, 
we have $\left(\sO_X^{X_0}(1)\cdot B^{X_0}\right)=1$. Therefore, there exists 
$1\leq l\leq m_j$ such that $B^{X_0}=B_{jl}^{X_0}$ holds. 

Assume that $\left(F_i\cdot B^{X_0}\right)=1$. Then, by Lemma \ref{lemma:negativity}, 
we have $B^{X_0}\cap\Exc(\chi_{i1})=\emptyset$. Then we have 
$\left(\sO_X^{X_0}(1)\cdot B^{X_0}\right)=\left(\sO_X^{W_i}(1)\cdot B^{W_i}\right)=2$. 
Thus we get the assertion. 
\end{proof}

\begin{corollary}\label{corollary:12goback}
Take $\{i,j\}=\{1,2\}$. 
Assume that $Z_1$, $Z_2$ are \emph{absolutely disjoint}. Set 
\[
\mathscr{C}:=\left\{\sC_1,\dots,\sC_{m_0}\right\}:=
\{\sC\subset X\text{ a conic}\mid\sC\cap Z_1\neq\emptyset\text{ and }
\sC\cap Z_2\neq\emptyset\}
\]
with $m_0:=\#\mathscr{C}$. 
Let $\sC_n^{X_0}\subset X_0$ be the strict transform of $\sC_n\subset X$ to $X_0$ 
for any $1\leq n\leq m_0$. 
\begin{enumerate}
\renewcommand{\theenumi}{\arabic{enumi}}
\renewcommand{\labelenumi}{(\theenumi)}
\item\label{corollary:12goback1}
We have $m_0\geq 1$ and 
\[
m_0\leq\begin{cases}
1 & \text{if }g=12, \\
2 & \text{if }g=10, \\
3 & \text{if }g=9.
\end{cases}
\]
\item\label{corollary:12goback2}
For any $1\leq n\leq m_0$, we can consider the strict transform $\sC_n^{W_i}$ 
of $\sC_n^{X_0}\subset X_0$ to $W_i$. Moreover, we have 
\[
\left\{\text{Flopping curves of }\chi_{i2}\right\}=
\left\{\sC_n^{W_i}\mid 1\leq n\leq m_0\right\}.
\]
\item\label{corollary:12goback3}
For any $1\leq l\leq m_j$, let $B_{jl}^{W_i}\subset W_i$ be the strict transform of 
$B_{jl}^{X_0}\subset X_0$ to $W_i$. Then $\chi_{i2}$ is an isomorphism around 
a neighborhood of $B_{jl}^{W_i}$. The strict transform $B_{jl}^{W_i^+}\subset W_i^+$ 
of $B_{jl}^{W_i}$ to $W_i^+$ is a flopping curve of $\chi_{i3}$. 
\item\label{corollary:12goback4}
If $g=12$ (resp., if $g=10$), then we have $m_1$, $m_2\leq 3$ 
(resp., $m_1$, $m_2\leq 4$). Moreover, the curves 
\[
B_{11}^{X_0},\dots,B_{1m_1}^{X_0},B_{21}^{X_0},\dots,B_{2m_2}^{X_0},
\sC_1^{X_0},\dots,\sC_{m_0}^{X_0}
\]
on $X_0$ are mutually disjoint. 
\end{enumerate}
\end{corollary}

\begin{proof}
\eqref{corollary:12goback1}
Take any flopping curve $B^{W_i}\subset W_i$ of $\chi_{i2}$. Since 
$Z_1$ and $Z_2$ are absolutely disjoint, the stroict transform $B^{X_0}\subset X_0$ 
satisfies that $\sigma(B^{X_0})\in\mathscr{C}$ by Lemma \ref{lemma:12goback} 
\eqref{lemma:12goback2}. In particular, we have $m_0\geq 1$. 
Moreover, for any $\sC\in\mathscr{C}$, we can consider its strict transform 
$\sC^{W_i}\subset W_i$ and is a flopping curve of $\chi_{i2}$ by Lemma 
\ref{lemma:conic-bridge}. From Lemma \ref{lemma:1st-flop} and Example 
\ref{example:3flop}, we get the assertion \eqref{corollary:12goback1}. 

\eqref{corollary:12goback2} Follows immediately from the proof of 
\eqref{corollary:12goback1}. 

\eqref{corollary:12goback3}
By \eqref{corollary:12goback2}, we can consider the strict transform 
$B_{jl}^{W_i^+}\subset W_i^+$ of $B_{jl}^{W_i}\subset W_i$ to $W_i^+$. Moreover, 
since $\sO_{\Y_i}^{W_i^+}(1)$ is nef on $W_i^+$ and $\chi_{i2}$ is 
$\sO_{\Y_i}^{W_i}(1)$-negative, we have 
\[
0=\left(\sO_{\Y_i}^{X_0}(1)\cdot B_{jl}^{X_0}\right)
=\left(\sO_{\Y_i}^{W_i}(1)\cdot B_{jl}^{W_i}\right)
\geq \left(\sO_{\Y_i}^{W_i^+}(1)\cdot B_{jl}^{W_i^+}\right)\geq 0
\]
by Lemma \ref{lemma:negativity}. Thus we get the assertion \eqref{corollary:12goback3}. 

\eqref{corollary:12goback4}
If $g\in\{12,10\}$, then the flopping curves of $\chi_{i3}$ are mutually disjoint. 
Moreover, by Example \ref{example:3flop} and Remark \ref{remark:conic-bundle}, 
the number of flopping curve of $\chi_{i3}$ is equal to 
\[
\begin{cases}
\#\left(C_i\cap\Gamma_i^{Q_i}\right)_{\operatorname{red}}\leq 3 & \text{if }g=12, \\
\#\operatorname{Sing}(\rho_i(\Gamma))\leq 4 & \text{if }g=10,
\end{cases}
\]
where $\operatorname{Sing}(\rho_i(\Gamma))$ is the set of singular points 
of $\rho_i(\Gamma)$. (Since the arithmetic genus of plane quintics are equal to $6$, 
the number of singular points is at most $4=6-2$.)
Since $\chi_{i1}$ (resp., $\chi_{i2}$) is an isomorphism around a neighborhood of 
$B_{jl}^{X_0}$ (resp., $B_{jl}^{W_i}$), the curves 
$B_{j1}^{X_0},\dots,B_{jm_j}^{X_0}$ are mutually disjoint. 
Thus the assertion \eqref{corollary:12goback4} follows from the assumption 
$Z_1$ and $Z_2$ are absolutely disjoint and from Lemma \ref{lemma:conic-bridge}
\eqref{lemma:conic-bridge2}. 
\end{proof}

\begin{remark}\label{remark:IKTT}
The proof of Corollary \ref{corollary:12goback}, especially the proof for the 
existence of the conic $\sC_1$ for the case $g=12$, 
is similar to the proof of \cite[Lemma 5.12]{IKTT}. 
We note that $m_0=1$ when $Z_1$ and $Z_2$ are absolutely disjoint and $g=12$. 
\end{remark}

We are ready to prove Proposition \ref{proposition:disjoint-lines}. 

\begin{proof}[Proof of Proposition \ref{proposition:disjoint-lines}]
Let us set $Z_1:=Z$. Take any $\left[Z_2\right]\in\Sigma(X)\setminus\D_{Z_1}$, 
where $\D_{Z_1}\subset \Sigma(X)$ be as in Lemma \ref{lemma:totally-disjoint}. 
Then $Z_1$, $Z_2$ are absolutely disjoint. Hence we can apply Corollary 
\ref{corollary:12goback}, and Proposition \ref{proposition:disjoint-lines} is just 
an immediate corollary of Corollary \ref{corollary:12goback}.
\end{proof}

\subsection{The cases $g=12$ and $g=10$}\label{subsection:flop1012}

In \S \ref{subsection:flop1012}, we further assume that $g\in\{12,10\}$. 
Moreover, after perturbation, by Lemma \ref{lemma:half-half} \eqref{lemma:half-half2}, we 
may assume that there exists $0\leq\bar{m}\leq\min\{m_1,m_2\}$ such that, 
for any $1\leq k\leq m_1$ and for any $1\leq l\leq m_2$, we have 
\[
B_{1k}^{X_0}\cap B_{2l}^{X_0}\neq\emptyset\Leftrightarrow 1\leq k=l\leq\bar{m}.
\]
Take $\{i,j\}=\{1,2\}$. We set 
\begin{eqnarray*}
\bar{\mathscr{B}}_i^{X_0}&:=&\left\{B_{ik}^{X_0}\subset X_0
\mid 1\leq k\leq \bar{m}\right\}, \\
\mathscr{B}_i^{X_0}&:=&\left\{B_{ik}^{X_0}\subset X_0
\mid \bar{m}+1\leq k\leq m_i\right\}.
\end{eqnarray*}
Moreover, we set 
\begin{eqnarray*}
\mathscr{C}&:=&\left\{\sC_n\subset X_0\mid 1\leq n\leq m_0\right\}\\
&:=&\left\{\sC\subset X\text{ a conic}\mid \sC\cap Z_1\neq\emptyset\text{ 
and }\sC\cap Z_2\neq\emptyset\right\}, \\
\mathscr{C}^{X_0}&:=&\left\{\sC_n^{X_0}:=\sigma^{-1}_*\sC_n\subset X_0\mid 
1\leq n\leq m_0\right\},\\
\mathscr{C}^{W_i}&:=&\left\{\sC_n^{W_i}:=(\chi_{i1})_*\sC_n^{X_0}\subset W_i\mid 
1\leq n\leq m_0\right\}
\end{eqnarray*}
as in Corollary \ref{corollary:12goback}, where $m_0:=\#\mathscr{C}$. 
By Proposition \ref{proposition:disjoint-lines} and Lemma \ref{lemma:conic-bridge} 
\eqref{lemma:conic-bridge2}, the $1$-dimensional schemes 
\[
B_{11}^{X_0}\cup B_{21}^{X_0},\dots,B_{1\bar{m}}^{X_0}\cup B_{2\bar{m}}^{X_0},
B_{1\bar{m}+1}^{X_0},\dots,B_{1m_1}^{X_0},B_{2\bar{m}+1}^{X_0},\dots,B_{2m_2}^{X_0}, 
\sC_1^{X_0},\dots,\sC_{m_0}^{X_0}
\]
are mutually disjoint. 

The purpose of \S \ref{section:flop} is to prove the following theorem.

\begin{thm}\label{thm:flop-curve}
Assume that $g\in\{12, 10\}$. Take $\{i,j\}=\{1,2\}$. 
\begin{enumerate}
\renewcommand{\theenumi}{\arabic{enumi}}
\renewcommand{\labelenumi}{(\theenumi)}
\item\label{thm:flop-curve1}
We have 
\[
\left\{\text{Flopping curves of }\chi_{i1}\right\}=\bar{\mathscr{B}}_i^{X_0}\sqcup
\mathscr{B}_i^{X_0}. 
\]
Thus we can set 
\begin{eqnarray*}
\bar{\mathscr{B}}_j^{W_i}&:=&\left\{B_{jl}^{W_i}:=(\chi_{i1})_*B_{jl}^{X_0}\subset W_i 
\mid 1\leq l\leq \bar{m}\right\}, \\
\mathscr{B}_j^{W_i}&:=&\left\{B_{jl}^{W_i}:=(\chi_{i1})_*B_{jl}^{X_0}\subset W_i 
\mid \bar{m}+1\leq l\leq m_j\right\}, \\
\mathscr{C}^{W_i}&:=&\left\{\sC_n^{W_i}:=(\chi_{i1})_*\sC_n^{X_0}\subset W_i 
\mid 1\leq n\leq m_0\right\}. 
\end{eqnarray*}
We also set 
\begin{eqnarray*}
\bar{\mathscr{B}}_i^{+W_i}&:=&\left\{B_{ik}^{+W_i}\subset W_i 
\mid 1\leq k\leq \bar{m}\right\}, \\
\mathscr{B}_i^{+W_i}&:=&\left\{B_{ik}^{+W_i}\subset W_i 
\mid \bar{m}+1\leq k\leq m_i\right\}, 
\end{eqnarray*}
where $B_{ik}^{+W_i}\subset W_i$ is the flopped curve of $B_{ik}^{X_0}\subset X_0$ with 
respects to the elementary flop $\chi_{i1}$ for any $1\leq k\leq m_i$. 
\item\label{thm:flop-curve2}
We have 
\[
\left\{\text{Flopping curves of }\chi_{i2}\right\}=\bar{\mathscr{B}}_j^{W_i}\sqcup
\mathscr{C}^{W_i}. 
\]
Thus we can set 
\begin{eqnarray*}
\bar{\mathscr{B}}_i^{+W^+_i}&:=&\left\{B_{ik}^{+W_i^+}:=(\chi_{i2})_*B_{ik}^{+W_i}\subset W^+_i 
\mid 1\leq k\leq \bar{m}\right\}, \\
\mathscr{B}_i^{+W^+_i}&:=&\left\{B_{ik}^{+W^+_i}:=(\chi_{i2})_*B_{ik}^{+W_i}\subset W^+_i 
\mid \bar{m}+1\leq k\leq m_i\right\}, \\
\mathscr{B}_j^{W^+_i}&:=&\left\{B_{jl}^{W^+_i}:=(\chi_{i2})_*B_{jl}^{W_i}\subset W^+_i 
\mid \bar{m}+1\leq l\leq m_j\right\}. 
\end{eqnarray*}
We also set 
\begin{eqnarray*}
\bar{\mathscr{B}}_j^{-W^+_i}&:=&\left\{B_{jl}^{-W^+_i}\subset W^+_i 
\mid 1\leq l\leq \bar{m}\right\}, \\
\mathscr{C}^{+W^+_i}&:=&\left\{\sC_n^{+W^+_i}\subset W^+_i 
\mid 1\leq n\leq m_0\right\}, 
\end{eqnarray*}
where $B_{jl}^{-W^+_i}\subset W^+_i$ (resp., $\sC_n^{+W^+_i}\subset W^+_i$) 
is the flopped curve of $B_{jl}^{W_i}\subset W_i$ (resp., $\sC_n^{W_i}\subset W_i$) with 
respects to the elementary flop $\chi_{i2}$ for any $1\leq l\leq \bar{m}$ (resp., 
for any $1\leq n\leq m_0$).
\item\label{thm:flop-curve3}
We have 
\[
\left\{\text{Flopping curves of }\chi_{i3}\right\}=\bar{\mathscr{B}}_i^{+W^+_i}\sqcup
\mathscr{B}_j^{W^+_i}. 
\]
Moreover, we have the following: 
\begin{enumerate}
\renewcommand{\theenumii}{\roman{enumii}}
\renewcommand{\labelenumii}{(\theenumii)}
\item\label{thm:flop-curve31}
For any $1\leq n\leq m_0$, we have
\[
(\chi_{i3})_*\sC_n^{+W_i^+}=(\chi_{j3})_*\sC_n^{+W_j^+}.
\]
\item\label{thm:flop-curve32}
For any $1\leq k\leq \bar{m}$, the curve 
\[
(\chi_{j3})_*B_{ik}^{- W_j^+}\subset X_0^+
\]
is equal to the flopped curve of $B_{ik}^{+ W_i^+}\subset W_i^+$ with respects to 
the elementary flop $\chi_{i3}$. 
\item\label{thm:flop-curve33}
For any $\bar{m}+1\leq l\leq m_j$, the curve 
\[
(\chi_{j3})_*B_{jl}^{+ W_j^+}\subset X_0^+
\]
is equal to the flopped curve of $B_{jl}^{W_i^+}\subset W_i^+$ with respects to 
the elementary flop $\chi_{i3}$. 
\end{enumerate}
\item\label{thm:flop-curve4}
Let us set 
\begin{eqnarray*}
\bar{\mathscr{B}}_i^{-X_0^+}&:=&\left\{B_{ik}^{-X_0^+}:=(\chi_{j3})_*B_{ik}^{-W_j^+}
\subset X_0^+ 
\mid 1\leq k\leq \bar{m}\right\}, \\
\mathscr{B}_i^{+X_0^+}&:=&\left\{B_{ik}^{+X_0^+}:=(\chi_{i3})_*B_{ik}^{+W^+_i}\subset X_0^+ 
\mid \bar{m}+1\leq k\leq m_i\right\}, \\
\mathscr{C}^{+X_0^+}&:=&\left\{\sC_n^{+X_0^+}:=(\chi_{i3})_*\sC_n^{+W_i^+}\subset X_0^+ 
\mid 1\leq n\leq m_0\right\}. 
\end{eqnarray*}
Then 
\[
\begin{cases}
\mathscr{C}^{+X_0^+}=\left\{\tau^{-1}_*\sC^+\subset X_0^+\mid\sC^+\subset\hat{Q}
\colon\text{ bi-line with }
\operatorname{length}\left(\sO_{\hat{C}\cap\sC^+}\right)=2\right\} & \text{if }g=12, \\
\mathscr{C}^{+X_0^+}=\left\{\tau^{-1}_*\sC^+\subset X_0^+\mid\sC^+\subset 
U\colon\text{ bi-line with }
\operatorname{length}\left(\sO_{\Gamma\cap\sC^+}\right)=4\right\} & \text{if }g=10. 
\end{cases}
\]
\item\label{thm:flop-curve5}
We have the following: 
\begin{enumerate}
\renewcommand{\theenumii}{\roman{enumii}}
\renewcommand{\labelenumii}{(\theenumii)}
\item\label{thm:flop-curve51}
Assume that $g=12$. Then we have 
\[
\bar{m}+m_0=1, \quad m_j=\#\left(C_i\cap\Gamma_i^{Q_i}\right)_{\operatorname{red}}, 
\]
where we recall that $C_i\subset Q_i$ is the center of the blowup $\psi_i\colon 
Y_i\to Q_i$ and $\Gamma_i^{Q_i}\subset Q_i$ is the center of the blowup 
$\rho_i\colon\hat{Q}\to Q_i$. (In particular, we have $1\leq m_j\leq 3$.)
\item\label{thm:flop-curve52}
Assume that $g=10$. Then we have 
\[
1\leq\bar{m}+m_0\leq 2, \quad m_j=\#\Sing\left(\rho_i(\Gamma)\right), 
\]
where we recall that $\Gamma\subset U$ is the center of the blowup $\tau\colon 
X_0^+\to U$. (In particular, we have $1\leq m_j\leq 4$.)
\end{enumerate}
\end{enumerate}
\end{thm}

\begin{proof}
\eqref{thm:flop-curve1} is trivial. \eqref{thm:flop-curve2} follows from 
Lemmas \ref{lemma:conic-bridge} \eqref{lemma:conic-bridge3}, 
\ref{lemma:half-half} \eqref{lemma:half-half2} and \ref{lemma:12goback} 
\eqref{lemma:12goback2}. 

\eqref{thm:flop-curve3}
Any curve in $\mathscr{B}_j^{W^+_i}$ is a flopping curve of $\chi_{i3}$ by 
Lemma \ref{lemma:half-half} \eqref{lemma:half-half1}. Take any 
$B_{ik}^{+W_i^+}\in\bar{\mathscr{B}}_i^{+W^+_i}$. Note that $B_{ik}^{+W_i}$ intersects with 
$B_{jk}^{W_i}\subset\Exc(\chi_{i2})$. By Lemma \ref{lemma:negativity} and Theorem 
\ref{thm:kollar}, we have 
\[
0\leq\left(\sO_{\Y_i}^{W_i^+}(1)\cdot B_{ik}^{+W_i^+}\right)<
\left(\sO_{\Y_i}^{W_i}(1)\cdot B_{ik}^{+W_i}\right)
=-\left(\sO_{\Y_i}^{X_0}(1)\cdot B_{ik}^{X_0}\right)=1. 
\]
This implies that the curve $B_{ik}^{+W_i^+}\subset W_i^+$ is a flopping curve of 
$\chi_{i3}$. 

Conversely, take any flopping curve $B^{W_i^+}\subset W_i^+$ of $\chi_{i3}$. 
Note that 
\[
\left(\sO_{\Y_i}^{W_i^+}(1)\cdot B^{W_i^+}\right)=\left(-K_{W_i^+}\cdot B^{W_i^+}\right)=0. 
\]
Moreover, since 
\[
\left(S_i^{W_i^+}\cdot B^{W_i^+}\right)=\begin{cases}
1 & \text{if }g=12, \\
2 & \text{if }g=10, \\
\end{cases}\]
we have 
\[
\left(\sO_X^{W_i^+}(1)\cdot B^{W_i^+}\right)=1,\quad
\left(F_i^{W_i^+}\cdot B^{W_i^+}\right)=0,\quad
\left(F_j^{W_i^+}\cdot B^{W_i^+}\right)=1.
\]
Since $\left(\sO_{\X_i}^{W_i^+}(1)\cdot B^{W_i^+}\right)=1>0$, we can consider the 
strict transform $B^{W_i}\subset W_i$ of $B^{W_i^+}\subset W_i^+$ to $W_i$. 
Note that, by Lemma \ref{lemma:negativity}, we have 
\[
0\leq\left(\sO_{\bar{X}_i}^{W_i}(1)\cdot B^{W_i}\right)\leq
\left(\sO_{\bar{X}_i}^{W^+_i}(1)\cdot B^{W^+_i}\right)=1.
\]
Thus we have one of: 
\begin{enumerate}
\renewcommand{\theenumi}{\alph{enumi}}
\renewcommand{\labelenumi}{(\theenumi)}
\item\label{negativity-flop1}
the curve $B^{W_i}$ is disjoint from $\Exc(\chi_{i2})$, does not contain 
$\Exc(\chi_{i1}^{-1})$, and we can consider the strict transform $B^{X_0}\subset X_0$ 
of $B^{W_i}$, or 
\item\label{negativity-flop2}
the curve $B^{W_i}$ intersects with $\Exc(\chi_{i2})$, and is a flopped curve 
of $\chi_{i1}$. 
\end{enumerate}

Let us consider the case \eqref{negativity-flop1}. By Lemma \ref{lemma:negativity}, 
we have 
\[
0\leq\left(\sO_{\bar{X}_j}^{X_0}(1)\cdot B^{X_0}\right)\leq
\left(\sO_{\bar{X}_j}^{W_i}(1)\cdot B^{W_i}\right)=0.
\]
Therefore, the curve $B^{X_0}$ is disjoint from $\Exc(\chi_{i1})$ and must be equal 
to $B_{jl}^{X_0}$ for some $1\leq j\leq m_j$. This implies that $\bar{m}+1\leq l\leq m_j$. 
Thus we get $B^{W_i^+}\in\mathscr{B}_j^{W^+_i}$. 

Let us consider the case \eqref{negativity-flop2}. 
There exists $1\leq k\leq m_i$ such that $B^{W_i}=B_{ik}^{+W_i}$. If $k\geq\bar{m}+1$, 
then $B_{ik}^{+W_i}$ is disjoint from $\Exc(\chi_{i2})$ by \eqref{thm:flop-curve2}, 
a contradiction. Thus we have $1\leq k\leq \bar{m}$, i.e., we have 
$B^{W_i^+}\in\bar{\mathscr{B}}_i^{+W^+_i}$. 

\eqref{thm:flop-curve31}
Note that $\chi_{i1}$ is an isomorphism around a neighborhood of $\sC_n^{X_0}$, 
and $\chi_{i3}$ is an isomorphism around a neighborhood of $\sC_n^{+W_i^+}$. 
Thus, the curve $(\chi_{i3})_*\sC_n^{+W_i^+}\subset X_0^+$ is equal to the fiber 
$(\alpha^+)^{-1}\left(\alpha(\sC_n^{X_0})\right)$. Thus we get the assertion 
\eqref{thm:flop-curve31}. 

\eqref{thm:flop-curve32}
By chasing the flops $X_0\dashrightarrow W_i\dashrightarrow W_i^+\dashrightarrow 
X_0^+$, the fiber 
\[
(\alpha^+)^{-1}\left(\alpha(B_{ik}^{X_0}\cup B_{jk}^{X_0})\right)
\]
consists of two numbers of curves. Since $\Nef(X_0^+)=\R_{\geq 0}\left[-K_{\Y}\right]
+\Nef(\Y_1)+\Nef(\Y_2)$, any flopped curve of $\chi_{i3}$ cannot be a flopped curve 
of $\chi_{j3}$. Thus we get the assertion \eqref{thm:flop-curve32}. 

\eqref{thm:flop-curve33}
As in the proof of \eqref{thm:flop-curve31}, we have 
\[
(\chi_{j3})_*B_{jl}^{+W_j^+}=(\alpha^+)^{-1}\left(\alpha(B_{jl}^{X_0})\right). 
\]
Thus the assertion \eqref{thm:flop-curve33} is also trivial. 

\eqref{thm:flop-curve4}
Take any $1\leq n \leq m_0$. Note that 
\[
\left(\sO_X^{X_0^+}(1)\cdot\sC_n^{+X_0^+}\right)=-2, \quad
\left(F_1^{X_0^+}\cdot\sC_n^{+X_0^+}\right)=
\left(F_2^{X_0^+}\cdot\sC_n^{+X_0^+}\right)=-1
\]
by Theorem \ref{thm:kollar}. Thus we get 
\[
\left(\sO_{\Y_1}^{X_0^+}(1)\cdot\sC_n^{+X_0^+}\right)
=\left(\sO_{\Y_2}^{X_0^+}(1)\cdot\sC_n^{+X_0^+}\right)=1, \quad
\left(E^+\cdot\sC_n^{+X_0^+}\right)=\begin{cases}
2 & \text{if }g=12, \\
4 & \text{if }g=10.\end{cases}\]

Conversely, take any curve $\tilde{\sC}:=\tau^{-1}_*\sC^+\subset X_0^+$ in the 
right hand side of \eqref{thm:flop-curve4}. 
 Note that 
\[
\left(\sO_X^{X_0^+}(1)\cdot\tilde{\sC}\right)=-2, \quad
\left(F_1^{X_0^+}\cdot\tilde{\sC}\right)=
\left(F_2^{X_0^+}\cdot\tilde{\sC}\right)=-1.
\]
Since $\left(\sO_{\Y_1}^{X_0^+}(1)\cdot\tilde{\sC}\right)=1$, we can consider the 
strict transform ${\tilde{\sC}}^{W_1^+}\subset W_1^+$ of $\tilde{\sC}$. By Lemma 
\ref{lemma:negativity}, we get 
\[
0\leq\left(\sO_{\X_1}^{W_1^+}(1)\cdot{\tilde{\sC}}^{W_1^+}\right)\leq
\left(\sO_{\X_1}^{X_0^+}(1)\cdot{\tilde{\sC}}\right)=0.
\]
This implies that the curve ${\tilde{\sC}}^{W_1^+}\subset W_1^+$ is a flopped curve 
of $\chi_{12}$, and is disjoint from $\Exc(\chi_{13})$. Therefore, from 
\eqref{thm:flop-curve2} and \eqref{thm:flop-curve3}, we have 
$\tilde{\sC}\in\mathscr{C}^{+X_0^+}$. 

\eqref{thm:flop-curve5}
This is an immediate consequence of the assertions 
\eqref{thm:flop-curve2} and \eqref{thm:flop-curve3}. 
\end{proof}

\begin{corollary}\label{corollary:flop-alpha}
We follow the notation in Theorem \ref{thm:flop-curve}. Then we have 
\begin{eqnarray*}
\Exc(\alpha)&=&\bigcup_{k=1}^{m_1}B_{1k}^{X_0}\cup\bigcup_{l=1}^{m_2}B_{2l}^{X_0}\cup
\bigcup_{n=1}^{m_0}\sC_n^{X_0}, \\
\Exc(\alpha^+)&=&\bigcup_{k=1}^{\bar{m}}\left(B_{1k}^{-X_0^+}\cup B_{2k}^{-X_0^+}\right)
\cup\bigcup_{k=\bar{m}+1}^{m_1}B_{1k}^{+X_0^+}\cup
\bigcup_{l=\bar{m}+1}^{m_2}B_{2l}^{+X_0^+}\cup\bigcup_{n=1}^{m_0}\sC_n^{+X_0^+}.
\end{eqnarray*}
\end{corollary}

\begin{proof}
Obviously, the set $\Exc(\alpha)$ (resp., the set $\Exc(\alpha^+)$) 
contains those curves. 
On the other hand, by Theorem \ref{thm:flop-curve}, if we exclude the union of 
those curves, then the rational map $(\alpha^{+1})^{-1}\circ\alpha\colon X_0
\dashrightarrow X_0^+$ (resp., the rational map $\alpha^{-1}\circ\alpha^+\colon X_0^+
\dashrightarrow X_0$)
is an isomorphism onto its image. 
Thus we get the assertion from Lemma \ref{lemma:negativity} \eqref{lemma:negativity1}. 
\end{proof}

\part{Applications in genus twelve}\label{part:application}

\section{Configurations of lines}\label{section:configuration}

In this section, we discuss several possibilities for the configurations of lines 
in prime Fano threefolds of genus $12$. 

\begin{lemma}\label{lemma:3lines}
Let $X$ be a prime Fano threefold of genus $12$. 
\begin{enumerate}
\renewcommand{\theenumi}{\arabic{enumi}}
\renewcommand{\labelenumi}{(\theenumi)}
\item\label{lemma:3lines1}
Assume that there exist distinct lines $\left[Z'_1\right], \left[Z'_2\right], 
\left[Z'_3\right]\in\Sigma(X)$ in $X$ with 
$Z'_1\cap Z'_2\neq\emptyset$, 
$Z'_2\cap Z'_3\neq\emptyset$ and 
$Z'_3\cap Z'_1\neq\emptyset$. 
Then there exists a point $p\in X$ such that $p\in Z'_i$ for all $i=1,2,3$. 
\item\label{lemma:3lines2}
For any point $p\in X$, the number of lines in $X$ passing through $p$ is at most $3$. 
\item\label{lemma:3lines3}
If there exists a point $p\in X$ and distinct lines $\left[Z'_1\right], \left[Z'_2\right], 
\left[Z'_3\right]\in\Sigma(X)$ in $X$ with $p\in Z'_i$ for all $i=1,2,3$, 
then we have 
$\sN_{Z'_i/X}\cong\sO_{\pr^1}\oplus\sO_{\pr^1}(-1)$
for all $i=1,2,3$. 
\end{enumerate}
\end{lemma}

\begin{proof}
\eqref{lemma:3lines1} is trivial from Theorem \ref{thm:iskovskikh} \eqref{thm:iskovskikh2}. 
Let us consider \eqref{lemma:3lines2} and \eqref{lemma:3lines3}. 
For a point $p\in X$, let $\left\{Z'_i\right\}_{1\leq i\leq m}$ be the 
set of lines in $X$ passing through $p$. Assume that $m\geq 3$. 
Consider the Sarkisov link
\[
\xymatrix{
 & F' \ar@{}[r]|{\subset}  & X'  \ar@{-->}[rr]^{\chi} 
\ar@{->}[ld]_-{\sigma} \ar@{->}[rd]_-{\beta} 
& & X^+ \ar@{->}[rd]^-{\tau} \ar@{->}[ld]^-{\beta^+} \ar@{}[r]|{\supset} & S^+ & \\
 Z'_1 \ar@{}[r]|{\subset} & X && \bar{X} && V \ar@{}[r]|{\supset}& \Gamma
}
\]
from the blowup of $X$ along $Z'_1$ as in \eqref{equation:iskovskikh} for the case 
$g=12$. Let $(Z'_i)^{X'}\subset X'$ be the strict transform of $Z'_i\subset X$ for any 
$2\leq i\leq m$, and set $l:=\sigma^{-1}(p)\subset X'$. 
The curves $(Z'_i)^{X'}$ are flopping curves of $\beta$ intersecting with $l$. 
Moreover, if $\sN_{Z'_1/X}\cong\sO_{\pr^1}(1)\oplus\sO_{\pr^1}(-2)$, then 
the $(-3)$-curve in $F'$ intersects with $l$. 
Note that, by Proposition \ref{proposition:disjoint-lines}, 
all flopping curves of $\beta$ are mutually disjoint. 
Thus, if $m\geq 4$, or if $m=3$ and 
$\sN_{Z'_1/X}\cong\sO_{\pr^1}(1)\oplus\sO_{\pr^1}(-2)$, 
then the curve $l\subset X'$ intersects with $\Exc(\beta)\subset X'$ at least 
$3$ numbers of points. By Lemma \ref{lemma:negativity}, since 
the elementary flop $\chi$ is $\chi^{-1}_*\tau^*\sO_V(1)$-negative 
and $\tau^*\sO_V(1)$ is nef, we get
\[
2=\left(\chi^{-1}_*\tau^*\sO_V(1)\cdot l\right)
\geq\left(\tau^*\sO_V(1)\cdot \chi_*l\right)+3\geq 3, 
\]
a contradiction. Thus we get the assertions 
\eqref{lemma:3lines2} and \eqref{lemma:3lines3}. 
\end{proof}

\begin{remark}\label{remark:3lines}
There exists a prime Fano threefold $X$ of genus $12$ and a point $p\in X$ 
such that there are $3$ numbers of lines in $X$ passing through $p$. 
See the following example. 
\end{remark}

\begin{example}\label{example:3lines}
Let $Q$ be the $3$-dimensional smooth hyperquadric in $\pr^4$, and let 
$F^Q, (F')^Q\subset Q$ be smooth hyperplane sections such that 
$F^Q\cap (F')^Q=f_1\cup f_2$ holds, where $f_1$, $f_2$ are distinct lines in $Q\subset 
\pr^4$. Set $q:=f_1\cap f_2$. We fix isomorphisms $F^Q\cong\pr^1\times\pr^1$, 
$(F')^Q\cong\pr^1\times\pr^1$ satisfying $f_1\in|\sO_{\pr^1\times\pr^1}(1,0)|$ and 
$f_2\in|\sO_{\pr^1\times\pr^1}(0,1)|$.
Let $C^Q\subset F^Q$ be a general smooth curve satisfying 
$C^Q\in|\sO_{\pr^1\times\pr^1}(2,1)|$ under the isomorphism 
$F^Q\cong\pr^1\times\pr^1$, and set $q_0:=f_1\cap C^Q$ and 
$\{q_1,q_2\}:=f_2\cap C^Q$. Since $C^Q$ is taken to be general in the complete linear 
system, the points $q, q_0,q_1,q_2$ are mutually distinct. 
Take a smooth curve $\Gamma^Q\subset (F')^Q$ such that 
$\Gamma^Q\in|\sO_{\pr^1\times\pr^1}(3,1)|$ under the isomorphism 
$(F')^Q\cong\pr^1\times\pr^1$ satisfying $f_1\cap\Gamma^Q=q_0$ and 
$f_2\cap\Gamma^Q=\{q_1,q_2,q'\}$ with $q'\neq q_1,q_2$. Consider the Sarkisov link 
\[
\xymatrix{
 & E^Y \ar@{}[r]|{\subset} & Y 
\ar@{}[r]|{\supset} \ar@{->}[ld]_{\psi} \ar@{->}[rd]^{\phi}
& F^Y & \\
 C^Q \ar@{}[r]|{\subset} & Q && 
 V \ar@{}[r]|{\supset} & Z^V
}
\]
from the blowup of $Q$ along $C^Q$ as in \eqref{equation:fujita-converse}. 
By Theorem \ref{thm:dP5} and Proposition \ref{proposition:fujita}, the divisor 
$F^Y\subset Y$ is the strict transform of $F^Q\subset Q$ on $Y$, $F^Y\cong F^Q$ 
under the restriction of $\psi$, and the restriction morphism $\phi|_Y\colon F^Y\to Z^V$ 
is nothing but the projection $\pr^1\times\pr^1\to\pr^1$
to the second projective line under the fixed isomorphism $F^Q\cong\pr^1\times\pr^1$. 
Moreover, if we set 
\begin{eqnarray*}
(F')^Y&:=&\psi^{-1}_*(F')^Q,\quad 
\psi':=\psi|_{(F')^Y}\colon (F')^Y\to (F')^Q, \\
(F')^V&:=&\phi_*((F')^Y), \quad
\phi':=\phi|_{(F')^Y}\colon {(F')^Y}\to (F')^V, 
\end{eqnarray*}
then $(F')^V\subset V$ is a normal surface such that $(F')^V\in|\sO_V(1)|$ with 
$Z^V\subset (F')^V$. Since $(F')^Q$ and $C^Q$ transversely intersect at the points 
$q_0,q_1,q_2$, the morphism $\psi'$ is nothing but the blowup of $(F')^Q$ at 
the points $q_0,q_1,q_2$. Let $\mathbf{e}_i^Y\subset (F')^Y$ be the $(-1)$-curves 
over the point $q_i$ ($i=0,1,2$). 
Set $f_2^Y:=\psi^{-1}_*f_2$. Then the curve $f_2^Y\subset (F')^Y$ is a $(-2)$-curve 
in $(F')^Y$ and the morphism $\phi'$ contracts only the curve $f_2^Y$. 
Therefore, the surface $(F')^V\subset V$ is a del Pezzo surface of degree $5$ having 
only $A_1$ singular point at $p':=\phi'(f_2^Y)\in (F')^V$. 
Set $\Gamma^Y:=\psi^{-1}_*\Gamma^V$ and $\Gamma:=\phi_*\Gamma^Y$. 
Note that, the curve $\Gamma^Y\subset Y$ satisfies that,  
\begin{eqnarray*}
\Gamma^Y&\in&\left|(\psi')^*\sO_{\pr^1\times\pr^1}(3,1)-\mathbf{e}_0^Y
-\mathbf{e}_1^Y-\mathbf{e}_2^Y\right|, \\
-K_{(F')^Y}&\sim&(\psi')^*\sO_{\pr^1\times\pr^1}(2,2)-\mathbf{e}_0^Y
-\mathbf{e}_1^Y-\mathbf{e}_2^Y.
\end{eqnarray*}
Thus we have 
\begin{eqnarray*}
&&H^1\left((F')^Y,-K_{(F')^Y}-\Gamma^Y\right)
\cong H^1\left((F')^Y,(\psi')^*\sO_{\pr^1\times\pr^1}(-1,1)\right)\\
&\cong&\left(H^1\left(\pr^1, \sO(-1)\right)\otimes H^0\left(\pr^1,\sO(1)\right)\right)
\oplus\left(H^0\left(\pr^1, \sO(-1)\right)\otimes H^1\left(\pr^1,\sO(1)\right)\right)
=0.
\end{eqnarray*}
This implies that the restriction homomorphism 
\[
H^0\left((F')^Y, -K_{(F')^Y}\right)\to H^0\left(\Gamma^Y, -K_{(F')^Y}|_{\Gamma^Y}\right)
\]
is surjective. Thus, the image $\Gamma\subset V\subset \pr^6$ of $\Gamma^Y$ is 
a twisted (hence smooth) quintic rational curve passing through $p'$. 
For $i=1,2$, let $(Z'_i)^V\subset V$ be the image of $\mathbf{e}_i^Y\subset Y$. 
Then, since $q'\in f_2$, the curve $(Z'_i)^V$ is a line in $V$ with 
$\operatorname{length}\left(\sO_{\Gamma\cap (Z'_i)^V}\right)\geq 2$ and 
$p'\in (Z'_i)^V$. 
Let us consider the Sarkisov link 
\[\xymatrix{
& S^+ \ar@{}[r]|{\subset}  & X^+  \ar@{-->}[rr]^{\chi^{-1}} 
\ar@{->}[ld]_-{\tau} \ar@{->}[rd]_-{\beta^+} 
& & X' \ar@{->}[rd]^-{\sigma} \ar@{->}[ld]^-{\beta} \ar@{}[r]|{\supset} & F' & \\
\Gamma \ar@{}[r]|{\subset} & V && \bar{X} && X \ar@{}[r]|{\supset}& Z_0
}\]
from the blowup of $V$ along $\Gamma$ as in \eqref{equation:iskovskikh-converse} 
for the case $g=12$. 
In particular, the variety $X$ is a prime Fano threefold of genus $12$, 
$Z_0\subset X$ is a line, and the divisor $F'\subset X'$ is the strict transform of 
$(F')^V\subset V$ on $X'$. Since $X^+$ is a smooth weak Fano threefold, 
we have $\operatorname{length}\left(\sO_{\Gamma\cap (Z'_i)^V}\right)=2$ and 
the strict transform $(Z'_i)^{X^+}\subset X^+$ of $(Z'_i)^V$ is a flopping curve 
of $\beta^+$ for any $i=1,2$. By Proposition \ref{proposition:disjoint-lines}, 
the curves $(Z'_1)^{X^+}$ and $(Z'_2)^{X^+}$ are mutually disjoint. 
Set $l^{X^+}:=\tau^{-1}(p')\subset X^+$. Then the curve $l^{X^+}$ intersects with 
both $(Z'_1)^{X^+}$ and $(Z'_2)^{X^+}$. Since the elementary flop $\chi^{-1}$ is 
$\chi_*\sigma^*\sO_X(-K_X)$-negative and $\sigma^*\sO_X(-K_X)$ is nef 
on $X'$, by Lemma \ref{lemma:negativity}, we have 
\[
2=\left(\chi_*\sigma^*\sO_X(-K_X)\cdot l^{X^+}\right)
\geq\left(\sigma^*\sO_X(-K_X)\cdot \chi^{-1}_*(l^{X^+})\right)+2\geq 2. 
\]
Thus, the curve $\chi^{-1}_*(l^{X^+})\subset X'$ is a fiber of $\sigma$. 
Set $p:=\sigma\left(\chi^{-1}_*(l^{X^+})\right)$. Note that $p\in Z_0$. 
Let $Z_i^{X'}\subset X'$ be the flopped curve of $(Z'_i)^{X^+}$ 
with respects to the elementary flop $\chi^{-1}$ 
for $i=1,2$. Since $\chi^{-1}_*(l^{X^+})$ and $Z_i^{X'}$ intersect, the image 
$Z_i\subset X$ of $Z_i^{X'}$ is a line passing through $p\in X$ for $i=1,2$ by 
Theorem \ref{thm:double-projection-from-line}. Moreover, since $(F')^V$ is a normal 
surface, we have $\sN_{Z_0/X}\cong\sO_{\pr^1}\oplus\sO_{\pr^1}(-1)$ by 
Theorem \ref{thm:double-projection-from-line-converse} 
\eqref{thm:double-projection-from-line-converse1}. Therefore, again by 
Theorem \ref{thm:double-projection-from-line}, the lines $Z_0$, $Z_1$, $Z_2$ in $X$ are 
mutually distinct, and passing through $p\in X$. 
\end{example}

\begin{thm}\label{thm:configuration}
Let $X$ be a prime Fano threefold of genus $12$. Take any $m\in\{4,5,6\}$. Then we 
have the following: 
\begin{align}\label{equation:no-chain}
\text{there is no series of distinct lines $\left[Z'_1\right],\dots,\left[Z'_m\right]\in
\Sigma(X)$ such that} \\
\text{$Z'_i\cap Z'_{i+1}\neq\emptyset$ for any $1\leq i\leq m$, 
where we set $Z'_{m+1}:=Z'_1$.} \nonumber
\end{align}
\end{thm}

In order to show this theorem, we prepare the following lemma: 

\begin{lemma}\label{lemma:configuration}
Let $X$ be a prime Fano threefold of genus $12$. Assume that $m'=4$, or 
$m'\geq 5$ and the property \eqref{equation:no-chain} holds for any $4\leq m\leq m'-1$. 
Assume that there is a series of distinct lines 
$\left[Z'_1\right],\dots,\left[Z'_{m'}\right]\in\Sigma(X)$ 
such that $Z'_i\cap Z'_{i+1}\neq\emptyset$ for any $1\leq i\leq m'$, 
where we set $Z'_{m'+1}:=Z'_1$. 
Then we have the following: 
\begin{enumerate}
\renewcommand{\theenumi}{\arabic{enumi}}
\renewcommand{\labelenumi}{(\theenumi)}
\item\label{lemma:configuration1}
For any point $p\in X$, we have 
\[
\#\left\{i\in\{1,\dots,m'\}\mid p\in Z'_i\right\}\leq 2. 
\]
\item\label{lemma:configuration2}
For any $1\leq i\leq m'$, we have 
\[
\#\left\{j\in\{1,\dots,m'\}\setminus\{i\}\mid Z'_i\cap Z'_j\neq\emptyset\right\}=2. 
\]
\end{enumerate}
\end{lemma}

\begin{proof}
\eqref{lemma:configuration1} 
Assume that $Z'_1\cap Z'_i\cap Z'_j\neq\emptyset$ for some $1<i<j\leq m'$. 
We may assume that $3\leq j<m'$. (If $j=m'$, then replace $Z'_1,Z'_i,Z'_{m'}$ with 
$Z'_2,Z'_{i+1},Z'_1$, and if $i+1=m'$, then replace $Z'_2,Z'_{m'},Z'_1$ with 
$Z'_3,Z'_1,Z'_2$, after changing the order of $Z'_1,\dots,Z'_{m'}$ suitably.)
Then the curves $Z'_1,\dots,Z'_j$ satisfy the property \eqref{equation:no-chain}. 
Thus we have $j=3$ and $Z'_1\cap Z'_2\cap Z'_3\neq\emptyset$ by 
Lemma \ref{lemma:3lines}. On the other hand, the curves 
$Z'_1,Z'_3,\dots,Z'_{m'}$ also satisfy the property \eqref{equation:no-chain}. 
Thus we have $m'-1=3$, and $Z'_1\cap Z'_3\cap Z'_4\neq\emptyset$ again by 
Lemma \ref{lemma:3lines}. This contradicts with Lemma \ref{lemma:3lines} 
\eqref{lemma:3lines2}. Thus we get the assertion \eqref{lemma:configuration1}.

\eqref{lemma:configuration2}
Assume that $Z'_1\cap Z'_m\neq\emptyset$ for some $m\neq 1,2, m'$. 
Then the curves $Z'_1,\dots,Z'_m$ satisfy the property \eqref{equation:no-chain}. 
Thus we must have $m=3$. By Lemma \ref{lemma:3lines}, we have 
$Z'_1\cap Z'_2\cap Z'_3\neq\emptyset$. Thus the assertion 
\eqref{lemma:configuration2}
follows from \eqref{lemma:configuration1}. 
\end{proof}

\begin{proof}[Proof of Theorem \ref{thm:configuration}]
Assume that there exist a series of distinct lines $Z'_1,\dots,Z'_m$ in $X$ 
such that $Z'_i\cap Z'_{i+1}\neq\emptyset$ for $1\leq i\leq m$, whee $Z'_{m+1}:=Z'_1$. 
By induction on $m$ and by Lemma \ref{lemma:configuration}, 
we may assume the properties in Lemma \ref{lemma:configuration}. 
Let 
\[
\xymatrix{
 & F' \ar@{}[r]|{\subset}  & X'  \ar@{-->}[rr]^{\chi} 
\ar@{->}[ld]_-{\sigma} \ar@{->}[rd]_-{\beta} 
& & X^+ \ar@{->}[rd]^-{\tau} \ar@{->}[ld]^-{\beta^+} \ar@{}[r]|{\supset} & S^+ & \\
 Z'_1 \ar@{}[r]|{\subset} & X && \bar{X} && V \ar@{}[r]|{\supset}& \Gamma
}
\]
be the Sarkisov link from the blowup of $X$ along $Z'_1$ as in \eqref{equation:iskovskikh}. 
Note that the strict transforms of $Z'_2$ and $Z'_m$ to $X'$ 
are flopping curves of the elementary flop $\beta$. 

Assume that $m=4$. Then, the strict transform $(Z'_3)^{X'}\subset X'$ of 
$Z'_3$ intersects with $\Exc(\beta)$ at least $2$ points. By Lemma 
\ref{lemma:negativity}, we get 
\[
1=\left(\chi^{-1}_*\tau^*\sO_V(1)\cdot (Z'_3)^{X'}\right)
\geq\left(\tau^*\sO_V(1)\cdot \chi_*(Z'_3)^{X'}\right)+2\geq 2. 
\]
This leads to a contradiction. 

Assume that $m=5$. Then, for $i\in\{3,4\}$, the strict transform $(Z'_i)^{X'}\subset X'$ of 
$Z'_i$ intersects with $\Exc(\beta)$ at least $1$ point. By Lemma 
\ref{lemma:negativity}, we get 
\[
1=\left(\chi^{-1}_*\tau^*\sO_V(1)\cdot (Z'_i)^{X'}\right)
\geq\left(\tau^*\sO_V(1)\cdot \chi_*(Z'_i)^{X'}\right)+1\geq 1. 
\]
Therefore, Both $\chi_*(Z'_3)^{X'}$ and $\chi_*(Z'_4)^{X'}$ are fibers of $\tau$. 
However, since $Z'_3$ and $Z'_4$ intersect, so are 
$\chi_*(Z'_3)^{X'}$ and $\chi_*(Z'_4)^{X'}$. 
(By induction, the rational map $\chi\circ\sigma^{-1}$ is an isomorphism around 
a neighborhood of the point $Z'_3\cap Z'_4$.)
Since any nontrivial fiber of $\tau$ is irreducible, this leads to a contradiction. 

Assume that $m=6$. From the result for $m=5$, the lines $Z'_1, Z'_4$ in $X$ must be 
totally disjoint. However, the curve $Z'_1$ intersects with both $Z'_2$ and $Z'_6$, and 
the curve $Z'_4$ intersects with both $Z'_3$ and $Z'_5$. Since 
$Z'_2\cap Z'_3\neq\emptyset$ and $Z'_4\cap Z'_5\neq\emptyset$, it contradicts 
with Theorem \ref{thm:flop-curve} \eqref{thm:flop-curve51}. 

As a consequence, we have completed the proof of Theorem \ref{thm:configuration}. 
\end{proof}

\section{Multiplicative group actions}\label{section:special}

In \S \ref{section:special}, we consider the case that the link in \S 
\ref{section:go} for the case $g=12$ and is effectively $\G_m$-equivariant. 
We firstly recall the following result: 

\begin{proposition}[{\cite[Lemmas 5.49 and 5.50]{Book}}]\label{proposition:book}
Let us fix an action $\G_m\curvearrowright\pr^4_{x_0\dots x_4}$ with 
\[
\lambda\cdot\left[x_0:x_1:x_2:x_3:x_4\right]=\left[x_0:\lambda x_1:\lambda^2 x_2
:\lambda^3 x_3:\lambda^4 x_4\right] 
\]
and a $\G_m$-equivariant twisted quartic curve $\Gamma\subset\pr^4$ 
defined by the image of 
\begin{eqnarray*}
\pr^1_{t_0t_1} &\to&\pr^4_{x_0\dots x_4} \\
\left[t_0:t_1\right]&\mapsto&
\left[t_0^4:t_0^3t_1:t_0^2t_1^2:t_0t_1^3:t_1^4\right]. 
\end{eqnarray*}
For any $q\in\pr^1\setminus\{0,1,\infty\}$, let us consider the $\G_m$-invariant 
smooth hyperquadric $Q\langle q\rangle\subset\pr^4$ containing $\Gamma$ 
defined by 
\[
x_1x_3-qx_0x_4+(q-1)x_2^2=0.
\]
Let $\rho^q\colon\hat{Q}\langle q\rangle\to Q\langle q\rangle$ be the 
blowup of $Q\langle q\rangle$ along $\Gamma$. Then, the threefold 
$\hat{Q}\langle q\rangle$ is a Fano threefold of type 2.21 with $\G_m\subset
\Aut(\hat{Q}\langle q\rangle)$, and the threefolds 
$\left\{\hat{Q}\langle q\rangle\right\}_{q\in\pr^1\setminus\{0,1,\infty\}}$ are 
mutually non-isomorphic to each other. 
Conversely, any Fano threefold $\hat{Q}$ of 
type 2.21 with $\G_m\subset \Aut(\hat{Q})$ is isomorphic to 
$\hat{Q}\langle q\rangle$ for some $q\in\pr^1\setminus\{0,1,\infty\}$. 
\end{proposition}

\begin{lemma}\label{lemma:gm}
Let $\hat{C}\subset\hat{Q}\langle q\rangle$ be a $\G_m$-invariant bi-cubic 
curve in $\hat{Q}\langle q\rangle$ under the $\G_m$-action in Proposition 
\ref{proposition:book}. Then, after twisting by the involution 
\[
\iota'_Q\colon\left[x_0:x_1:x_2:x_3:x_4\right]\mapsto\left[x_4:x_3:x_2:x_1:x_0\right]
\]
if necessary, the strict transform 
$(\rho^q)_*\hat{C}\subset Q\langle q \rangle$ must be 
equal to the curve $C\langle q\rangle$ defined by the image of 
\begin{eqnarray*}
\pr^1_{u_0u_1} &\to& \pr^4_{x_0\dots x_4}\\ 
\left[u_0:u_1\right]&\mapsto&\left[u_0^3:u_0^2u_1:u_0u_1^2:(1-q)u_1^3:0\right]. 
\end{eqnarray*}
(We remark that the involution $\iota'_Q$ is an element in 
$\Aut(Q\langle q\rangle; \Gamma)$.)
\end{lemma}

\begin{proof}
From $\hat{C}\subset\hat{Q}\langle q\rangle$, the link in Corollary 
\ref{corollary:back-g12} 
ends with a prime Fano threefold $X$ with an effective $\G_m$-action and 
two $\G_m$-invariant lines $Z_1\cup Z_2$ on $X$ which are totally 
disjoint. 
By \cite[Lemma 21 (2)]{DFK} (see also \cite[Proposition 5.4.4]{KPS18}), 
we have $\sN_{Z_i/X}\cong\sO_{\pr^1}(-2)\oplus\sO_{\pr^1}(1)$. 

By looking at the reverse direction of the link, 
we get the assertion as in the proof of \cite[Lemma 23]{DFK}. 
We give an alternative (but an essentially same and precise) 
proof for readers' convenience. 
Set $Q_1:=Q\langle q\rangle$, $\hat{Q}:=\hat{Q}\langle q \rangle$, $\rho_1:=\rho^q$ 
and $C_1:=(\rho^q)_*\hat{C}$. Consider the link 
\[\xymatrix{
& E^{Y_1} \ar@{}[r]|{\subset} & Y_1 
\ar@{}[r]|{\supset} \ar@{->}[ld]_{\psi_1} \ar@{->}[rd]^{\phi_1}
& F_2^{Y_1} & \\
C_1  \ar@{}[r]|{\subset} & Q_1 && V_1 \ar@{}[r]|{\supset} & Z_2^{V_1}
}\]
from the blowup $Q_1$ along $C_1$ 
as in \eqref{equation:fujita-converse}. 
Set $\Gamma_1:=(\phi_1\circ\psi_1^{-1})_*\Gamma$. 
Note that $F_2^{Q_1}:=(\psi_1)_*F_2^{Y_1}$ is nothing but the hyperplane 
section of $Q_1$ containing $C_1$. 
Consider the diagram
\[\xymatrix{
& W_1^+ \ar@{->}[d]_{\tau_1^+} \ar@{-->}[rr]^-{\chi_{12}^{-1}} \ar[rd]^-{\beta_{12}^+} 
& & W_1 \ar@{->}[d]^{\phi_1^+} \ar[ld]_-{\beta_{12}} & \\
F_2^{Y_1} \ar@{}[r]|{\subset} & Y_1 \ar@{->}[rd]_{\phi_1} & V'_1 
\ar[d]_(.4){\gamma_{12}} & X_1^+ \ar@{}[r]|{\supset} 
\ar@{->}[ld]^{\tau_1}  & S_1^+ \\
& Z_2^{V_1} \ar@{}[r]|{\subset} & V_1  \ar@{}[r]|{\supset} & \Gamma_1 & 
}\]
as in \S \ref{section:go}, where $\tau_1^+$ is the blowup along 
$(\phi_1)^{-1}_*\Gamma_1$ and $\phi_1^+$ is the blowup along $(\tau_1)^{-1}_*Z_2^{V_1}$. 
Since $Z_1$, $Z_2$ is a pair of totally disjoint $2$ lines, we have 
\[
\sN_{(\tau_1)^{-1}_*Z_2^{V_1}/X_1^+}\cong\sN_{Z_2/X}
\cong\sO_{\pr^1}(-2)\oplus\sO_{\pr^1}(1).
\] 
In particular, the exceptional divisor of $\phi_1^+$ is isomorphic to the Hirzebruch 
surface $\F_3$ having a $(-3)$-curve. Since the rational map 
$\chi_{12}^{-1}\colon W_1^+\dashrightarrow W_1$ is obtained by Atiyah's flop, 
the surface $F_2^{Y_1}$ must be isomorphic to the Hirzebruch surface 
$\F_2$ having a $(-2)$-curve. 
(Indeed, by Theorem \ref{thm:dP5} \eqref{thm:dP53}, 
the surface $F_2^{Y_1}$ is isomorphic to 
either $\pr^1\times\pr^1$ or $\F_2$.)
Therefore, we can show that 
$\sN_{Z_2^{V_1}/V_1}\cong\sO_{\pr^1}(-1)\oplus\sO_{\pr^1}(1)$. 
(This result follows from \cite[Proposition 20 (4)]{DFK}, but we gave an alternative 
proof for convenience.)
Moreover, the curve 
$(\phi_1)^{-1}_*\Gamma_1$ in $Y_1$ must intersects with $F_2^{Y_1}$ at a point 
in the $(-2)$-curve. 
In particular, the surface $F_2^{Q_1}$ is a singular 
quadric surface, and the vertex of $F_2^{Q_1}$ 
lies on the curve $\Gamma$. 

Since the divisor $F_2^{Q_1}\subset Q_1$ is a $\G_m$-invariant hyperplane section, 
the defining equation of $F_2^{Q_1}\subset Q_1$ is one of $x_0=0,\dots,x_4=0$. 
Since $F_2^{Q_1}$ is singular and its vertex lies on $\Gamma$, we must have 
$x_0=0$ or $x_4=0$. 
By twisting by $\iota'_Q$ if necessary. we may assume that 
$x_4=0$. Let us set $p_{x_0}:=\left[1:0:0:0:0\right]$. Since 
$\left[0:0:0:0:1\right]\not\in F_2^{Q_1}$, the curves $\Gamma$ and $C_1$ meet 
only at the point $p_{x_0}$. Thus we have 
\[
\operatorname{length}_{p_{x_0}}\left(\sO_{\Gamma\cap C_1}\right)=3.
\]
A general point in $C_1$ can be written as 
\[
\left[1:\alpha_1:\alpha_2:(1-q)\alpha_1^{-1}\alpha_2^2:0\right]\in\pr^4
\]
for some $\alpha_1,\alpha_2\in\Bbbk^\times$ with $\alpha_2\neq 1-q$, and $C_1$ 
is the closure of the $\G_m$-orbit of the point. 
Thus $C_1$ is scheme-theoretically defined by the equations 
\[
\left\{
\begin{aligned}
x_4=&0,\\
\alpha_1^2x_0x_2-\alpha_2x_1^2=&0,\\
\alpha_1^2x_0x_3-(1-q)\alpha_2x_1x_2=&0,\\
x_1x_3-(1-q)x_2^2=&0.
\end{aligned}
\right.\]
Thus, around $p_{x_0}\in\A^4_{x_1\dots x_4}\subset\pr^4_{x_0\dots x_4}$, 
the defining ideal $I_{C_1}\subset\Bbbk\left[x_1,\dots,x_4\right]$ of $C_1$ satisfies that 
\[
I_{C_1}=\left(x_4,\,\,\alpha_1^2x_2-\alpha_2x_1^2,\,\, 
\alpha_1^2x_3-(1-q)\alpha_2x_1x_2, \,\,x_1x_3-(1-q)x_2^2\right). 
\]
The image of $I_{C_1}$ under the surjection 
\begin{eqnarray*}
\Bbbk\left[x_1,\dots,x_4\right]&\twoheadrightarrow&\Bbbk[t]\cong\Bbbk\left[
x_1,\dots,x_4\right]/I_\Gamma \\
x_i &\mapsto& t^i
\end{eqnarray*}
is equal to 
\[
\left(t^4,\,\,(\alpha_1^2-\alpha_2)t^2,\,\,(\alpha_1^2-(1-q))t^3,\,\,q t^4\right). 
\]
Together with the condition on the length, we must have $\alpha_2=\alpha_1^2$. 
Thus the assertion follows. 
\end{proof}

\begin{remark}\label{remark:kuznetsov}
Alexander Kuznetsov pointed out to the author that the proof of Lemma \ref{lemma:gm} 
can be simplified. 
In fact, since $C_1\subset Q_1$ is a $\G_m$-invariant \emph{cubic} curve, the 
linear span of $C_1$ must be generated by four consecutive weights of $\G_m$. 
Thus, we can immediately show that 
the defining equation of $F_2^{Q_1}$ must be $x_1=0$ or $x_4=0$.
\end{remark}

\begin{definition}\label{definition:x22-from-q}
Let $X_0^+\langle q \rangle\to \hat{Q}\langle q \rangle$ be the blowup along 
the strict transform $\hat{C}\langle q\rangle:=(\rho^q)^{-1}_*(C\langle q\rangle)
\subset\hat{Q}\langle q\rangle$ of $C\langle q\rangle$. 
Moreover, let 
\[
\hat{Q}\langle q\rangle \, \longleftarrow X_0^+\langle q \rangle \, 
\dasharrow \, X_0\langle q \rangle \, \longrightarrow \, X\langle q \rangle
\]
be the link as in \eqref{equation:back-g12}. 
The $X\langle q\rangle$ is a prime Fano threefold of genus $12$ 
such that $\G_m\subset\Aut(X\langle q\rangle)$ holds.
\end{definition}

We set $G:=\G_m\rtimes \boldsymbol{\mu}_2$. 
From now on, we will prove that 
$\Aut(\hat{Q}\langle q \rangle;\hat{C}\langle q\rangle)\cong G$. 

\begin{lemma}\label{lemma:mumu2}
Fix a square root $s$ of $q$. 
Consider the birational involution $\iota''_Q$ on $Q\langle q \rangle$
defined by 
\begin{eqnarray*}
&\iota''_Q&\colon\left[x_0:x_1:x_2:x_3:x_4\right]\mapsto\\
&&\left[x_0x_2-x_1^2:s(x_0x_3-x_1x_2):q(x_0x_4-x_2^2):s(x_1x_4-x_2x_3)
:x_2x_4-x_3^2\right]. 
\end{eqnarray*}
Then, as in \cite[\S 5.9]{Book}, both $\iota'_Q$ and $\iota''_Q$ lift to regular involutions 
$\iota'$, $\iota''\in\Aut(\hat{Q}\langle q\rangle)$. Set 
$\hat{\iota}:=\iota''\circ\iota'\in\Aut(\hat{Q}\langle q\rangle)$. 
\begin{enumerate}
\renewcommand{\theenumi}{\arabic{enumi}}
\renewcommand{\labelenumi}{(\theenumi)}
\item\label{lemma:mumu21}
We have $\hat{\iota}\in\Aut(\hat{Q}\langle q\rangle; \hat{C}\langle q\rangle)$ 
such that $\hat{\iota}^*\rho_1^*\sO_{Q_1}(1)\cong\rho_2^*\sO_{Q_2}(1)$. 
\item\label{lemma:mumu22}
The subgroup of $\Aut(\hat{Q}\langle q\rangle; \hat{C}\langle q\rangle)$ generated 
by $\G_m$ and $\hat{\iota}$ is isomorphic to $G$. 
\end{enumerate}
\end{lemma}

\begin{proof}
We can directly check that $\hat{\iota}(\hat{C}\langle q \rangle)
=\hat{C}\langle q \rangle$. 
The remaining assertions follow directly from \cite[Remark 5.52]{Book}. 
\end{proof}

\begin{thm}\label{thm:mumu2}
We have $\Aut(\hat{Q}\langle q \rangle; \hat{C}\langle q \rangle)\cong G$. 
In particular, we have 
\[
\Aut(X_0^+\langle q \rangle)\cong\Aut(X_0\langle q \rangle)
\cong G.
\] 
\end{thm}

\begin{proof}
By Lemma \ref{lemma:mumu2}, it is enough to show that 
$\Aut(Q\langle q \rangle; C\langle q\rangle\cup \Gamma)=\G_m$. 
As in \cite[Lemma 5.50]{Book}, for any $\theta\in
\Aut(Q\langle q \rangle; C\langle q\rangle\cup \Gamma)$, if 
$\begin{pmatrix}a & b\\c & d\end{pmatrix}$ with $ad-bc=1$ be the image of 
the embedding 
\[
\Aut(Q\langle q \rangle; C\langle q\rangle\cup \Gamma)\subset\Aut(\pr^4; \Gamma)
=\PGL(2,\Bbbk), 
\]
then $\theta$ is given by 
\[
\begin{bmatrix}x_0\\ x_1\\ x_2 \\x_3\\ x_4
\end{bmatrix}\mapsto
\begin{bmatrix}
a^4 & 4a^3b & 6a^2b^2 & 4ab^3 & b^4\\
a^3c & a^3d+3a^2bc & 3a^2bd+3ab^2c & 3ab^2d+b^3c & b^3d\\
a^2c^2& 2a^2cd+2abc^2 & a^2d^2+4abcd+b^2c^2 & 2abd^2+2b^2cd & b^2d^2\\
ac^3 & 3ac^2d+bc^3 & 3acd^2+3bc^2d & ad^3+3bcd^2 & bd^3 \\
c^4 & 4c^3d & 6c^2d^2 & 4cd^3 & d^4
\end{bmatrix}
\begin{bmatrix}x_0\\ x_1\\ x_2 \\x_3\\ x_4
\end{bmatrix}.
\]
Since $\theta(p_{x_0})=p_{x_0}$, we have $c=0$. Moreover, since 
$\theta(C\langle q\rangle)=C\langle q\rangle$, 
we can directly check that $b=0$. Thus we get the assertion. 
\end{proof}

\begin{remark}\label{remark:mumu2}
Let $X$ be a prime Fano threefold of genus $12$ with $\Aut^0(X)\cong\G_m$. 
\begin{enumerate}
\renewcommand{\theenumi}{\arabic{enumi}}
\renewcommand{\labelenumi}{(\theenumi)}
\item\label{remark:mumu21}
In \cite[Theorem 22]{DFK}, the authors showed that $\Aut(X)\cong G$ 
by focusing on the involution $\hat{\iota}$. The above theorem gives a new, but 
an almost same, proof to show $\Aut(X)\cong G$. In fact, there are exactly 
two $\G_m$-invariant lines $Z_1$, $Z_2$ in $X$.
Moreover, for any $\left[Z\right]\in\Sigma(X)\setminus\{\left[Z_1\right],
\left[Z_2\right]\}$, we have $Z\cap Z_1=\emptyset$ and $Z\cap Z_2=\emptyset$, 
see \cite[Proposition 5.4.3]{KPS18} and \cite[Lemma 3.1]{KP18} (see also 
the proof of \cite[Theorem 22]{DFK}). In particular, 
the pair $Z_1$, $Z_2$ of lines is an absolutely disjoint pair of lines in $X$. Therefore 
we get $\Aut(X)=\Aut(X; Z_1\cup Z_2)\cong\Aut(X_0\langle q \rangle)\cong G$ 
for some $q\in\pr^1\setminus\{0,1,\infty\}$. 
\item\label{remark:mumu22}
In \cite[Proposition 3.2 and Lemma 4.1]{KP18} and \cite[Lemma 5.12]{IKTT}, 
the authors showed that there 
uniquely exists a conic on $X$ passing through both $Z_1$ and $Z_2$. 
Corollary \ref{corollary:12goback} gives an alternative proof of it. 
\end{enumerate}
\end{remark}

We can focus on the images of flopped curves. 

\begin{proposition}\label{proposition:gm-flop}
Let 
\[
\hat{Q}\langle q\rangle \, \longleftarrow X_0^+\langle q \rangle \, 
\dasharrow \, X_0\langle q \rangle \, \longrightarrow \, X\langle q \rangle
\]
be the link as in \eqref{equation:back-g12}. Then there are exactly $3$ number of the flopped 
curves $B_1^{+X_0^+}$, $B_2^{+X_0^+}$, $\sC^{+X_0^+}\subset X_0^+\langle q\rangle$ 
of $X_0\langle q\rangle\dashrightarrow X_0^+\langle q \rangle$. 
The image of the curve $B_2^{+X_0^+}$ to $Q\langle q\rangle$ is the point 
$\left[1:0:0:0:0\right]\in Q\langle q \rangle$. 
The image $B_1^+\langle q \rangle\subset Q\langle q \rangle$ of the curve 
$B_1^{+X_0^+}$ is the line in $Q\langle q \rangle$ defined by $x_0=x_1=x_2=0$. 
The image $\sC^+\langle q \rangle\subset Q\langle q \rangle$ of the curve 
$\sC^{+X_0^+}$ is the line in $Q\langle q \rangle$ defined by $x_1=x_2=x_4=0$. 
\end{proposition}

\begin{proof}
By Remark \ref{remark:mumu2} and Corollary \ref{corollary:flop-alpha}, 
there are exactly $3$ flopped curves 
$B_1^{+X_0^+}$, $B_2^{+X_0^+}$, $\sC^{+X_0^+}\subset X_0^+\langle q\rangle$.
(In the terminologies in \S \ref{section:flop}, we have $\bar{m}=0$ and 
$m_0=m_1=m_2=1$.) Moreover, the curve $B_2^{+X_0^+}$ is contracted 
by $\rho^q\circ\tau$ onto the point 
$C\langle q\rangle\cap\Gamma=\left[1:0:0:0:0\right]$, 
the image $B_1^+\langle q \rangle\subset Q\langle q \rangle$ of the curve 
$B_1^{+X_0^+}$ is the image of the inverse image of the point 
$C\langle q\rangle\cap\Gamma=\left[1:0:0:0:0\right]$ by the graph of the birational 
involution $\hat{\iota}$, which is nothing but the line defined by $x_0=x_1=x_2=0$. 
The image $\sC^+\langle q \rangle \subset Q\langle q \rangle$ 
of $\sC^{+X_0^+}$ must be a $\G_m$-invariant line such that 
\[
\operatorname{length}\left(\sO_{C\langle q\rangle\cap\sC^+\langle q \rangle}\right)
\geq 2. 
\]
Any point in $C\langle q\rangle\cap\sC^+\langle q \rangle$ is $\G_m$-invariant. 
Thus the line $\sC^+\langle q\rangle$ must be either $x_2=x_3=x_4=0$ or 
$x_1=x_2=x_4=0$. On the other hand, since $\sC^+\langle q\rangle$ is 
$\hat{\iota}$-invariant, the curve $\sC^+\langle q \rangle$ must be equal to the line 
$x_1=x_2=x_4=0$. 
\end{proof}

\section{On parametrizations}\label{section:parameter}

Let $X$ be a prime Fano threefold of genus $12$ with $\G_m\subset\Aut(X)$. 
Such prime Fano threefolds $X$ of genus $12$ have been studied 
in \cite{KPS18} and \cite{KP18}, and there is an exactly 
$1$-dimensional family of such isomorphism classes. 
In \cite[Theorem 22]{DFK}, the authors parametrized such family of isomorphism 
classed of $X$ as $X_{22}^m(v)$ $(v\in\pr^1\setminus\{0,1,\infty\})$. 
If $v=-4$, we write $X_{22}^{\MU}:=X_{22}^m(-4)$ and called it 
the \emph{Mukai--Umemura threefold}, which was deeply studied in \cite{MU}. 
If $v\neq -4$, then we have $\Aut^0(X_{22}^m(v))=\G_m$. 
Under fixing the embedding $\G_m\subset\Aut(X)$, 
there are exactly two $\G_m$-invariant lines $Z_1$, $Z_2$ in $X$. 
Moreover, the two curves are totally disjoint.

\subsection{Relation with Fano threefolds of type 2.21}\label{subsection:221}

Let $X:=X_{22}^m(v)$ $(v\in\pr^1\setminus\{0,1,\infty\})$ and let $Z_1$, 
$Z_2\subset X$ be the lines as above (under the fixed embedding 
$\G_m\subset\Aut(X)$). From the variety $X$ and the pair $Z_1$, $Z_2$, 
we can consider the link in \S \ref{section:go}, which ends with 
a Fano threefold $\hat{Q}$ of type 2.21 with a bi-cubic curve with 
$\G_m\subset\Aut(\hat{Q}; \hat{C})$. 

\begin{proposition}\label{proposition:parameter}
We have 
$X_{22}^m(v)\cong X\langle q\rangle$ with $q=\frac{1}{1-v}\Leftrightarrow
v=\frac{q-1}{q}$. 
\end{proposition}

\begin{proof}
Set $Q_1:=Q\langle q\rangle$, $C_1:=C\langle q\rangle$, let 
\[\xymatrix{
& E^{Y_1} \ar@{}[r]|{\subset} & Y_1 
\ar@{}[r]|{\supset} \ar@{->}[ld]_{\psi_1} \ar@{->}[rd]^{\phi_1}
& F_2^{Y_1} & \\
C_1  \ar@{}[r]|{\subset} & Q_1 && V_1 \ar@{}[r]|{\supset} & Z_2^{V_1}
}\]
be the link as in \eqref{equation:fujita-converse}, and let $\Gamma_1\subset V_1$ be 
the strict transform of $\Gamma\subset Q_1$. The rational map 
\[
\phi_1\circ\psi_1^{-1}\colon Q_1\dashrightarrow V_1\subset\pr^6_{z_0\dots z_6}
\]
is obtained by the linear system $\left|\sO_{Q_1}(2)\otimes I_{C_1}\right|$. 
Thus we may assume that $\phi_1\circ\psi_1^{-1}$ is obtained by 
\begin{eqnarray*}
\left[x_0:\cdots:x_4\right]
&\mapsto&
\Bigl[\frac{q-1}{q}(x_0x_2-x_1^2):-x_0x_4:x_2x_4:-\frac{q}{q-1}x_4^2:
x_1x_4\\
&&:-\frac{1}{q}\left(x_0x_3+(q-1)x_1x_2\right):\frac{1}{q-1}x_3x_4\Bigr]. 
\end{eqnarray*}
The image $V_1\subset \pr^6_{z_0\dots z_6}$ of $Q_1$ is scheme-theoretically 
defined  by the equations 
\[
\left\{
\begin{aligned}
&z_4z_5-z_0z_2+z_1^2&=&0,\\
&z_4z_6-z_1z_3+z_2^2&=&0,\\
&z_4^2-z_0z_3+z_1z_2&=&0,\\
&z_1z_4-z_0z_6-z_2z_5&=&0,\\
&z_2z_4-z_3z_5-z_1z_6&=&0.
\end{aligned}
\right.\]
In fact, the image of a general point on $Q_1$ lies on the above subscheme, and 
the subscheme defines the del Pezzo threefold of degree $5$ by 
\cite[\S 5.8]{Book}. Moreover, the curve $\Gamma_1\subset V_1$ is defined by 
\[
\left\{\left[0:-u_0^4u_1:u_0^2u_1^3:-\frac{q}{q-1}u_1^5:
u_0^3u_1^2:-u_0^5:\frac{1}{q-1}u_0u_1^4
\right]\,\,\bigg|\,\,\left[u_0:u_1\right]\in\pr^1\right\}.
\]
In particular, $F_1^{V_1}:=\langle \Gamma_1\rangle\cap V_1\subset V_1$ 
is defined by $z_0=0$. 
Let $\F_3$ be the Hirzebruch surface admitting a $(-3)$-curve, and let
$\mathbf{x}_0$, $\mathbf{x}_1$, $\mathbf{y}_0$, $\mathbf{y}_1$ be the 
multi-homogeneous coordinates of $\F_3$, where 
$\deg\mathbf{x}_0=\deg\mathbf{x}_1=(1,0)$, $\deg\mathbf{y}_0=(-3,1)$ and 
$\deg\mathbf{y}_1=(0,1)$. We can directly check that the normalization morphism 
$\nu\colon\F_3\to F_1^{V_1}\subset\pr^6$ is given by 
\begin{eqnarray*}
\left[\mathbf{x}_0:\mathbf{x}_1; \mathbf{y}_0:\mathbf{y}_1\right]&\mapsto&
\Bigl[0:\mathbf{x}_0^3\mathbf{x}_1\mathbf{y}_0: 
-\mathbf{x}_0\mathbf{x}_1^3\mathbf{y}_0: -\mathbf{x}_1\mathbf{y}_1:
-\mathbf{x}_0^2\mathbf{x}_1^2\mathbf{y}_0:
\mathbf{x}_0^4\mathbf{y}_0:\mathbf{x}_0\mathbf{y}_1+\mathbf{x}_1^4\mathbf{y}_0
\Bigr].
\end{eqnarray*}
On the other hand, the image of the curve 
\[
\Upsilon^m(v):=\left\{v\mathbf{x}_0\mathbf{y}_1+\mathbf{x}_1^4\mathbf{y}_0=0\right\}
\]
(see \cite[Example 15]{DFK}) by $\nu$ is given by 
\[
\left\{\left[0:-vu_0^4u_1:vu_0^2u_1^3:-u_1^5:
vu_0^3u_1^2:-vu_0^5:(1-v)u_0u_1^4
\right]\,\,\Big|\,\,\left[u_0:u_1\right]\in\pr^1\right\}.
\]
Thus we have the equality $v=\frac{q-1}{q}$. 
\end{proof}

\begin{remark}\label{remark:counter-intuitive}
As in \cite[\S 5.9]{Book}, $\Aut^0(\hat{Q}\langle q\rangle)=\G_m$ if and only if 
$q\neq 1/4$, and $\Aut^0(\hat{Q}\langle 1/4\rangle)=\PGL(2,\Bbbk)$. 
It is counter-intuitive that $X\langle 1/4\rangle$ is \emph{not} isomorphic
to $X_{22}^{\MU}$. In fact, $X_{22}^{\MU}\cong X\langle 1/5\rangle$ holds. 
\end{remark}

\subsection{On Kuznetsov--Prokhorov's parametrizations}\label{subsection:KP}

Next, let us consider Kuznetsov--Prokhorov's parametrization \cite{KP18}. 
We largely follow the terminologies in \cite[\S 2]{CS}. 
We again set $G:=\G_m\rtimes \boldsymbol{\mu}_2$. 
Consider the $4$-dimensional projective space $\pr^4_{\ttx \tty \ttz \ttt \ttw}$ 
together with a $G$-action by 
\begin{eqnarray*}
\lambda\colon\left[\ttx:\tty:\ttz:\ttt:\ttw\right]&\mapsto&
\left[\ttx:\lambda\tty:\lambda^3\ttz:\lambda^5\ttt:\lambda^6\ttw\right],\\
\iota\colon\left[\ttx:\tty:\ttz:\ttt:\ttw\right]&\mapsto&
\left[\ttw:\ttt:\ttz:\tty:\ttx\right],
\end{eqnarray*}
where $\lambda\in\G_m$ and $\iota\in\boldsymbol{\mu}_2$. 
Let $\Lambda\subset\pr^4$ be the $G$-invariant rational sextic curve defined by 
the image of 
\[
\left[s_0:s_1\right]\mapsto\left[s_0^6:s_0^5s_1:s_0^3s_1^3:s_0s_1^5:s_1^6\right]. 
\]

Take any $u\in\pr^1\setminus\{0,1,\infty\}$, and consider the following $G$-invariant 
smooth hyperquadric $Q\left[u\right]\subset\pr^4$ containing $\Lambda$ 
defined by the equation 
\[
u(\ttx \ttw-\ttz^2)+(\ttz^2-\tty\ttt)=0. 
\]
Consider the following $G$-equivariant Sarkisov link 
\[\xymatrix{
& \E_\Lambda \ar@{}[r]|{\subset}  & \tilde{Q}\left[u\right]  \ar@{-->}[rr]^{\tilde{\chi}} 
\ar@{->}[ld]_-{\sigma_\Lambda} \ar@{->}[rd] 
& & \tilde{Q}^+\left[u\right] \ar@{->}[rd]^-{\sigma_{\sC_2\left[u\right]}} 
\ar@{->}[ld] \ar@{}[r]|{\supset} 
& \mathbb{S}_{\sC_2\left[u\right]} & \\
\Lambda \ar@{}[r]|{\subset} & Q\left[u\right] && \bar{Q}\left[u\right] && 
X^{\KP}\left[u\right] \ar@{}[r]|{\supset}& \sC_2\left[u\right]
}\]
starting from the blowup of $Q\left[u\right]$ along $\Lambda$ 
as in \cite[Theorem 2.2]{KP18} (see also \cite[(2.5)]{CS}), 
which is the converse of the link \eqref{equation:takeuchi}. 
By \cite[Theorem 1.3]{KP18}, all $X^{\KP}\left[u\right]$ are pairwise non-isomorphic 
to each other, $G\subset\Aut\left(X^{\KP}\left[u\right]\right)$, and any prime 
Fano threefold $X$ of genus $12$ with $\G_m\subset \Aut(X)$ is isomorphic to 
$X^{\KP}\left[u\right]$ for some $u\in\pr^1\setminus\{0,1,\infty\}$. 
Thus, there are $q\in\pr^1\setminus\{0,1,\infty\}$ and 
$v\in \pr^1\setminus\{0,1,\infty\}$ such that $X^{\KP}\left[u\right]\cong 
X\langle q\rangle\cong X_{22}^m(v)$, where $X\langle q\rangle$, $X_{22}^m(v)$
are as in \S \ref{section:special}, \S \ref{subsection:221}, respectively. 

\begin{thm}\label{thm:parameter}
We have $X\langle q\rangle\cong X^{\KP}\left[u\right]$ with $q=\frac{u}{u-1}
\Leftrightarrow u=\frac{q}{q-1}$. In particular, we have 
$X_{22}^m(v)\cong X^{\KP}\left[u\right]$ with $u=\frac{1}{v}$. 
\end{thm}

\begin{proof}
The proof is divided into 4 numbers of steps. 

\noindent\underline{\textbf{Step 1}}\\
We firstly recall that, for any prime Fano threefold $X$ of genus $12$ and 
a totally disjoint pair of lines $Z_1$, $Z_2$ in $X$, the rational map 
\[
\mu_1:=\psi_1\circ\phi_1^{-1}\circ\tau_1\circ\chi_1\circ\sigma_1^{-1}\colon 
X\dashrightarrow Q_1
\quad(\text{resp., }
\mu_2:=\psi_2\circ\phi_2^{-1}\circ\tau_2\circ\chi_2\circ\sigma_2^{-1}\colon
X\dashrightarrow Q_2) 
\]
in the link \eqref{equation:big} is given by the sub-linear system 
\[
\left|\sigma^*(-K_X)-2F_1-F_2\right|\quad 
(\text{resp., }\left|\sigma^*(-K_X)-F_1-2F_2\right|) 
\]
of the complete linear system $|-K_X|$. 

\noindent\underline{\textbf{Step 2}}\\
From now on, let us consider $X:=X^{\KP}\left[u\right]$ and let $Z_1$, $Z_2\subset X$ 
be the $\G_m$-invariant lines in $X$. 
We recall the result in \cite[\S 2]{CS}. Let us set 
$\ttf:=\ttx\ttw-\tty\ttt$. As in \cite[\S 2]{CS}, the equation $\ttf=0$ is 
the defining equation of the 
strict transform of $\mathbb{S}_{\sC_2\left[u\right]}$ to $Q\left[u\right]$. 
Moreover, let us set 
\[\begin{aligned}
&\tth_3:=\tty^3-\ttx^2\ttz, \quad \tth_5:=\ttx^2\ttt-\tty^2\ttz, \quad
\tth_6:=\ttx\ttf, \quad
\tth_7:=\tty\ttf, \quad \tth_8:=\tty^2\ttw-\ttx\ttz\ttt, \quad
\tth_9=\ttz\ttf, \\
&\tth_{10}:=\ttx\ttt^2-\tty\ttz\ttw, \quad
\tth_{11}:=\ttt\ttf, \quad \tth_{12}:=\ttw\ttf, \quad \tth_{13}:=\tty\ttw^2-\ttz\ttt^2, 
\quad \tth_{15}:=\ttt^3-\ttz\ttw^2, \\
&\ttg_{10}:=(u-1)\ttx^2\tty\ttz\ttw-3\ttx\tty^2\ttz\ttt+(2-u)\ttx\tty\ttz^3+\tty^4\ttw+
\ttx^3\ttt^2, \\
&\ttg'_{15}:=(u-1)\ttx^2\ttt^3+(u-1)\tty^3\ttw^2-(u+4)\tty^2\ttz\ttt^2+(3u+2)
\ttx\tty\ttz\ttt\ttw+4(1-u)\tty\ttz^3\ttt,\\
&\ttg_{20}:=(u-1)\ttx\ttz\ttt\ttw^2-3\tty\ttz\ttt^2\ttw+(2-u)\ttz^3\ttt\ttw+\ttx\ttt^4+
\tty^2\ttw^3, 
\end{aligned}
\]
and $\ttg_{j+6}:=\ttf\tth_{j}$ for $j=3,5,\dots,13,15$, 
as in \cite[(2.15), (2.17), (2.18)]{CS}. 
By \cite[(2.20)]{CS}, the linear span of $\ttg_9,\dots,\ttg_{21}, \ttg'_{15}$ in 
$H^0\left(Q\left[u\right], \sO_{Q\left[u\right]}(5)\right)$ gives the birational map 
\[
\mu_0:=\sigma_{\sC_2\left[u\right]}\circ\tilde{\chi}\circ\sigma_\Lambda^{-1}
\colon Q\left[u\right]\dashrightarrow X=X^{\KP}\left[u\right].
\] 
Therefore, the elements $\ttg_9,\dots,\ttg_{21}, \ttg'_{15}$ can naturally be seen as 
a basis of $H^0\left(X, -K_X\right)$. Note that there is a natural $G$-action 
to $H^0\left(X, -K_X\right)$. 

\noindent\underline{\textbf{Step 3}}\\
Let $\mathbf{W}_i$ ($i=1,2$) be the $5$-dimensional subspace 
$H^0\left(X, -K_X\right)$ corresponds 
to the rational map $\mu_i\colon X\dashrightarrow Q_i$ in Step 1. 
As in Remark \ref{remark:IKTT} or \cite[Proposition 5.1]{KP18}, both $Z_1$ and $Z_2$
intersect with $\sC_2\left[u\right]$. 
Together with the observation in Step 1, any element in the linear system 
$|\mathbf{W}_i|$ must contains the conic $\sC_2\left[u\right]$. 
Therefore, by \cite[Remark 2.21 and Lemma 2.22]{CS}, 
$\mathbf{W}_1$ and $\mathbf{W}_2$ must be contained in the $9$-dimensional 
$G$-invariant subspace 
\[
\mathbf{W}_0:=\langle \ttg_{11},\ttg_{12}, \ttg_{13}, \ttg_{14}, \ttg_{15}, \ttg_{16}, \ttg_{17}, 
\ttg_{18}, \ttg_{19}
\rangle\subset H^0\left(X, -K_X\right). 
\]
Both $\mathbf{W}_1$ and $\mathbf{W}_2$ are $\G_m$-invariant subspaces of 
$\mathbf{W}_0$, and $\iota^*\mathbf{W}_1=\mathbf{W}_2$ holds. 
Moreover, by Lemma \ref{lemma:pic} \eqref{lemma:pic1}, the subspace 
$\mathbf{W}_1\cap \mathbf{W}_2$ 
is $1$-dimensional. (Indeed, $E\subset X_0$ is the unique element in $H^0\left(
X_0,\sigma^*(-K_X)-2F_1-2F_2\right)$.)
Furthermore, as in Proposition \ref{proposition:book}, 
the weights of semi-invariant bases of $\mathbf{W}_i$ are arranged in equally spaced 
intervals. 
Therefore, after twisting by $\iota$ if necessary, we may assume that 
\[
\mathbf{W}_1=\langle\ttg_{15}, \ttg_{16}, \ttg_{17}, 
\ttg_{18}, \ttg_{19}\rangle\subset H^0\left(X, -K_X\right). 
\]

\noindent\underline{\textbf{Step 4}}\\
Therefore, we may assume that the composition 
$\mu_1\circ\mu_0\colon Q\left[u\right]\dashrightarrow Q_1$ is given by 
\[
\mu_1\circ\mu_0\colon\left[\ttx:\tty:\ttz:\ttt:\ttw\right]\mapsto
\left[\frac{1-u}{u}\tth_{13}:\tth_{12}:\tth_{11}:\tth_{10}:\tth_{9}\right]\in\pr^4_{x_0\dots x_4}.
\]
Let us consider the image of $\mu_1\circ\mu_0$. 
A general point of $Q\left[u\right]$ can be written as 
\[
\left[\ttx:u(\ttx\ttw-\ttz^2)+\ttz^2:\ttz:1:\ttw\right]\in\pr^4 
\]
with $(\ttx,\ttz,\ttw)\in\A^3$. 
The image of the point by $\mu_1\circ\mu_0$ is 
\[
\left[\ttx\ttw^3+\frac{1-u}{u}\ttz^2\ttw^2-\frac{1}{u}\ttz:(\ttx\ttw-\ttz^2)\ttw: 
\ttx\ttw-\ttz^2:\frac{1}{1-u}(\ttx-u\ttx\ttz\ttw^2-(1-u)\ttz^3\ttw)
:(\ttx\ttw-\ttz^2)\ttz\right]
\]
and thus the defining equation of $Q_1$ is 
\[
x_1x_3-x_2^2+\frac{u}{u-1}(x_2^2-x_0x_4)=0.
\]
Let us consider the image $S_i^X$ of the prime divisor $S_i\subset X_0$ in $X$ 
for $i=1, 2$ (in the sense of \S \ref{section:go}). 
Again by \cite[Remark 2.21 and Lemma 2.22]{CS}, the defining equation of $S_1^X$ 
(and also $S_2^X$) is given by either $\ttg_9=0$ or $\ttg_{21}=0$. 
As we have seen in \S \ref{section:go}, the image in $Q_1$ 
of the prime divisor $S_i^X$ is 
either contracted onto a twisted quartic curve or an element in $|\sO_{Q_1}(3)|$. 
Consider the surface in $Q\left[u\right]$ defined by the equation $\tth_{15}=0$. 
The surface is a rational surface and a general point of the surface can be written as 
\[
\left[\ttx:u\ttx\ttw+\frac{1-u}{\ttw^4}:\frac{1}{\ttw^2}:1:\ttw\right]\in\pr^4 
\]
with $\ttx,\ttw\in\Bbbk^\times$. The image of the point by $\mu_1\circ\mu_0$ is 
\[
\left[\ttw^4(\ttx\ttw^5-1):\ttw^3(\ttx\ttw^5-1):\ttw^2(\ttx\ttw^5-1):
\ttw(\ttx\ttw^5-1):\ttx\ttw^5-1\right].
\]
Therefore, the closure of the dominant image of the surface is the twisted quartic curve 
defined by the image of 
\begin{eqnarray*}
\pr^1_{t_0t_1} &\to&\pr^4_{x_0\dots x_4} \\
\left[t_0:t_1\right]&\mapsto&
\left[t_0^4:t_0^3t_1:t_0^2t_1^2:t_0t_1^3:t_1^4\right]. 
\end{eqnarray*}
In particular, the defining equation of $S_1^X$ is equal to $\ttg_{21}=0$ (and the line 
$Z_1$ is equal to the line $\ell_1\subset X^{\KP}\left[u\right]$ in the sense of 
\cite[\S 2]{CS}). 
Thus we have $q=\frac{u}{u-1}$. The equality $u=\frac{1}{v}$ immediately follows from 
Proposition \ref{proposition:parameter}. 
\end{proof}

\end{document}